\documentclass[11pt,a4paper]{amsart}
\title{Reflection groups on smooth projective threefolds}
\author{Wendelin Lutz}

\usepackage{amsfonts, amssymb}
\usepackage{tikz}
\usepackage{float}
\usepackage{xcolor}

\usepackage{tikz-cd}
\usepackage[margin=2.5cm]{geometry}
\usepackage{booktabs}
\usepackage{hyperref}
\newcommand{\ob}{\operatorname{ob}}
\usepackage{xfrac}
\usepackage{epsfig}
\usepackage{setspace}
\usepackage{enumerate}
\usepackage{mathrsfs}
\newcommand{\Adm}{\mathrm{Adm}}
\newcommand{\Amp}{\mathrm{Amp}}
\newcommand{\Aut}{\mathrm{Aut}}
\newcommand{\Bbig}{\mathrm{Big}}

\newcommand{\Def}{\mathrm{Def}}
\newcommand{\DR}{\mathrm{DR}}

\newcommand{\Eff}{\mathrm{Eff}}

\newcommand{\GL}{\mathrm{GL}}
\newcommand{\Gr}{\mathrm{Gr}}

\newcommand{\Hilb}{\mathrm{Hilb}}
\newcommand{\Hom}{\mathrm{Hom}}
\newcommand{\id}{\mathrm{id}}

\newcommand{\Mov}{\mathrm{Mov}}
\newcommand{\NE}{\mathrm{NE}}
\newcommand{\Nef}{\mathrm{Nef}}

\newcommand{\Pic}{\mathrm{Pic}}

\newcommand{\relint}{\mathrm{relint}}

\newcommand{\cA}{\mathcal{A}}
\newcommand{\cC}{\mathcal{C}}
\newcommand{\cD}{\mathcal{D}}
\newcommand{\cE}{\mathcal{E}}
\newcommand{\cF}{\mathcal{F}}
\newcommand{\cH}{\mathcal{H}}

\newcommand{\cL}{\mathcal{L}}
\newcommand{\cM}{\mathcal{M}}
\newcommand{\cO}{\mathcal{O}}

\newcommand{\cX}{\mathcal{X}}
\newcommand{\cY}{\mathcal{Y}}
\newcommand{\cZ}{\mathcal{Z}}
\newcommand{\AAA}{\mathbb{A}}
\newcommand{\CC}{\mathbb{C}}
\newcommand{\DD}{\mathbb{D}}
\newcommand{\FF}{\mathbb{F}}

\newcommand{\PP}{\mathbb{P}}
\newcommand{\QQ}{\mathbb{Q}}
\newcommand{\RR}{\mathbb{R}}

\newcommand{\ZZ}{\mathbb{Z}}

\newcommand{\ip}{\raise1pt\hbox{\large $\lrcorner$}}

\theoremstyle{plain}
\newtheorem{theorem}{Theorem}[section]
\newtheorem{proposition}[theorem]{Proposition}
\newtheorem{lemma}[theorem]{Lemma}
\newtheorem*{lem*}{Lemma}

\newtheorem{corollary}[theorem]{Corollary}

\theoremstyle{definition}
\newtheorem{definition}[theorem]{Definition}

\newtheorem*{remark*}{Remark}

\begin{document}
\begin{abstract}
	Let $Y$ be a smooth projective threefold with $H^2(Y, \cO_Y)=0$ and unobstructed deformations. We define reflection groups $W_Y^{sm}$ and $W_Y^{big}$ canonically associated to $Y$, acting on the nef, resp. movable cone of a very general deformation of $Y$. If $-K_Y$ is nef, we show that $\Nef(Y)\cap\Bbig(Y)$ (resp. $\overline{\Mov}(Y)\cap\Bbig(Y)$) is a fundamental domain for the action of the corresponding group on the big part of the generic cone. This generalizes work of \cite{WilsonCone, Lutz} on Calabi--Yau threefolds.
	The group $W_Y^{big}$ recovers and gives a deformation-theoretic explanation for the Weyl group action on divisor classes of point blowups of $\mathbb P^3$ studied by Dolgachev--Ortland~\cite{Dolgachev-Ortland} and Mukai~\cite{Mukai}.
\end{abstract}
\maketitle
	\section{Introduction and overview of results}
	The nef and the movable cones of a projective variety $Y$ are fundamental invariants of $Y$. The nef cone is dual to the cone of curves and closely related to contractions of $Y$ through Mori theory. The movable cone further encodes information about contractions of small birational modifications of $Y$.
It is therefore natural to ask how the cones $\Nef(Y)$ and $\overline{\Mov}(Y)$ change under deformation.
We assume throughout that $H^2(Y, \cO_Y)=0$, and $\Def(Y)$ is smooth. The assumption $H^2(Y, \cO_Y)=0$ implies that parallel transport identifies the spaces of numerical divisor classes of the fibers in a sufficiently small family. We may then regard the nef (resp. movable) cones of $Y$ and a very general deformation $Y^{\mathrm{gen}}$ 
as cones in the same vector space $N^1(Y, \RR)$.
The assumption that deformations of $Y$ are unobstructed implies in particular that we may unambiguously speak of a very general deformation of $Y$. 

The behaviour of $\Nef(Y)$ under deformation was studied by Wilson~\cite{WilsonCone} for Calabi--Yau threefolds, by Minagawa~\cite{Minagawa} and Paoletti~\cite{Paoletti} for weak Fano threefolds, and by Wi\'{s}niewski~\cite{Wisniewski} for Fano manifolds of arbitrary dimension. The behaviour of $\overline{\Mov}(Y)$ for Calabi--Yau threefolds was studied by Lutz~\cite{Lutz}. 

Our starting point is that, in many cases, a disappearing facet of $\Nef(Y)$ or $\overline{\Mov}(Y)$ induces a symmetry of the generic cone, given by reflection in the supporting hyperplane of the facet.
To describe the simplest example, suppose that $E \subset Y$ is a surface ruled over a curve $C$, and $K_Y \cdot \ell_E=0$, where $\ell_E \in N_1(Y)$ is the class of a fiber of the ruling. 
We define  
\[
\sigma_E \in \Aut(N^1(Y, \ZZ)), \qquad \sigma_E(x) = x+(x \cdot \ell_E)[E].
\] 
The adjunction formula shows that $E \cdot \ell_E=-2$, so $\sigma_E$ preserves the hyperplane $\ell_E^\perp$ pointwise and sends $[E]$ to $-[E]$. We call $\sigma_E$ the reflection associated to $E$. The pairing in the definition of $\sigma_E$ is the duality pairing between $N^1(Y)$ and $N_1(Y)$ and does not use any bilinear form on $N^1(Y)$.  

Suppose now that $E$ is the exceptional divisor of a $K$-trivial contraction $p \colon Y \rightarrow \bar{Y}$ contracting $E$ along its ruling, so that $\ell_E^\perp$ supports a facet of $\Nef(Y)$ and $\overline{\Mov}(Y)$. Under our hypotheses, the contraction extends to a sufficiently small deformation of $Y$. If the deformed contraction is an isomorphism on a very general fibre, then the nef facet disappears, and $\sigma_E$ preserves the generic nef cone. If the exceptional divisor disappears but exceptional curves remain, the movable facet disappears, and $\sigma_E$ preserves the generic movable cone.
These two behaviours lead to reflection groups
\[
W_Y^{sm}\subset W_Y^{big},
\]
preserving the generic nef and movable cones, respectively. The full construction allows quasi-ruled surfaces for the small group and conic fibrations on crepant small modifications of $Y$ for the big group, see Definition~\ref{def:deformation-root-classes}. 

For Calabi--Yau threefolds, these reflection groups and their fundamental-domain properties were studied by Wilson and Lutz. We construct the reflection groups and their actions without any global positivity assumption on $-K_Y$. When $-K_Y$ is nef, we further show that the nef and movable cones of $Y$, intersected with the big cone, are fundamental domains for the corresponding actions.
\begin{theorem}[Reflection actions and fundamental domains]
	\label{thm:withoutnef}\label{thm:withnef}
	Let $Y$ be a smooth projective threefold such that
	$H^2(Y,\cO_Y)=0$ and $\Def(Y)$ is smooth, and let $Y^{\mathrm{gen}}$ be a very
	general deformation of $Y$.
	\begin{enumerate}
		\item 
		The group $W_Y^{sm}$ acts faithfully on $\Nef(Y^{\mathrm{gen}})$, and
		$W_Y^{big}$ acts faithfully on $\overline{\Mov}(Y^{\mathrm{gen}})$.
		\item 
		If $-K_Y$ is nef, then
		\[
		\Nef(Y)\cap\Bbig(Y)
		\quad\text{and}\quad
		\overline{\Mov}(Y)\cap\Bbig(Y)
		\]
		are fundamental domains for the actions of $W_Y^{sm}$ and $W_Y^{big}$ on
		\[
		\Nef(Y^{\mathrm{gen}})\cap\Bbig(Y^{\mathrm{gen}})
		\quad\text{and}\quad
		\overline{\Mov}(Y^{\mathrm{gen}})\cap\Bbig(Y^{\mathrm{gen}}),
		\]
		respectively. In particular, $\Nef(Y)=\Nef(Y^{\mathrm{gen}})$
		if and only if $W_Y^{\mathrm{sm}}=\{1\}$, and
		$\overline{\Mov}(Y)=\overline{\Mov}(Y^{\mathrm{gen}})$ if and only if
		$W_Y^{\mathrm{big}}=\{1\}$.
	\end{enumerate}
\end{theorem}
We find that the group $W_Y^{big}$ gives a deformation-theoretic origin for the appearance of Weyl groups in the literature on point-blowups of $\PP^3$.
Indeed, for each integer $n \geq 4$, we construct a variety $Y_e$ which is the blowup of $\PP^3$ in a set of $n$ points in a carefully chosen infinitely near position, and compute that 
$W(T_{2,4,n-4})\subset W_{Y_e}^{big}$, where the Dynkin diagram
$T_{2,4,n-4}$ is shown in Figure~\ref{fig:dynkin-t24}.
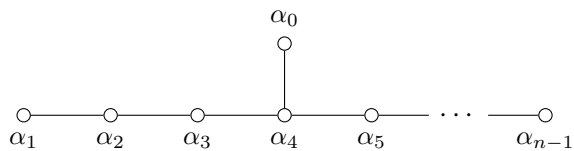
\begin{figure}[H]
\centering
\begin{tikzpicture}[x=1.15cm,y=0.95cm,
    root/.style={circle,draw,fill=white,inner sep=0pt,minimum size=5pt},
    every label/.style={font=\small}]
  \node[root,label=below:{$\alpha_1$}] (a1) at (0,0) {};
  \node[root,label=below:{$\alpha_2$}] (a2) at (1,0) {};
  \node[root,label=below:{$\alpha_3$}] (a3) at (2,0) {};
  \node[root,label=below:{$\alpha_4$}] (a4) at (3,0) {};
  \node[root,label=above:{$\alpha_0$}] (a0) at (3,1) {};
  \node[root,label=below:{$\alpha_5$}] (a5) at (4,0) {};
  \node (dots) at (5,0) {$\cdots$};
  \node[root,label=below:{$\alpha_{n-1}$}] (an) at (6,0) {};
  \draw (a1)--(a2)--(a3)--(a4)--(a5);
  \draw (a4)--(a0);
  \draw (a5)--(dots);
  \draw (dots)--(an);
\end{tikzpicture}
\caption{The diagram $T_{2,4,n-4}$. If $n=4$, by convention, the diagram only contains the nodes $\alpha_0$ through $\alpha_3$.}
\label{fig:dynkin-t24}\label{fig}
\end{figure}
A very general deformation of $Y_e$ is of course the blowup of $\PP^3$ in $n$ points in general position, and $\Def(Y_e)$ is smooth. Theorem~\ref{thm:withoutnef} thus recovers the Weyl group action defined by Dolgachev--Ortland~\cite{Dolgachev-Ortland} and Mukai~\cite{Mukai}, using the auxiliary form 
\begin{equation}\label{eq:Mukai}
D_1 \star D_2=D_1 \cdot D_2 \cdot -\tfrac{1}{2}K_Y.
\end{equation}
If $n = 8$, $-K_{Y_e}$ is nef. We prove that
\[
W_{Y_e}^{big}=W(T_{2,4,4})
\]
and that $\overline{\Mov}(Y_e)$ is rational polyhedral. Theorem~\ref{thm:withoutnef} therefore recovers a theorem of Stenger--Xie.
\begin{theorem}(Theorem~\ref{thm:Stenger--Xie--proof}) \label{thm:stenger-xie-eight-point-blowup}\cite[Theorem 1.2]{StengerXie}
	Let $Y^{\mathrm{gen}}$ be the blowup of $\PP^3$ in $ 8$ points in very general position. The group $W(\tilde{E}_7)=W(T_{2,4,4})$ acts on ${\Mov}^+(Y^{\mathrm{gen}})$ with a rational polyhedral fundamental domain.
\end{theorem}
Here $C^+=\mathrm{Conv}(\bar{C} \cap V_\QQ)$ for $C$ a convex cone in a vector space $V$ with rational structure. Our proof gives a geometric meaning to the rational polyhedral fundamental domain: it is the movable cone of a special degeneration. 

\subsection{Deformation Roots}
Let $p \colon Y \rightarrow \bar{Y}$ be a divisorial contraction contracting a surface $E$ to a smooth curve $C$, so that $q=p|_E \colon E \rightarrow C$ is a conic bundle. Suppose also that the contraction is $K$-trivial, i.e. that $K_Y \cdot \ell_E=0$ where $\ell_E$ is the class of a general fiber of $q$.
The Leray spectral sequence gives an isomorphism 
\[
H^1(E, N_{E/Y}) \cong H^0(C, M_E)
\]
where
\begin{equation}\label{eq:obstruction-bundle}
M_E:=\omega_{\bar{Y}}^{-1}|_C\otimes\omega_C\in\Pic(C).
\end{equation}
We prove in Theorem~\ref{thm:surjective} that the obstruction map
\[
\ob_E \colon H^1(Y,T_Y)\longrightarrow H^0(C,M_E)
\]
is surjective. Therefore, if $\Def(Y)$ is smooth and $H^0(C,M_E)\neq0$, the Kodaira--Spencer class $\kappa$ of a general one-parameter deformation $\cY\to\Delta$ has nonzero image $\sigma\in H^0(C,M_E)$, so that $E$ does not deform with $Y$ to first order, although some exceptional fibers of $E$ may still deform. 
We will show that this implies that the facet of $\overline{\Mov}(Y)$ supported on $\ell_E^\perp$ disappears.
In view of this, we call $M_E$ the obstruction bundle of $E$. 

The obstruction to deforming the fiber $F_c$ over $c \in C$ with $Y$ factors through the evaluation map 
\[
H^0(C, M_E) \rightarrow H^0(c, M_E|_c) \cong \CC.
\]
Therefore, if $M_E$ is trivial, the nonzero section $\sigma$ is nowhere vanishing, so no fiber of $E$ deforms to first order. If $E$ is also quasi-ruled, we show that the facet of $\Nef(Y)$ supported on $\ell_E^\perp$ disappears.

For the movable cone, one also needs to consider $K$-trivial conic bundles on small $\QQ$-factorial modifications (SQMs) of $Y$. This leads to the following definition.
\begin{definition}
		Let $E \subset Y$ be an irreducible surface. 
		\begin{itemize}
\item A \emph{\(K\)-trivial conic model} of \(E\subset Y\) consists of a crepant SQM \(\alpha\colon Y\dashrightarrow Z\) together with a conic fibration~\footnote{See Definition~\ref{def:conic-bundle} for our definition of a conic fibration, which is more general than the one sometimes considered in the literature.} \(q\colon E_Z\to C\) such that \(K_Z\equiv_q0\), where $E_Z$ denotes the strict transform of $E$. 
\item
A $K$-trivial conic model is \emph{contractible} if there exists a primitive birational contraction $p \colon Z \rightarrow \bar{Z}$ contracting $q \colon E_Z \rightarrow C$ along its fibers.
\end{itemize}
\end{definition}
With this terminology in place, we can now define our deformation roots.
\begin{definition}[Deformation root classes]\label{def:deformation-root-classes}
	Let $Y$ be a smooth projective threefold with $H^2(Y, \cO_Y)=0$ and $\Def(Y)$ smooth.
	\begin{itemize}
		\item Define $\Delta_Y^{\mathrm{big}}$ to be the set of classes $[E]\in N^1(Y,\ZZ)$, where $E \subset Y$ is an irreducible surface for which there is a \(K\)-trivial conic model satisfying 
	\[
	H^0(E_Z, \omega_Z^{-1}|_{E_Z}\otimes q^*\omega_C) \neq 0 \qquad \text{(non-vanishing condition)}
	\]
	\item Define $\Delta_Y^{\mathrm{sm}}\subset\Delta_Y^{\mathrm{big}}$ to be the subset consisting
	of classes represented on $Y$ itself by a quasi-ruled surface
	$q\colon E\to C$ satisfying $K_Y \cdot \ell_E=0$ and 
	\[
	\omega_Y^{-1}|_E\otimes q^*\omega_C\cong\cO_E.
	\]
	\end{itemize}
\end{definition} 
We will show in Proposition~\ref{prop:existenceofcontractionofE} that every $E \in \Delta_Y^{\mathrm{big}}$ admits a \emph{contractible} $K$-trivial conic model for $E$ satisfying the non-vanishing condition. Moreover, we show that the \emph{fiber class} $\ell_E:=\alpha_*^{-1}\ell_{E_Z}$ is independent of the choice of such model. This allows us to define the reflection $\sigma_E$ for every $E \in \Delta_Y^{\mathrm{big}}$. 
On any such contractible model, we have $K_Z \cdot \ell_{E_Z}=0$ and 
\[
\omega_Z^{-1}|_{E_Z}\otimes q^*\omega_C \cong q^*M_{E_Z},
\]
so by the above discussion, the facet of $\overline{\Mov}(Z)$ supported on $\ell_{E_Z}^\perp$ disappears. Hence the facet of $\overline{\Mov}(Y)$ supported on $\ell_E^{\perp}$ disappears.

As in \cite{WilsonCone, Lutz}, the reflection $\sigma_E$ admits a geometric realization by flops in deformation families. Suppose first that $E \in \Delta_Y^{\mathrm{big}}$ is the exceptional divisor of a primitive $K$-trivial type III contraction $p \colon Y \rightarrow \bar{Y}$. Over a sufficiently small miniversal deformation $\cY \rightarrow S$ of $Y$, the contraction extends to a small contraction $\cY \rightarrow \bar{\cY}$, whose flop $\cY^+ \rightarrow \bar{\cY}$ exists. The induced birational map on central fibers extends to an isomorphism $Y \cong Y^+$. Using this identification, the composite
\[
N^1(Y) \cong N^1(\cY/S)
\cong N^1(\cY^+/S)\cong N^1(Y),
\]
where the middle isomorphism is given by strict transform, is precisely $\sigma_E$ (Proposition~\ref{prop:flopofroot}). For a general $E \in \Delta_Y^{\mathrm{big}}$, we first pass by a sequence of flops to a contractible conic model and apply the same construction there; conjugating by the resulting identification of divisor spaces gives $\sigma_E$.

In the case that $-K_Y$ is nef, we find that elements $E \in \Delta_Y^{\mathrm{big}}$ capture all disappearing facets of the movable cone:
\begin{proposition}(see Propositions~\ref{prop:injection-roots-faces} and~\ref{prop:bijection-roots-faces}) \label{prop:roots-vs-disappears} Suppose that $-K_Y$ is nef. 
	The assignment sending $E$ to the facet supported on $\ell_E^\perp$
	defines a bijection between $\Delta_Y^{\mathrm{big}}$ (resp. $\Delta_Y^{\mathrm{sm}}$) and the set of facets of $\overline{\Mov}(Y)$ (resp. $\Nef(Y)$) which disappear under deformation. 
\end{proposition}
It is certainly possible for an irreducible divisor $E \subset Y$ to correspond to several distinct faces of the movable cone, corresponding to different contractions on different birational models of $Y$.  Proposition~\ref{prop:roots-vs-disappears} shows that every $E \in \Delta_Y^{\mathrm{big}}$ corresponds to a unique disappearing $K$-trivial facet. 
Note that $K$-negative contractions always deform by \cite{MoriThreefolds}.

We investigate the meaning of Definition~\ref{def:deformation-root-classes} in a few special cases. 
If $Y$ is Calabi--Yau, $\omega_Y$ is trivial, so this recovers the definition in \cite{WilsonCone, Lutz}:
\begin{itemize}
	\item The condition $H^0(E,\omega_Y^{-1}|_{E}\otimes q^*\omega_C) \neq 0$ is equivalent to requiring that $g(C) \geq 1$.
	\item The condition $\omega_Y^{-1}|_{E}\otimes q^*\omega_C \cong \cO_E$ is equivalent to $g(C)=1$.
\end{itemize}
It is also instructive to consider the case that $-K_Y$ is base-point free. For a $K$-trivial quasi-ruled surface $q \colon E \rightarrow C$, the condition that $E \in \Delta_Y^{\mathrm{sm}}$ is equivalent to either
\begin{enumerate}
	\item $g(C)=0$ and $\omega_Y^{-1}|_E=q^*\cO_{\PP^1}(2)$.
	\item $g(C)=1$ and $\omega_Y^{-1}|_E=\cO_E$.
\end{enumerate}
We can view (1) as a log Calabi--Yau analogue of (2). Indeed, choose a very general (smooth) element $D \in |-K_Y|$. Then $(Y, D)$ is a log Calabi--Yau pair, and the open complement $U=Y \setminus D$ has trivial canonical bundle.

If $g(C)=1$ and $\omega_Y^{-1}|_E$ is trivial, then $E \subset U$, and $E$ is quasi-ruled over the Calabi--Yau curve $C$.

If $g(C)=0$ and $\omega_Y^{-1}|_E=q^*\cO_{\PP^1}(2)$, then $D$ meets $E$ in two fibers, so that $E \cap U$ is ruled over the log Calabi--Yau curve $\PP^1 \setminus \{0, \infty\}$.
\subsection{Rational elliptic threefolds and discriminant quartics}
Let $N$ be a net of quadrics in $\PP^3$ satisfying a nondegeneracy assumption of Prendergast--Smith~\cite[Assumption 1]{PrendergastSmithextremal}: there exist $Q,Q'\in N$ such that $Q\cap Q'$ is smooth along the base
locus of $N$. Let $Y \rightarrow \PP^3$ be the iterated blowup of the (possibly infinitely near) eight base points. The induced map $\pi \colon Y \rightarrow \PP^2$ is an elliptic fibration. 
Prendergast--Smith~\cite{PrendergastSmithextremal} defined an affine root lattice $\tilde{R}$, associated to $Y$, generated by irreducible components of \emph{reducible} divisors of the form $\pi^{-1}(B)$ for $B \subset \PP^2$ a curve, and with respect to the form \eqref{eq:Mukai}. This is an analogue of $(-2)$-curves on rational elliptic surfaces, which by the Kodaira classification are exactly the irreducible components of reducible fibers.

We find that the generators of $\tilde{R}$ biject with $\Delta_Y^{big}$, and in particular, $W(\tilde{R})=W_Y^{big}$. 
We also give an answer to a question raised in \cite[Section 6]{PrendergastSmithextremal}:

The plane quartic 
\[
\Delta=\left\{\det(\lambda_1Q_1+\lambda_2Q_2+\lambda_3Q_3)=0\right\}
\subset(\mathbb P^2)^*,
\]
in the dual projective space parametrizes singular quadrics in $N$ and has at worst ADE singularities. In the case that the elliptic fibration is extremal, Prendergast--Smith observed that the ADE type of the singularities of $\Delta$ agrees with the finite root system $R$ of ADE type associated to $\tilde{R}$. 
We prove this assertion for every net of quadrics satisfying the nondegeneracy assumption.
\begin{theorem}(see Corollary~\ref{cor:PS})
	Let $N$ be a net of quadrics satisfying the nondegeneracy assumption, let $R$ be the associated finite root system. Then the ADE type of $R$ agrees with the ADE type of the discriminant quartic $\Delta$. 
\end{theorem}
\subsection{Proof overview}
Two ingredients allow us to remove global positivity assumptions on $-K_Y$: the surjectivity of the obstruction map $\ob_E$ for $E$ the exceptional divisor of a \(K\)-trivial contraction (Theorem~\ref{thm:surjective}), and the existence of contractible conic models for deformation roots (Proposition~\ref{prop:existenceofcontractionofE}). We briefly sketch the ideas going into their proofs.

Via the perfect pairing 
$\Omega^1_Y \otimes \Omega^2_Y \rightarrow \omega_Y$ and the adjunction formula, the obstruction map is identified with 
\[
\ob_E \colon H^1(\Omega^2_Y \otimes \omega_Y^{-1}) \rightarrow H^1(\omega_E \otimes \omega_Y^{-1}|_E),
\]
induced by restriction of forms. 
If $Y$ is Calabi--Yau, $\ob_E$ is a graded piece of the Hodge filtration for the natural restriction map 
\[
i^* \colon H^3(Y, \CC) \rightarrow H^3(E, \CC)
\]
and by strictness of the Hodge filtration, it is enough to show that $i^*$ is surjective~\cite{WilsonCone}.
In the weak Fano case, Minagawa chooses a smooth divisor $D \in |-2K_Y|$ and deduces surjectivity of $\ob_E$ from the known surjectivity result on the Calabi--Yau double cover $\tilde{Y}$ of $Y$ branched along $D$. 

In the absence of any hypothesis on $-K_Y$, we use Saito's theory of Hodge modules to prove surjectivity of $\ob_E$. The necessary background material for the argument is collected in Appendix~\ref{app:hodge}. The idea is as follows. Let $p \colon Y \rightarrow \bar{Y}$ denote the contraction of $E$ to a curve $C$, so we have a diagram
\[
\begin{tikzcd}
	E \arrow[r, hook, "i"] \arrow[d, "q"'] &
	Y \arrow[d, "p"] \\
	C \arrow[r, hook, "j"'] &
	\bar{Y}
\end{tikzcd}
\]
The morphism of sheaves $\QQ_Y \rightarrow i_*\QQ_E$ and $\QQ_C \rightarrow Rq_*\QQ_E$, together with the fundamental class and Verdier duality, give rise to a morphism 
\[
 \Phi \colon Rp_*\QQ_Y[3] \rightarrow j_*\QQ_C[1]
\]
in the derived category of constructible sheaves. 
Since the morphism $p$ is semismall (see Appendix), the decomposition theorem shows that $Rp_*\QQ_Y[3]$ is a semisimple perverse sheaf. $\Phi$ is easily seen to be nonzero, so, since  $j_*\QQ_C[1]$ is simple, $\Phi$ is split surjective. 

Using Saito's theory of mixed Hodge modules, the morphism $\Phi$ can be constructed in the category of pure Hodge modules (with appropriate Tate twists inserted), and $\Phi$ continues to be split surjective as a morphism of pure Hodge modules. Applying the graded de Rham functor, we obtain a split surjective morphism 
\[
Rp_*T_Y \rightarrow j_*Rq_*N_{E/Y}
\]
in $D^b_{coh}(\bar{Y})$. The induced map $\ob_E$ on hypercohomology is therefore surjective as well.

The existence of contractible $K$-trivial conic models for an arbitrary $E \in \Delta_Y^{\mathrm{big}}$ (Proposition~\ref{prop:existenceofcontractionofE}) is easy to establish under the assumption that $-K_Y$ is nef, the difficulty in the general case is to prove the existence of an extremal ray $R$ of the cone of curves satisfying 
\[
K_Y \cdot R=0, \qquad E \cdot R<0.
\]
Spicer's cone theorem for foliations on surfaces allows us to find such a ray and construct a contractible model by a sequence of flops.

\subsection{Outline}
Section~\ref{sec:background} gives the necessary background on birational geometry and deformation theory used in this paper. 

Section~\ref{sec:contractions} studies in detail the geometry of $K$-trivial type III contractions, that is, divisorial contractions $Y \rightarrow \bar{Y}$ contracting a divisor $E$ to a curve $C$. 
 
Section 4 shows that the defining conditions of $E \in \Delta_Y^{\mathrm{big}}$, resp. $E \in \Delta_Y^{\mathrm{sm}}$ are equivalent to the corresponding facet of the movable, resp. nef cone disappearing under deformation. 

In Section 5, we prove the existence of a contractible $K$-trivial conic model for $E \in \Delta_Y^{\mathrm{big}}$ and independence of the fiber class $\ell_E$ among such models.

In Section 6 we show that the sets $\Delta_Y^{\mathrm{sm}}$ and $\Delta_Y^{\mathrm{big}}$ form sets of generalized roots, so that classical results on reflection groups can be used.
In particular this gives the first half of Theorem~\ref{thm:withoutnef}. 

Section 7 is then devoted to the second half of Theorem~\ref{thm:withoutnef} using the assumption that $-K_Y$ is nef. 

Section~\ref{sec:examples} is devoted to the applications mentioned in the introduction.
 
\subsection{Acknowledgements}
I would like to thank Paul Hacking, Isabel Stenger, and Artie Prendergast-Smith for helpful discussions. I am especially grateful to Zhixin Xie for pointing out to me that many of the results of the MMP continue to hold for generalized pairs. 
\subsection{AI Usage}
This project started before the advent of AI tools and the main ideas of this paper were developed without the use of AI tools. However, ChatGPT was used to strengthen the main results of this paper. 

An earlier draft assumed throughout that $-K_Y$ is semiample. GPT suggested that Saito's theory of Hodge modules could be used to prove the surjectivity result of Theorem~\ref{thm:surjective} without any global positivity assumption on $-K_Y$, and the final proof given in Theorem~\ref{thm:surjective} was developed in conversations with GPT. A previous version of the argument assumed that $-K_Y$ is semiample and then used the logarithmic complex $\Omega_Y^\bullet(\log(D))$ for $D \in |-mK_Y|$ smooth. 
GPT also pointed out that Proposition~\ref{prop:existenceofcontractionofE} could be proved without any assumptions on $-K_Y$ by using Spicer's cone theorem for foliations. 

Aside from the above, GPT was used for expositional suggestions, copyediting and minor simplifications of some of the proofs. 
\section{Background}\label{sec:background}
Throughout this section, $Y$ will denote a smooth projective variety over the field of complex numbers. 
\subsection{Background on cones}
Given any cone $C$, we denote by $\bar{C}$ its closure, and $C^{\circ}$ its interior. 
The Mori cone $\NE(Y)$ is the convex cone in $N_1(Y, \RR)$ spanned by the set of effective curves on $Y$. 
The nef cone $\Nef(Y) \subset N^1(Y, \RR)$ is the dual cone of $\overline{\NE}(Y)$, and by the Kleiman criterion, the ample cone $\Amp(Y)$ is identified with the interior of $\Nef(Y)$. 
Recall that an effective divisor $D$ is movable if the stable base locus of $D$ has codimension at least $2$ in $Y$. We denote by $\Mov(Y) \subset N^1(Y, \RR)$ the cone spanned by movable divisors. In general, $\Mov(Y)$ is neither open nor closed, and we denote $\overline{\Mov}(Y)$ its closure. 
$\Eff(Y)$ is the cone spanned by effective divisor classes (which is also not open or closed in general), and its interior is the cone of big divisors $\Bbig(Y)$.
We have inclusions
\begin{align*}
	\Amp(Y)  &\subset \Mov(Y)^{\circ} \subset \Bbig(Y)\\
	{\Nef}(Y) & \subset \overline{\Mov}(Y) \subset \overline{\Eff}(Y)
\end{align*}
Integral points on the boundary of $\Nef(Y)$ (respectively, $\overline{\Mov}(Y)$) are not always represented by effective divisor classes, so we define the effective nef and movable cones
\[
\Nef^e(Y)=\Nef(Y) \cap \Eff(Y), \qquad
{\Mov}^e(Y)=\overline{\Mov}(Y) \cap \Eff(Y).
\]
A small $\QQ$-factorial modification (SQM) of $Y$ is a birational map $\alpha \colon Y \dashrightarrow Z$ with $Z$ projective and $\QQ$-factorial which is an isomorphism in codimension one. An SQM is crepant if each curve in the locus of indeterminacy has degree zero against the canonical bundle. By \cite[Theorem 4.6]{KawamataD}, a crepant SQM of threefolds factors as a composite of flops. 
An SQM $\alpha \colon Y \dashrightarrow Z$ induces an isomorphism $\alpha^* \colon N^1(Z, \ZZ) \rightarrow N^1(Y, \ZZ)$ by strict transform of divisors. The strict transform of a movable divisor is again movable, so that $\alpha^*(\Mov(Z))=\Mov(Y)$. Analogous equalities hold for ${\Mov}^e$, $\overline{\Mov}$, $\Bbig$, $\Eff$, and $\overline{\Eff}$.
On curve classes, $\alpha_*\colon N_1(Y)\to N_1(Z)$ is dual to $\alpha^*$.

A contraction is a surjective projective morphism $p \colon Y \rightarrow \bar{Y}$ to a normal projective variety with connected fibers. It is primitive if the relative Picard rank of $p$ is one. 
A primitive contraction defines an extremal ray $R_p$ of the cone of curves: let $A$ be ample on $\bar{Y}$, and let $H=p^*A$. Then $H^\perp$ meets $\NE(Y)$ in a face of dimension one, spanned by the class of any curve contracted by $p$. 
Dually, a primitive contraction defines a codimension one face (facet) $F$ of $\Nef(Y)$. If the contraction is birational, then $F$ meets $\Bbig(Y)$, so that $F \cap \Bbig(Y)$ is a facet of $\Nef(Y) \cap \Bbig(Y)$. 
A contraction is small if the exceptional locus has codimension at least two, and divisorial if the exceptional locus is a divisor. 
By \cite[Lemma 2.3]{KawamataCrepant}, a birational contraction morphism is divisorial if and only if the corresponding facet of $\Nef(Y)$ is contained in a facet of $\overline{\Mov}(Y)$. 
Indeed, since curves with class in $R_p$ sweep out a divisor, every movable divisor is nonnegative on $R_p$, so that $F \subset \partial \overline{\Mov}(Y)$.
We say that $p$ is $K$-trivial if $K_Y \cdot R_p=0$. 

Suppose that $(Y, \Delta)$ is a klt pair. We say that $p$ is a Mori contraction for $(Y, \Delta)$ if $(K_Y+\Delta) \cdot R_p<0$. 
We have the following result:
\begin{lemma}\label{lem:Every-Contraction-Is-Mori}
	Let $p \colon Y \rightarrow \bar{Y}$ be a primitive $K$-trivial birational contraction. Then there exists an effective $\QQ$-divisor $\Delta$ such that $(Y, \Delta)$ is klt and $p$ is a Mori contraction for $(Y, \Delta)$.
	In particular, a line bundle $L$ descends to $\bar{Y}$ iff $L \cdot R_p=0$.
\end{lemma}
\begin{proof}
	Let $A$ be ample on $\bar{Y}$.
	Since $p$ is a projective morphism, there is a $p$-ample divisor $B$. Since $p$ is birational, $p_*\cO_Y(-B)$ is nonzero, and hence
	\[
	p_*\cO_Y(-B) \otimes A^m
	\]
	is globally generated for $m \gg 0$. By the projection formula, it follows that $-B+mp^*A$ has a section $D$ for some $m$. Since $Y$ is smooth, the pair $(Y, \Delta)$ is klt, where $\Delta=\epsilon D$ for $\epsilon$ small. 
	Since $p$ is $K$-trivial, we have that 
	\[
	(K_Y+\Delta) \cdot R_p=\epsilon(-B+mp^*A) \cdot R_p<0,
	\]
	as required.
	The last claim is \cite[Theorem 3.7(4)]{KollarMori}.
\end{proof}
Note that the exceptional divisor of a primitive divisorial contraction is irreducible by the negativity lemma. In particular, we can take $\Delta=\epsilon E$ in Lemma~\ref{lem:Every-Contraction-Is-Mori} if $p$ is divisorial. 

\subsection{Background on deformation theory}
Throughout this section, $Y$ is a smooth projective variety over $\CC$ satisfying $H^2(Y, \cO_Y)=0$. 

A \emph{deformation} $\pi \colon \cY \rightarrow (S,0)$ of $Y$ is a flat projective morphism over a pointed smooth base $S$, together with an isomorphism of the central fiber with $Y$. 
All constructions are understood locally around the marked point in the
\'etale topology. Thus we may shrink $S$ and replace it by a pointed
\'etale neighbourhood, retaining the same notation for the pulled-back
family. We refer to this convention by saying that $S$ is an
\'etale germ. The analytically inclined reader is welcome to think of the germ of the analytic space $(S,0)$ throughout.

We write $\Def(Y)$ for the miniversal analytic deformation germ of $Y$.
Its tangent space is naturally identified with $H^1(Y,T_Y)$. A deformation
$\pi\colon\cY\to(S,0)$ has a Kodaira--Spencer map
\[
\operatorname{KS}_{\pi,0}\colon T_0S\longrightarrow H^1(Y,T_Y),
\]
which sends a tangent vector to the corresponding first-order deformation
of $Y$. We say that $Y$ has unobstructed deformations if $\Def(Y)$ is
smooth. 

The assumption $H^2(Y,\cO_Y)=0$ allows to lift an ample line bundle on $Y$ to a formal miniversal deformation. By Artin algebraization~\cite[Section 2.5]{Sernesi}, \cite{ArtinAlgebraization},
there exists an algebraic miniversal deformation
$
\pi\colon\cY\longrightarrow(S,0),
$
that is, 
\[
\widehat{\cO}_{S,0}\cong\widehat{\cO}_{\Def(Y),0},
\qquad
\operatorname{KS}_{\pi,0}\colon T_0S\xrightarrow{\sim}H^1(Y,T_Y).
\]
If $\Def(Y)$ is smooth, we may take $S$ smooth after shrinking.

As mentioned above, the condition that $H^2(Y, \cO_Y)=0$ implies that any line bundle on $Y$ extends etale locally to the total space of any deformation $\cY \rightarrow S$. This implies that 
\[
N^1(\cY/S) \cong N^1(Y)
\]
after making an etale base change. To see this, we consider the real local system $GN^1(\cY/S)$ defined in \cite{KollarMoriFlips}[12.2]. By \cite{KollarMoriFlips}[12.2.5], it has finite monodromy and 
\[
GN^1(\cY/S)|_s=N^1(Y_s, \RR)
\]
for $s$ very general. 
We claim that if $H^2(Y, \cO_Y)=0$, then this equality holds for all $s$. Indeed, after shrinking $S$ around $0$, semicontinuity gives $H^2(Y_s, \cO_{Y_s})=0$ for all $s$. 
We have a natural inclusion 
\[
GN^1(\cY/S) \subset R^2\pi_*\RR
\]
of local systems given by the first Chern class. Since $H^2(Y_s, \cO_{Y_s})=0$, the exponential sequence gives an isomorphism 
$
N^1(Y_s, \RR) \cong H^2(Y_s, \RR)
$
for all $s$. It follows that the inclusion is an isomorphism on very general stalks, hence an isomorphism. 
After passing to an etale cover around zero, we may assume that $GN^1(\cY/S)$ has trivial monodromy, giving isomorphisms
\begin{equation}\label{eq:restriction}
	N^1(\cY/S) \cong N^1(Y_s)
\end{equation}
for all $s \in S$. In particular, we can always choose a miniversal family satisfying \eqref{eq:restriction}.
In the analytic category, the corresponding statement for singular cohomology is of course just Ehresmann's theorem.

Using \eqref{eq:restriction}, we identify the spaces $N^1(Y_s)$ for $s \in S$, and we consider the nef, movable, effective cones of $Y_s$ as a varying family of cones in the fixed vector space $N^1(Y, \RR)$ . 

We say that $Y$ has unobstructed deformations if $\Def(Y)$ is a smooth analytic germ. A sufficient condition for unobstructedness is that $H^2(Y, T_Y)=0$, but it is not necessary: by the BTT theorem, Calabi--Yau threefolds have unobstructed deformations even though this group does not vanish. 

If $\Def(Y)$ is smooth, every class $\xi\in H^1(Y,T_Y)$ can be
realized as the Kodaira--Spencer class of a deformation over a smooth
pointed curve.

We wish to compare the nef and movable cones of the fibers $\cY_s$ in a miniversal deformation.
If $Y$ has unobstructed deformations and $S$ denotes its miniversal deformation space, then standard Hilbert-scheme arguments give a subset
\[
S^{\mathrm{vg}}=S\setminus\bigcup_{i\geq1}Z_i,
\]
where each $Z_i$ is a proper closed subset of $S$, on which the nef,
effective, and movable cones of the fibers are constant, see~\cite[Section 12.2]{KollarMoriFlips}. We denote these cones by $\Nef(Y^{\mathrm{gen}})$,
$\Eff(Y^{\mathrm{gen}})$, and $\Mov(Y^{\mathrm{gen}})$, and write $\Bbig(Y^{\mathrm{gen}})$
for the interior of $\Eff(Y^{\mathrm{gen}})$, so that $Y^{\mathrm{gen}}$ denotes a very
general fiber of a miniversal deformation. 
\begin{lemma}\label{lem:ConeInclusions}
	Let $Y$ be a smooth projective variety. We have the following inclusions:
	\begin{enumerate}
		\item $\Nef(Y) \subset \Nef(Y^{\mathrm{gen}})$
		\item $\Eff(Y^{\mathrm{gen}}) \subset \Eff(Y)$
		\item $\Bbig(Y^{\mathrm{gen}}) \subset \Bbig(Y)$
		\item $\Nef(Y) \cap \Bbig(Y) = \Nef(Y) \cap \Bbig(Y^{\mathrm{gen}})$ 
	\end{enumerate}
\end{lemma}
\begin{proof}
	See \cite[Lemma 3.3]{Lutz}.
\end{proof}
In general, we do not have an inclusion of movable cones, but if $-K_Y$ is nef, such an inclusion does hold (Lemma~\ref{lem:movable-cone-inclusions}). 

\begin{definition}[Persistence of facets]\label{def:face-persistence}
	Assume $H^2(Y,\cO_Y)=0$, and let
	$\pi\colon\cY\rightarrow S$ be a versal deformation of $Y$.
	Let $F$ be a facet of $\Nef(Y)$
	(respectively, of $\overline{\Mov}(Y)$).
	We say that $F$ \emph{persists} if its supporting hyperplane also supports a facet of 
	$\Nef(Y^{\mathrm{gen}})$ (respectively, of
	$\overline{\Mov}(Y^{\mathrm{gen}})$).
	Otherwise, we say that $F$ \emph{disappears}.
\end{definition}
We next have a lemma about deforming contraction morphisms.
\begin{lemma}[Extension of contractions in families]\label{lemma:contractions-extend-in-families}
	Let $p \colon Y \rightarrow \bar{Y}$ be a primitive Mori contraction for a klt pair $(Y, \Delta)$. Let $\pi\colon\cY \rightarrow (S,0)$ be a deformation of $Y$.
	Suppose that $H^2(Y,\cO_Y)=0$.
	Then there is a morphism $P \colon \cY \rightarrow \bar{\cY}$ over the etale germ of $(S,0)$ of relative Picard rank one which restricts to $p$ on the central fiber.
\end{lemma}	
\begin{proof}
	Let $R\subset \overline{\NE}(Y)$ be the extremal ray contracted by $p$. Since $p$ has relative Picard rank one,
	\[
	p^*N^1(\bar Y,\RR)=R^\perp\subset N^1(Y,\RR).
	\]
	Choose line bundles $\bar L_1,\dots,\bar L_r$ whose numerical classes span $N^1(\bar Y,\RR)$, and choose $A\in\Pic(\bar Y)$ sufficiently ample that $A$ and all $A\pm\bar L_i$ are basepoint free and have vanishing higher cohomology. Set
	\[
	H=p^*A,\qquad L_i=p^*\bar L_i.
	\]
	The line bundle $H$ defines $p$. Since $p$ is a Mori contraction for the klt pair $(Y,\Delta)$, $-(K_Y+\Delta)$ is $p$-ample. Hence $H-(K_Y+\Delta)$ and $(H\pm L_i)-(K_Y+\Delta)$ are also $p$-ample. Relative Kawamata--Viehweg vanishing \cite[Theorem 1.2.5]{KawamataMatsudaMatsuki} and the Leray spectral sequence give
	\[
	H^j(Y,\cO_Y(H))=H^j(\bar Y,\cO_{\bar Y}(A))=0\qquad (j>0),
	\]
	and similarly
	\[
	H^j(Y,\cO_Y(H\pm L_i))=H^j(\bar Y,\cO_{\bar Y}(A\pm\bar L_i))=0\qquad (j>0).
	\]
		Since $H^2(Y,\cO_Y)=0$, the relative Picard functor is smooth at every point of $\Pic(Y)$. Thus, after shrinking $S$ and making an \'etale base change, the finitely many line bundles $H,L_1,\dots,L_r$ extend to line bundles $\cH,\cL_1,\dots,\cL_r$ on $\cY$.

	Cohomology and base change shows, after shrinking $S$ further, that 
	\[
	H^0(\cY,\cH) \rightarrow H^0(Y, H)
	\]
	is surjective. Hence $\cH$ defines a morphism
$
	P\colon\cY\longrightarrow\bar{\cY}
	$
	over $S$ whose restriction to the central fiber is $p$.

	It remains to prove that $P$ has relative Picard rank one. Using cohomology and base change again, we see that $\cH \pm \cL_i$ are $P$-nef. If a curve $C\subset\cY_s$ is contracted by $P_s$, then $\cH_s\cdot C=0$, and it follows that $\cL_{i,s}\cdot C=0$ for every $i$. Under parallel transport, $[C]$ therefore lies in the line spanned by $R$; positivity against an ample class shows that it lies in the ray $R$. It follows that every curve contracted by $P$ lies in the ray $R$, and hence $\rho(\cY/\bar{\cY})=1$.
\end{proof}
Whether a nef or movable facet associated to a contraction $p \colon Y \rightarrow \bar{Y}$ disappears can be read off from the deformed morphism $P \colon \cY \rightarrow \bar{\cY}$ in a versal family.

\begin{proposition}\label{prop:persistence-via-contraction}
	Let $Y$ be a smooth projective threefold with $H^2(Y,\cO_Y)=0$. Let $p \colon Y \rightarrow \bar{Y}$ be a primitive contraction which is a Mori contraction for a klt pair, and let $F$ be the corresponding facet of $\Nef(Y)$. Let $\cY \rightarrow (S,0)$ be a versal deformation of $Y$, and let $P \colon \cY \rightarrow \bar{\cY}$ be the induced contraction from Lemma~\ref{lemma:contractions-extend-in-families}.
	\begin{itemize}
		\item $F$ persists as a nef facet iff $P_s$ is not an isomorphism for $s \in S$ very general. 
		\item Suppose that $p$ is divisorial. Then the facet of $\overline{\Mov}(Y)$ containing $F$ persists iff $P_s$ is divisorial. 
	\end{itemize}
\end{proposition}
\begin{proof}
	Let $P\colon\cY\to\bar{\cY}$ be the extension from Lemma~\ref{lemma:contractions-extend-in-families}.
	Let $H\in \relint(F)$ be a nef line bundle defining the contraction, and let $\cH$ be the corresponding line bundle on $\cY$. 
	$F$ disappears as a nef facet precisely when $\cH_s$ is ample on $\cY_s$ for $s$ very general. This happens iff $P_s$ is an isomorphism. 
	
	The movable facet containing $F$ disappears precisely when $\cH_s$ is in the interior of $\overline{\Mov}(\cY_s)$ for $s$ very general. By~\cite[Lemma 2.3]{KawamataCrepant}, this happens precisely when the exceptional locus of $P_s$ has codimension at least two. 
\end{proof}
In general, it is not easy to show whether a facet disappears or deforms. This becomes much easier if $F$ is the facet associated to a primitive contraction morphism, as in this case, the question of persistence of $F$ is closely related to the deformation theory of the exceptional locus of the contraction.

Fix $\cY \rightarrow S$ a miniversal deformation of $Y$. We say that a subscheme $Z\subset Y$ deforms with $Y$ if, after an \'etale base change, there exists a subscheme $\cZ \subset \cY$, flat over $S$ and restricting to $Z$ on the central fiber.
If $Z$ is an lci subscheme, then $N_{Z/Y}$ is a vector bundle, and the natural morphism $T_Y\to i_*N_{Z/Y}$
induces an obstruction map
\[
\ob_Z\colon H^1(Y,T_Y)\longrightarrow H^1(Z,N_{Z/Y}).
\]
If $\xi\in H^1(Y,T_Y)$ is a first-order deformation of $Y$, then $Z$
extends as an embedded flat subscheme over that deformation if and
only if $\ob_Z(\xi)=0$.
In particular, if $H^1(Z,N_{Z/Y})=0$, then $Z$ always deforms with $Y$. 

The statement that $Z$ deforms with $Y$ is stronger than
requiring an irreducible component of $\operatorname{Hilb}(\cY/S)$
through $[Z]$ to dominate $S$. Indeed, such a component might ramify at $[Z]$, so that there is no flat family $\cZ \rightarrow S$ extending $Z$. 

We have the following easy result:
\begin{proposition}[Exceptional-locus criterion for persistence]\label{prop:easy-persistence-criteria}
Let $p \colon Y \rightarrow \bar{Y}$ be a primitive contraction as above, and let $R_p$ be the contracted extremal ray. With notation as above, we have
	\begin{itemize}
		\item $F$ persists as a nef facet iff $\Hilb(\cY/S)$ has a component dominating $S$ and containing a point $[Z]$ such that the numerical class of $Z$ is in $R _p\setminus \{0\}$.
		\item Suppose that $p$ is divisorial with exceptional divisor $E$. If $E$ deforms with $Y$, then $F$ and the corresponding movable facet both persist. In particular, this holds if $H^1(N_{E/Y})=0$.
	\end{itemize}
\end{proposition}
\begin{proof}
Let $H\in \relint(F)$ be a nef line bundle defining the primitive contraction, and let $\cH$ be the corresponding line bundle on $\cY$. The facet $F$ disappears if and only if $\cH_s$ is ample. 
	By Lemma~\ref{lemma:contractions-extend-in-families},
	$\cH_s$ is ample precisely when $\cY_s$ contains no curve with class in $R_p$. This happens precisely if the components of the Hilbert scheme containing such curves don't dominate $S$.
	
	Suppose that $E$ deforms to a relative Cartier divisor $\cE$. 
	For every $s$ near the central point,
	\[
	\cH_s^2\cdot E_s=H^2\cdot E=0.
	\]
	Hence $P_s$ contracts the divisor $E_s$ and $P_s$ is divisorial. It follows that $F$ and the corresponding movable facet both persist.
	
For the last claim,	it is well-known that an lci subscheme $E$ deforms with $Y$ if $H^1(N_{E/Y})=0$, see \cite[Proposition 4.4.7]{Sernesi}.
\end{proof}
\section{$K$-trivial contractions of smooth threefolds}\label{sec:contractions}
Throughout this section, $Y$ is a smooth projective threefold. We do not assume that $-K_Y$ is nef. 
Following established terminology, we say that a primitive birational contraction morphism is 
\begin{itemize}
	\item type I if it is small.
	\item type II if it is divisorial and the exceptional divisor is contracted to a point.
	\item type III if it is divisorial and the exceptional divisor is contracted to a curve.
\end{itemize}
We will find it convenient to adopt the following terminology:
\begin{definition}\label{def:conic-bundle}
	Let $E$ be a projective irreducible Cohen--Macaulay surface, and $C$ a smooth curve. 
	
	We say that a flat morphism $q \colon E \rightarrow C$ with connected fibers is a conic fibration if 
	 the geometric generic fiber of $q$ is either a smooth conic or a reducible conic. 
	 
	We say that $E$ is a conic bundle if $E$ additionally admits a $q$-very ample line bundle of fiber degree two.
	
	We say that $E$ is a ruled surface if $q$ is a $\PP^1$-bundle.
	 
	We say that $E$ is quasi-ruled if the normalization $\tilde{E} \rightarrow E$ is a $\PP^1$-bundle over an unramified cover $\tilde{C} \rightarrow C$. So $E$ is either ruled, or a conic bundle all of whose fibers are line pairs. 
\end{definition}
With this terminology, a conic fibration can have very complicated singular fibers, they can be arbitrary trees of rational curves. On the other hand, the singular fibers of a conic bundle are either a pair of lines or a double line. Our definition of a conic bundle is still more general than the one usually considered as we allow the general fiber to be a line pair. In particular, the total space is not necessarily normal.
 
Our next result is a classification of the structure of exceptional divisors of primitive $K$-trivial type III contractions. We only summarize the main steps whenever the corresponding argument in \cite{WilsonCone, WilsonSymplectic} for Calabi--Yau threefolds carries over verbatim. 
\begin{proposition}[Wilson-type analysis of contractions to curves]\label{prop:wilson-type-contraction-to-curve}
	Let $p\colon Y\to\bar Y$ be a primitive $K$-trivial type III contraction. Let $E$ be the exceptional divisor, and let $C=p(E)$. Then:
	\begin{enumerate}
		\item
		$\bar Y$ is normal and Gorenstein with cDV singularities along $C$, and the contraction 
	$p$ is crepant.
		\item By \cite[Theorem 3.1]{WilsonSymplectic}, the curve $C$ is smooth~\footnote{On 23 October 2026, a preprint by Brown and Wemyss appeared on the arXiv, giving a local example of a crepant type III contraction with singular centre \cite[Proposition 9.14]{Wemyss}. As discussed in \cite[Remark 9.17]{Wemyss}, this example is incompatible with the smoothness of $C$ asserted in \cite{WilsonCone, WilsonSymplectic}, on which the present paper relies. We are working on a revised version of this paper that is independent of these results of Wilson.} , and $p$ is the blowup of $\bar{Y}$ in $C$. In particular, the restriction $q:=p|_E\colon E\to C$ is a conic bundle, and 
		\[
		E \cdot \ell_E=-2,
		\]
		where $\ell_E\in N_1(Y,\ZZ)$ is the class of a fiber of $q$. 
		\item We have an isomorphism
		\[
		\omega_Y^{-1}|_E\otimes q^*\omega_C\cong q^*M_E,
		\]
		where $M_E=\omega_{\bar{Y}}^{-1}|_C\otimes\omega_C\in\Pic(C)$ is the obstruction bundle~ \eqref{eq:obstruction-bundle}.
		\item The singularities of $E$ are du Bois. If the general fiber of $E$ is $\PP^1$, $E$ has at worst du Val singularities. 
	\end{enumerate}
\end{proposition}
\begin{proof}
	\smallskip\noindent\emph{Singularities of $\bar{Y}$.}
	Let $\ell$ be the extremal ray contracted by $p$. By negativity of contraction, we have that $E \cdot \ell<0$. The pair $(Y, \epsilon E)$ is klt for $\epsilon$ small and $K_Y \cdot \ell=0$ by assumption, so it follows that $p$ is a $(K_Y+\epsilon E)$-negative contraction for the klt pair $(Y, \epsilon E)$. In particular, the cone theorem applies~\cite[Theorem 3.7]{KollarMori}, and by~\cite[Corollary 3.18]{KollarMoriFlips} $\bar Y$ is normal and
	$\QQ$-factorial. 
	
	Since $p$ is $K$-trivial, \cite[Theorem 3.7(4)]{KollarMori} shows that $\omega_Y\cong p^*L$ for some $L \in \Pic(\bar{Y})$. Let $\omega_{\bar{Y}}$ be the dualizing sheaf of $\bar{Y}$, a reflexive sheaf of rank one. 
	On a subset $U$ whose complement has codimension two, we have that $L|_{U} \cong \omega_{\bar{Y}}|_U$. Since $\bar{Y}$ is normal, this implies that $L \cong \omega_{\bar{Y}}$ is a line bundle, so that $\bar{Y}$ is Gorenstein. 
	Intersecting the equality
	\[
	K_Y =p^*K_{\bar{Y}}+a(E)E
	\]
	with $\ell$ and using that $K_Y \cdot \ell=0$ shows that $a(E)=0$, so that $p$ is crepant. Since $Y$ is smooth, $\bar{Y}$ has canonical singularities. 
	Since $E$ is the only exceptional divisor of a resolution of $\bar{Y}$, \cite{Reid1}[Theorem 2.2] shows that $\bar{Y}$ has cDV singularities. 	

	\smallskip\noindent\emph{Conic bundle}
	The fact that $C$ is smooth and that $p$ is the blowup is proved by the same argument as in Wilson~\cite[Proposition 3.1]{WilsonSymplectic}. Indeed, the proof given there is local, so goes through under the weaker assumption that $p$ is a crepant contraction from a smooth threefold, see also \cite[Theorem 5.5]{Wierzba}. 
	 The equality $E \cdot \ell_E=-2$ follows from the adjunction formula on the Cohen--Macaulay surface $E$. 
	
		Finally we show the existence of a relatively ample line bundle of fiber degree two. 
	Every fiber $F$ of $q$ is a tree of rational curves, and since the contraction is primitive, all of its components are numerically proportional to $\ell_E$ in $N_1(Y)$. Since $E \cdot \ell_E=-2$, it follows that $\cO_F(-E)$ has positive degree on every component of $F$, and is therefore very ample on $F$. 
	It follows that $\cO_E(-E)$ is $q$-very ample of fiber degree two. We conclude that $q$ is a conic bundle. 
	
	\smallskip\noindent\emph{Descent of $\omega_Y^{-1}|_E$.}
	Denoting $i \colon E \rightarrow Y$ and $j \colon C \rightarrow \bar{Y}$, we have that 
	\[
	\omega_Y|_E \cong i^*p^*\omega_{\bar{Y}}\cong q^*(\omega_{\bar{Y}}|_C),
	\]
	and the claim follows.
	
	\smallskip\noindent\emph{Singularities of $E$.}
	The statement about the singularities of $E$ follows from \cite[Theorem 2.2]{WilsonCone} in the $\PP^1$ case and \cite[Proposition 3.2]{WilsonSymplectic} in the case that the general fiber is two $\PP^1$'s. In the former case, $E$ has at worst du Val singularities, which are du Bois by \cite{Kovacs}. In the latter case, $E$ has generically a normal crossing singularity at the points where two fiber components meet, and possibly also a pinch point singularity on special fibers.
	Both are du Bois, see e.g. \cite[Examples~3.2--3.3]{Doherty}.
\end{proof}
\section{Deformations of $K$-trivial facets}\label{sec:deformations}
We now turn to the question whether $K$-trivial facets of the nef and movable cone persist under deformation. 

Our first result shows that faces of the nef cone associated to small $K$-trivial contractions always persist.
\begin{proposition}[Deformation of isolated $K$-trivial curves]\label{prop:k-trivial-curves-persist-in-families}
	Let $Y$ be a smooth projective threefold with $H^2(Y,\cO_Y)=0$, let $\pi\colon\cY\to S$ be an arbitrary deformation of $Y$, and let $Z \subset Y$ be a smooth rational $K$-trivial curve which does not deform inside $Y$. Then the union of components of $\Hilb(\cY/S)$ containing $[Z]$ dominates $S$. 
	In particular, if $p \colon Y \rightarrow \bar{Y}$ is a $K$-trivial small contraction, the corresponding nef facet persists. 
\end{proposition}
\begin{proof}
	The proof is very similar to \cite[Proposition 3.1]{WilsonCone}. Let $\cH$ be the component of $\operatorname{Hilb}(\cY/S)$ containing $Z$, and let $g\colon\cH\to S$ be the structure morphism. We have that $\dim_{[Z]}(\cH) \geq h^0(N_{Z/\cY})-h^1(N_{Z/\cY})$. By Riemann--Roch, the right-hand side is
	\[
	\chi(N_{Z/\cY})=\deg(N_{Z/\cY})+\dim(S)+2.
	\]
	By adjunction, we have 
	\[
	\deg(N_{Z/\cY})=\deg(N_{Z/Y})=-K_Y \cdot Z-2=-2,
	\]
	so that $\dim_{[Z]}(\cH)\geq\dim(S)$. Since $Z$ does not deform in $Y$, the fiber of $g$ through $[Z]$ is zero-dimensional. Hence $g$ is dominant. It follows that every deformation of $Y$ contains a curve $C$ with $[C]=[Z]$. We conclude by Proposition~\ref{prop:easy-persistence-criteria} that the corresponding facet of $\Nef(Y)$ persists. 
\end{proof}
We now turn to divisorial contractions to points. 
\begin{lemma}[Normal-bundle obstructions for exceptional divisors]
	\label{lem:normal-obstructions-exceptional-divisor}
	Let $p \colon Y \rightarrow \bar{Y}$ be a primitive $K$-trivial
	divisorial contraction with exceptional divisor $E$. 
	\begin{enumerate}
	\item 
If $p$ is of type II, then
		\[
		H^1(E,N_{E/Y})=0.
		\]
		\item If $p$ is of type III, then 
		\[
		H^1(E,N_{E/Y})
		\cong H^0(C,M_E).
		\]
		where $M_E=\omega_{\bar{Y}}^{-1}|_C\otimes\omega_C$.
	\end{enumerate}
\end{lemma}
\begin{proof}
	Since $E$ is a Cartier divisor, it is Cohen--Macaulay, so Serre duality and the adjunction formula yield
	\begin{equation}\label{eq:normal-bundle-H1}
	H^1(E,N_{E/Y})^\vee
	\cong H^1(E,\omega_E\otimes N_{E/Y}^{-1})
	\cong H^1(E,\omega_Y|_E).
	\end{equation}
	For (1), let $H$ be a big and nef line bundle defining $p$.	Write $H=p^*A$, for $A$
	ample on $\bar Y$. For $m\gg0$, the divisors $mH$ and $mH-E$ are big
	and nef, so Kawamata--Viehweg vanishing gives
	\[
	H^i(Y,\cO_Y(K_Y+mH))=
	H^i(Y,\cO_Y(K_Y+mH-E))=0\qquad(i>0).
	\]
	The exact sequence
	\[
	0\to\cO_Y(K_Y+mH-E)\to\cO_Y(K_Y+mH)
	\to\cO_E(K_Y|_E)\to0
	\]
	therefore gives $H^1(E,\omega_Y|_E)=0$, where we used
	$H|_E\cong\cO_E$ since $E$ is contracted to a point. We conclude using \eqref{eq:normal-bundle-H1}.
	
	For (2), write $\omega_Y|_E\cong q^*(\omega_{\bar{Y}}|_C)$ as in the proof of Proposition~\ref{prop:wilson-type-contraction-to-curve}. The Leray spectral sequence applied to $q \colon E \rightarrow C$ gives an isomorphism 
	\[
	H^1(E,\omega_Y|_E)\cong H^1(C,\omega_{\bar{Y}}|_C),
	\]
	where we used $q_*\cO_E=\cO_C$ and $R^1q_*\cO_E=0$. 
	Using Serre duality on $C$, we obtain 
	\[
	H^1(C,\omega_{\bar{Y}}|_C)^\vee \cong H^0(C,M_E)
	\]
	and we conclude using~\eqref{eq:normal-bundle-H1}.
\end{proof}
In view of this, we see that the only case where the exceptional divisor $E$ of a $K$-trivial primitive contraction can disappear under deformation is when $q \colon E \rightarrow C$ is a conic fiber space and $H^0(C,M_E)\neq 0$. By definition, this implies that $E \in \Delta_Y^{\mathrm{big}}$ and motivates Definition~\ref{def:deformation-root-classes}. 

We now turn to the much harder question whether an exceptional divisor $E$ with $E \in \Delta_Y^{\mathrm{big}}$ actually disappears under deformation. 
Recall that if \(\alpha\in H^1(Y,T_Y)\) is a first-order deformation of \(Y\), then \(E\) deforms with \(Y\) to first order if and only if \(\ob_E(\alpha)=0\), where
\[
\ob_E\colon H^1(Y,T_Y)\longrightarrow H^1(E,N_{E/Y})
\]
is induced by the composite \(f\colon T_Y\to T_Y|_E\to N_{E/Y}\). In order to show that $E$ disappears under deformation, it is therefore enough to show that \(H^1(E,N_{E/Y})\neq 0\) and that \(\ob_E\) is surjective.
For this, we use Saito's theory of Hodge modules, which we recall in Appendix~\ref{app:hodge}.

\begin{theorem}\label{thm:surjective}
	Let $p \colon Y \rightarrow \bar{Y}$ be a divisorial contraction as in Proposition~\ref{prop:wilson-type-contraction-to-curve}. Then the obstruction map
	\[
	\ob_E \colon H^1(Y,T_Y)\longrightarrow H^1(E,N_{E/Y})
	\]
	is surjective.
\end{theorem}
\begin{proof}
	We first make a few reductions. 
	Since $Y$ is smooth, we have $T_Y\simeq \Omega_Y^2\otimes \omega_Y^{-1}$.
	Since $E$ is a hypersurface, it is Cohen--Macaulay, and the adjunction formula gives
	$
	N_{E/Y}\simeq \omega_E\otimes \omega_Y^{-1}|_E.
	$
	We view $f\colon T_Y\to T_Y|_E\to N_{E/Y}$ as a morphism
	\[
	f\colon \Omega_Y^2 \otimes \omega_Y^{-1} \longrightarrow i_*(\omega_E \otimes \omega_Y^{-1}|_E).
	\]
	This agrees with the morphism $g \otimes \omega_Y^{-1}$, where
	\[
	g \colon\Omega_Y^2\longrightarrow i_*\Omega_E^2
	\xrightarrow{\,c\,}i_*\omega_E,
	\]
	given by restriction of forms, followed by the morphism \(c\colon \Omega_E^2\to\omega_E\) of \cite[Lemma 48.29.2]{stacks-project}. On the smooth locus of $E$, $c$ is the canonical isomorphism, so compatibility is clear on the smooth locus of $E$. Since $\omega_E$ is locally free, compatibility follows everywhere. 
	
	By Proposition~\ref{prop:wilson-type-contraction-to-curve}, we have \(\omega_Y\cong p^*\omega_{\bar{Y}}\), so the projection formula gives
	\[
	Rp_*f\simeq Rp_*g\otimes \omega_{\bar{Y}}^{-1}.
	\]
	It is
	therefore enough to prove that
	\[
	Rp_*g\colon Rp_*\Omega_Y^2\longrightarrow j_*Rq_*\omega_E
	\]
	is split surjective in $D^b_{coh}(\bar{Y})$. 
	
	We will prove this by lifting everything to morphisms in the derived category of Hodge modules.
	Consider the morphism
	\[
	\Psi\colon \mathbb Q_Y^H
	\longrightarrow i_*\mathbb Q_E^H
	\xrightarrow{c_E} i_*\mathbb{D}(\mathbb Q_E^H)(-2)[-4],
	\]
	where $c_E$ is Saito's fundamental class morphism~\eqref{eq:appendix-fundamental-class}.
	We apply the functor $G_2=\Gr^F_{-2} \DR (-)$ to $\Psi$. Using the compatibility of the functor $G_2$ with duality~\eqref{eq:compatibilitydirectimageduality} and \eqref{eq:appendix-fundamental-graded} gives
	\begin{align*}
		G_2(\QQ_Y^H)&\cong \Omega^2_Y[-2]\\
		G_2(\mathbb{D}(\QQ_E^H)(-2)[-4])&\cong D_E(G_0(\QQ_E^H))[-4] \cong D_E(	\underline{\Omega^0_E})[-4]=\omega_E[-2]
	\end{align*}
	where we have used that $E$ has du Bois singularities by  Proposition~\ref{prop:wilson-type-contraction-to-curve}, and that $E$ is Cohen--Macaulay, so that \eqref{eq:appendix-fundamental-graded} gives
	\[
	D_E(\underline{\Omega^0_E})\cong D_E(\cO_E) \cong \omega_E[2].
	\] 
	Restricting to the smooth locus, $G_2(c_E)$ is the isomorphism $\Omega^2_{E^{reg}} \cong \omega_{E^{reg}}$, so it follows that $G_2(\Psi)[2]=g$. 
	Let 
	\[
	u \colon \QQ_C^H \rightarrow Rq_*\QQ_E^H
	\]
	be the canonical morphism induced by
	\(q\colon E\to C\).
	Using that $E$ has du Bois singularities again, we see that 
	\[
	G_0(u) \colon \cO_C \rightarrow Rq_*\cO_E
	\]
	is a quasi-isomorphism: indeed, since $q$ is a conic bundle, we have $q_*\cO_E=\cO_C$ and $R^iq_*\cO_E=0$ for $i>0$.
  Denoting
		\[
		\tau\colon Rq_*\mathbb{D}(\QQ_E^H)\longrightarrow \mathbb{D}(\QQ_C^H) \cong \QQ_C^H(1)[2]
		\]
	 the dual of $u$, we obtain that
	\[
	G_0(\tau) \colon Rq_*\omega_E[2] \rightarrow \omega_C[1]
	\]
	is a quasi-isomorphism as well. 

	Write $P=Rp_*\QQ_Y^H[3]$, and let 
	\begin{equation}\label{eq:hodge-trace-map}
		\Phi \colon P \xrightarrow{Rp_*\Psi[3]} j_*Rq_*\mathbb{D}(\QQ_E^H)(-2)[-1] \xrightarrow{j_*\tau[-1](-2)} j_*\QQ_C^H(-1)[1].
	\end{equation}
	Since $p$ is a type III contraction, $p$ is a semismall morphism. Since $Y$ is nonsingular, the Hodge module decomposition theorem~\ref{thm:semismall-hodge} shows that $P$ is a semisimple pure Hodge module of weight three. 
	On the other hand, $j_*\QQ_C^H[1](-1)$ is a simple pure Hodge module of the same weight $1-2(-1)=3$, so it follows that $\Phi$ is a morphism of pure Hodge modules. 
	We claim that $\Phi$ is nonzero, and therefore split surjective. This is easy to see on the underlying perverse sheaves. Taking the stalk of the cohomology sheaf in degree $-1$ at $c \in C$ general gives 
	\[
	H^2(E_c, \QQ)  \rightarrow \QQ(-1),
	\]
	the fundamental class map, which is nonzero.
	Therefore 
	\[
	G_2(\Phi)[-1]\colon Rp_*\Omega^2_Y\xrightarrow{Rp_*g} j_*Rq_*\omega_E \xrightarrow{j_*G_0(\tau)[-2]} j_*\omega_C[-1]
	\]
	is split surjective as well.
	By the above, $G_0(\tau) $ is a quasi-isomorphism, so $Rp_*g$ is split surjective, and induces surjections on hypercohomology. In particular, $\ob_E$ is surjective.
\end{proof}

We have already seen in Proposition~\ref{prop:easy-persistence-criteria} that if $E$ deforms, then the corresponding movable facet and nef facet persist under deformation. The converse is far from obvious: even if $E$ does not deform, some thickening of it might, causing the movable facet to persist. 
The next result shows that this does not happen for the exceptional divisors in Proposition~\ref{prop:wilson-type-contraction-to-curve}. The proof uses the very special geometry of such contractions: $C$ is a curve of transverse du Val singularities, so locally in the analytic topology, we may view the contraction $Y \rightarrow \bar{Y}$ as a family of du Val singularities with simultaneous resolution. We briefly recall the theory. 

Let $X$ be the germ of a du Val singularity and let $\cX \rightarrow \Def(X)$ be the versal deformation. By the theory of Brieskorn--Tyurina, there is a finite map $T \rightarrow \Def(X)$ such that the pulled back family admits a simultaneous resolution $\tilde{\cX} \rightarrow \cX_T \rightarrow T$. 
The discriminant locus $\Gamma \subset T$ is given by the root hyperplanes of the root system of ADE type corresponding to the du Val singularity, and the finite map $T \rightarrow \Def(X)$ is a Galois cover for the associated Weyl group, when restricted to $T \setminus \Gamma$. 
We will only consider the case of $A_1$ and $A_2$ singularities. In the former case, we have that 
\[
\Def(X)=\AAA^1, \qquad T=\AAA^1, \qquad \Gamma=0
\]
In the latter case, we have 
\[
\Def(X)=\AAA^2, \qquad T=\AAA^2
\]
and $\Gamma$ is the union of the three lines in the standard $A_2$ root system picture.
The simultaneous resolution of $X$ has a versality property:
\begin{lemma}\cite[Corollary 5.3]{Huikeshoven}\label{lem:versality}
Let $\pi \colon \cY \rightarrow S$ be a deformation of a du Val singularity $X$, and suppose that $\pi$ admits a simultaneous resolution 
\[
\tilde{\cY} \rightarrow \cY \rightarrow S
\]
Then there is an analytic morphism $\varphi\colon S\to T$ such that
$
\tilde{\cY}\rightarrow \cY \rightarrow S
$
is isomorphic to the pullback of
$
\tilde{\cX} \rightarrow \cX_T \rightarrow T
$
via $\varphi$.
\end{lemma}
\begin{proposition}\label{prop:root-faces-disappear-movable}
	Let $Y$ be a smooth projective threefold with $\Def(Y)$ smooth and $H^2(Y, \cO_Y)=0$, and let $\pi \colon \cY \rightarrow (S,0)$ be a versal deformation. 
	Let $p \colon Y \rightarrow \bar{Y}$ be a primitive $K$-nonpositive divisorial contraction with exceptional divisor $E$, and let $F$ be the corresponding facet of $\overline{\Mov}(Y)$.  TFAE
	\begin{enumerate}
		\item $p$ is a $K$-trivial type III contraction with exceptional divisor $E$, and $H^0(M_E) \neq 0$.
		\item The deformed contraction $P \colon \cY \rightarrow \bar{\cY}$ is a small contraction.
		\item The facet $F$ disappears.
	\end{enumerate}
\end{proposition}
\begin{proof}
	 The equivalence $(2)\iff(3)$ is Proposition~\ref{prop:persistence-via-contraction} and only included here for the convenience of the reader. The real content is the equivalence with condition $(1)$. 
	
	If $F$ is $K$-negative, then $E$ deforms by~\cite[Theorem 3.47]{MoriThreefolds} and $F$ persists as a movable face, so none of the conditions are satisfied.
	
	Suppose therefore that $F$ is $K$-trivial. If $p$ is a type II or a type III contraction with $H^0(M_E)=0$, then 
	$H^1(E,N_{E/Y})=0
	$
	by Lemma~\ref{lem:normal-obstructions-exceptional-divisor}.
	Therefore $E$ deforms, and $F$ persists as a movable face, so none of the conditions are satisfied. 

	It follows that if $F$ disappears as a movable face, then $p$ is a $K$-trivial type III contraction, and $H^0(M_E) \neq 0$. Therefore the implication $(3) \implies (1)$ holds. 

	$(1) \implies (2)$.	
By Theorem~\ref{thm:surjective}, 
\[
\ob_E \colon H^1(T_Y) \rightarrow H^1(N_{E/Y})
\]
 is surjective and by assumption $H^1(N_{E/Y}) \cong H^0(M_E) \neq 0$. It follows that if $v \in H^1(T_Y)$ is very general, then $\ob_E(v) \neq 0$. Since $\Def(Y)$ is smooth, $v$ is the Kodaira--Spencer class of some (very general) deformation $\cY_\Delta \rightarrow \Delta$. 
 It is enough to show that $P_\Delta \colon \cY_\Delta \rightarrow \cY_{\bar{\Delta}}$ is a small contraction.
	
	Suppose for a contradiction that $P_\Delta$ is divisorial, and denote $\cE$ its exceptional divisor. Then 
	$\cE|_0$ is a hypersurface supported on $E$, and therefore equal to $mE$  for some $m \geq 0$. Here $mE$ denotes the codimension one scheme locally defined by $f^m$ where $f$ is a local equation for $E$. We will show that $m=1$, which will imply that $E$ deforms, contradicting $\ob_E(v) \neq 0$.
	
	As described in \cite{Reid2}, we can locally consider $\cY \rightarrow \bar{\cY}$ as a family of surface singularities with simultaneous resolution over a parameter space $B=C \times \Delta$. 
	Over a general point $b=(c,0)$, ${\mathcal{Y}}_b \rightarrow \bar{\mathcal{Y}}_b$ is the minimal resolution of a surface with a unique $A_1$ or $A_2$ singularity, depending on whether the general fiber of $q \colon E \rightarrow C$ is $\PP^1$ or the union of two $\PP^1$'s.
	
Since $P_\Delta$ is divisorial, the induced morphism $\cE \rightarrow B$ is dominant and its fibers are one-dimensional after shrinking $B$. Since every adjacent singularity to an $A_1$ or $A_2$-singularity is again an $A_1$ or $A_2$, it follows that for general $b$, $\cY_b \rightarrow \bar{\cY}_b$ is the minimal resolution of a surface with a unique $A_1$ or $A_2$ singularity, and hence $\cE_b^2=-2$ in the surface $\cY_b$ for $b$ general. 

On the other hand, over $b=(c,0)$, we have that $[\cE_b]=m[E|_b]=m[G]$, where $G$ is a general fiber of $q$, and $G^2=-2$ in the surface $\cY_b$. We conclude that $m=1$, as required.
	\end{proof}
We now turn to deformations of the fibers of $q$, following
\cite{WilsonCone, Minagawa}.
The case where the general fiber of $E$ is a line pair requires some special arguments, for which we set up some notation. We let
$\nu\colon\tilde{E}\to E$ be the normalization,
write $f=i\circ\nu\colon\tilde{E}\to Y$, and
let 
\[
\tilde{E} \xrightarrow{\tilde{q}} \tilde{C} \xrightarrow{\nu_C} C
\]
be the Stein factorization of $q\circ\nu$, so that $\tilde{q} \colon \tilde{E} \rightarrow \tilde{C}$ has general fiber $\PP^1$.
The normal bundle of $f$ is the sheaf defined by the exact sequence
\[
0 \rightarrow T_{\tilde{E}} \rightarrow f^*T_Y \rightarrow N_f \rightarrow 0
\]
As noted in \cite[p.703]{Minagawa}, $N_f$ is torsion free and locally free away from the points lying over the pinch points of $E$. In particular $N_f$ is locally free if $E$ is quasi-ruled, and we have an isomorphism
\[
N_f^{\vee\vee} \cong f^*\omega_Y^{-1} \otimes \omega_{\tilde{E}}
\]
Define
\[
\ob_f\colon H^1(T_Y)\longrightarrow H^1(N_f^{\vee\vee})
\]
by the composite $T_Y\to f_*f^*T_Y\to f_*N_f\to f_*N_f^{\vee\vee}$.
We have the following easy result.
\begin{lemma}\label{lem:obstruction-compatibility}
	Suppose that the general fiber of $E$ is a line pair. The obstruction map $\ob_E$ factors through $\ob_f$.
\end{lemma}
\begin{proof}
Since $\nu$ is finite, the dual of $\cO_E\to\nu_*\cO_{\tilde{E}}$ gives a morphism
	$
	\nu_*\omega_{\tilde{E}}\to\omega_E
	$
by~\cite[48.20.12]{stacks-project}. The composites
\[
T_Y \rightarrow f_*N_f^{\vee\vee}
\]
and 
\[
\nu_*N_f^{\vee\vee} \cong \nu_*(\omega_{\tilde{E}})\otimes \omega_Y^{-1}|_E \rightarrow \omega_E \otimes \omega_Y^{-1}|_E \cong N_{E/Y}
\]
The composite sheaf morphism agrees with $T_Y \rightarrow i_*N_{E/Y}$ on the smooth locus of $E$, as $\nu$ is an isomorphism there. Since $N_{E/Y}$ is locally free, they agree everywhere. It follows that the induced map 
\[
H^1(T_Y) \xrightarrow{\ob_f} H^1(N_f^{\vee\vee}) \xrightarrow{T} H^1(N_{E/Y})
\]
agrees with $\ob_E$.
\end{proof}
Let $F_c$ be the (not necessarily smooth) fiber of $E$ over $c \in C$. Recall from \cite{WilsonCone} and more specifically from the proof of~\cite[Proposition 3.5]{Minagawa} that naturality of the Leray spectral sequence gives a commutative diagram
\[
\begin{tikzcd}[column sep=large, row sep=large]
	H^1(T_Y)
	\arrow[r, "\operatorname{ob}_E"]
	&
	H^1(N_{E/Y})
	\arrow[r]
	\arrow[d, "\sim"]
	&
	H^1(N_{E/Y}|_{F_c}) \cong H^1(N_{F_c/Y})
	\arrow[d, "\sim"]
	\\
	&
	H^0(C,M_E)
	\arrow[r, "\operatorname{ev}_c"]
	&
	H^0\!\left(c,(M_E)|_c\right)
\end{tikzcd}
\]
where the top composite is $\ob_{F_c}$. If $\kappa \in H^1(T_Y)$ is the KS class of a one parameter deformation $\cY \rightarrow \Delta$, and $\sigma=\ob_E(\kappa) \in H^0(M_E)$ is not identically zero, then $F_c$ deforms to first order iff the section $\sigma$ has a zero at $c$. In particular, no fiber deforms to first order iff $M_E$ is trivial. 

In the case that the general fiber is a line pair, write
\[
M_{\tilde{E}}:=
\nu_C^*(\omega_{\bar{Y}}^{-1}|_C)\otimes\omega_{\tilde{C}}
\cong\nu_C^*M_E\otimes\cO_{\tilde{C}}(R_{\nu_C}),
\]
where $R_{\nu_C}$ is the ramification divisor. 
Let $F_{\tilde{c}}$ be the fiber component corresponding to $\tilde{c} \in \tilde{C} \setminus R_{\nu_C}$. We also have a diagram
\[
\begin{tikzcd}[column sep=large, row sep=large]
	H^1(T_Y)
	\arrow[r, "\operatorname{ob}_f"]
	&
	H^1(N_{f}^{\vee\vee})
	\arrow[r]
	\arrow[d, "\sim"]
	&
	H^1(N_{f}^{\vee\vee}|_{F_{\tilde{c}}}) \cong H^1(N_{F_{\tilde{c}}/Y})
	\arrow[d, "\sim"]
	\\
	&
	H^0(\tilde{C},M_{\tilde{E}})
	\arrow[r, "\operatorname{ev}_{\tilde{c}}"]
	&
	H^0\!\left({\tilde{c}},(M_{\tilde{E}})|_{\tilde{c}}\right)
\end{tikzcd}
\]
where the top row is the obstruction map for $F_{\tilde{c}}$. Denoting $\tilde{\sigma}=\ob_f(\kappa)$, Lemma~\ref{lem:obstruction-compatibility} shows that $\tilde{\sigma} \neq 0$ whenever $\sigma \neq 0$. In particular the fiber component $F_{\tilde{c}}$ deforms to first order iff the section $\tilde\sigma$ has a zero at $\tilde{c}$.

\begin{proposition}
	\label{prop:ruling-fiber-persistence-criterion}
	\label{prop:root-faces-disappear-nef}
	Let $Y$ be a smooth projective threefold with $\Def(Y)$ smooth and
	$H^2(Y,\cO_Y)=0$, and let $\pi \colon \cY \rightarrow S$ be a miniversal deformation. Let $p\colon Y\rightarrow\bar Y$ be a $K$-nonpositive primitive birational contraction, and let $F$ be the corresponding facet of $\Nef(Y)$. The following are equivalent:
	\begin{enumerate}
		\item  $p$ is a $K$-trivial type III contraction, and, denoting $E$ the exceptional divisor, we have 
		that $q=p|_E \colon E \rightarrow C$ is quasi-ruled and $M_E \cong \cO_C$.
		\item The deformed contraction
		$P\colon\cY\rightarrow\bar{\cY}$ over $S$ restricts to an
		isomorphism on a very general fiber.
		\item The facet $F$ disappears.
	\end{enumerate}
\end{proposition}

\begin{proof}
	 The equivalence $(2)\iff(3)$ is Proposition~\ref{prop:persistence-via-contraction} and only included here for the convenience of the reader. The real content is the equivalence with condition $(1)$. 
	 
	If $F$ is $K$-negative, then $E$ deforms by~\cite[Theorem 3.47]{MoriThreefolds}, so none of the conditions are satisfied.
	
If $p$ is of type I or II, or of type III with $H^0(C,M_E)=0$, then none of the conditions are satisfied. Indeed, if $p$ is of type I, then Proposition~\ref{prop:k-trivial-curves-persist-in-families} shows that the component of the Hilbert scheme containing the exceptional curve dominates $\Def(Y)$, so that the nef facet persists. If $p$ is of type II or of type III with $H^0(C,M_E)=0$
	then we have
	$
	H^1(E,N_{E/Y})=0
	$
	by Lemma~\ref{lem:normal-obstructions-exceptional-divisor}.
	Therefore $E$ deforms, and $F$ persists, so none of the conditions are satisfied. 
	
	We henceforth assume that $p$ is a $K$-trivial type III contraction and $H^0(C,M_E)\neq0$.
	By Theorem~\ref{thm:surjective}, $\ob_E$ is surjective, so if $v \in H^1(T_Y)$ is very general, then 
	\[
	\sigma:=\ob_E(v)\in H^0(C,M_E)
	\]
	is nonzero.
	If the general fiber of $q$ is two lines, Lemma~\ref{lem:obstruction-compatibility} shows that we may also assume that
	\[
	\tilde\sigma:=\ob_f(v)\in
	H^0\bigl(\tilde{C},M_{\tilde{E}}\bigr)
	\]
	is nonzero. Since $\Def(Y)$ is smooth, $v$ is the Kodaira--Spencer class of some (very general) deformation $\cY_\Delta \rightarrow \Delta$ over a curve $\Delta$. 
	We continue to use $P$ for its restriction to $\cY_\Delta$.

Recall that for an lci subscheme $Z \subset Y$, we have the tangent-obstruction sequence~\cite[Proposition 4.4.7]{Sernesi}
	\[
	0\longrightarrow H^0(N_{Z/Y})
	\longrightarrow T_{[Z]}\Hilb(\cY/\Delta)
	\longrightarrow T_0\Delta
	\xrightarrow{\ob_Z}H^1(N_{Z/Y}).
	\]
	$(3) \implies (1)$.

	\smallskip\noindent\emph{Step 1: We show that $M_E$ is trivial.}
	By assumption, $F$ disappears, so the component of $\Hilb(\cY/\Delta)$ containing fibers of the ruling is supported on the central fiber. It follows that $\Hilb(\cY/\Delta)$ has dimension one at every point $[F_c]$, $c \in C$. We claim that $\Hilb(\cY/\Delta)$ is reduced at $[F_c]$.
For every fiber $F_c$ (not necessarily smooth), the flag
	$F_c\subset E\subset Y$ consists of regular embeddings, so the normal bundle sequence gives
	\[
	h^0(N_{F_c/Y})=1, \qquad h^1(N_{F_c/Y})=1.
	\]
	By \cite{KollarRationalCurves}, the completed local ring of
	$\Hilb(\cY/\Delta)$ at $[F_c]$ has a presentation
	\[
	\CC[[u,t]]/(F),
	\]
	where $t$ is a parameter on $\Delta$ and $F$ is a power series.
	In particular, $\Hilb(\cY/\Delta)$ is Cohen--Macaulay at every point $[F_c]$. 
	Since $\sigma \neq 0$, $\ob_{F_c}$ is nonzero for generic $c$, and the obstruction sequence shows that $T_{[F_c]}\Hilb(\cY/\Delta)$ has dimension one for generic $c$. It follows that $\Hilb(\cY/\Delta)$ is smooth, and hence reduced, at $[F_c]$ for generic $c$. A generically reduced Cohen--Macaulay scheme is reduced, so the claim follows. 
	
	It now follows that $t=0$ in the local ring of $\Hilb(\cY/\Delta)$ at $[F_c]$, and therefore, the differential 
	\[
	T_{[F_c]}\Hilb(\cY/\Delta)
	\longrightarrow T_0\Delta
	\]
	is zero for all $c \in C$. The tangent-obstruction sequence now shows that $\ob_{F_c}$ is nonzero for all $c$, so $\sigma$ must be nowhere vanishing. It follows that $M_E\cong\cO_C$. 

	\smallskip\noindent\emph{Step 2: We show that $E$ is quasi-ruled}. There are two cases to consider:
		
		\smallskip\noindent\emph{Case a: The general fiber of $q$ is irreducible. }
	
	If $q$ has a singular fiber, then $E$ contains a $K$-trivial vertical curve $Z$ which does not deform inside $Y$. By Proposition~\ref{prop:k-trivial-curves-persist-in-families}, the component of $\Hilb(\cY/\Delta)$ containing $[Z]$ dominates $\Delta$ and the nef facet persists, contradiction. It follows that $E$ is ruled in this case. 

		\smallskip\noindent\emph{Case b: The general fiber of $q$ is reducible. }
	
The fibers of $E$ are either line pairs or double lines. 
Since $F$ disappears under deformation, $\Hilb(\cY/\Delta)$ has dimension one at the point corresponding to every reduced fiber component. For every component $F_c$, $c \in \tilde{C}$ of a line pair fiber, the morphism $f \colon \tilde{E} \rightarrow Y$ is unramified in a neighbourhood of $F_c$, so~\cite[37.63(4)]{stacks-project} gives an exact sequence
\begin{equation}\label{eq:quasi-ruled-normal-bundle-sequence}
0 \longrightarrow N_{F_c/\tilde{E}}
\longrightarrow N_{F_c/Y}
\longrightarrow N_{f}|_{F_c}
\longrightarrow 0,
\end{equation}
and hence
\[
h^0(N_{F_c/Y})=1, \qquad h^1(N_{F_c/Y})=1.
\]
Using $\tilde{\sigma} \neq 0$, we conclude as before that $\Hilb(\cY/\Delta)$ is generically smooth at the points $[F_c]$, hence generically reduced. 

Suppose for a contradiction that $E$ has a double fiber, and denote $Z$ its reduction. Using  Riemann--Roch and $K_Y \cdot Z=0$, we have
	\[
	\chi(N_{Z/Y})=0
	\]
 giving $m=h^0(N_{Z/Y})=h^1(N_{Z/Y})$.
	By \cite[Theorem 2.10]{KollarRationalCurves}, the completed local ring of
	$\Hilb(\cY/\Delta)$ at $[Z]$ has a presentation
	\[
	\CC[[t, x_1, \dots, x_m]]/I,
	\]
	where $t$ is a parameter on $\Delta$ and $I$ is generated by at most $m$ elements. Since $\Hilb(\cY/\Delta)$ has dimension one at $[Z]$, $I$ is generated by a regular sequence. It follows that $\Hilb(\cY/\Delta)$ is Cohen--Macaulay at $[Z]$ as well. As before we conclude that $\Hilb(\cY/\Delta)$ is reduced at every point corresponding to a reduced fiber component. 
	In particular, the Hilbert scheme of the central fiber is reduced at $[Z]$.

However, as noted in \cite[p.22]{WilsonSymplectic} the Hilbert scheme has an embedded component
at $[Z]$. Contradiction.
It follows that $E$ cannot have a double fiber and is quasi-ruled.
	
	$(1)\implies(3)$.
	By assumption, $M_E$ is trivial, so $\sigma$ is nowhere vanishing. $E$ is quasi-ruled, so the tangent obstruction sequence shows that $\dim(T_{[F_c]}\Hilb(\cY/\Delta))=1$ for all $c$, so no horizontal component
	of $\Hilb(\cY/\Delta)$ contains $[F_c]$ for any $c\in C$. 
	If every fiber of $E$ is the union of two
	$\PP^1$'s, then the cover $\nu_C$ is etale, and
	\[
	M_{\tilde{E}}
	\cong\nu_C^*M_E\cong\cO_{\tilde{C}}.
	\]
	Then $\tilde\sigma$ is also nowhere vanishing, so using the tangent-obstruction sequence again, no horizontal component contains a fiber component $[F_c]$ for $c \in \tilde{C}$. 
	
	It remains to rule out the possibility that a curve on a general fiber specializes to a scheme with support equal to a fiber or fiber component, but with a different scheme structure. For this, we again use the local description of $p$ as a deformation of the $A_1/A_2$-singularity.
	
\smallskip\noindent\emph{Case a): $E$ is ruled.}
	Every point of $C$ is a
	$cA_1$-singularity, so locally we view
	$\cY\rightarrow\bar{\cY}$ as a deformation of the
	$A_1$-singularity with simultaneous resolution. By the versality
	property of Lemma~\ref{lem:versality}, if $P$ is not an isomorphism
	on a general fiber, the exceptional-curve family in the
	simultaneous resolution gives a horizontal component of
	$\Hilb(\cY/\Delta)$ containing $[F_c]$ for some $c\in C$.
	Contradiction. 
	
	\smallskip\noindent\emph{Case b: Every fiber of $E$ is a line pair.}
	 Every point of $C$ is a $cA_2$-singularity, and we view
	$\cY\rightarrow\bar{\cY}$ as a deformation of the
	$A_2$-singularity with simultaneous resolution. Its exceptional-cycle
	families corresponding to the three positive roots have special
	fibers
	\[
	F_1,\qquad F_2,\qquad F_1+F_2,
	\]
	with their reduced scheme structures.
	Consequently, if $P$ is not an isomorphism on a general fiber,
	there is a horizontal component of $\Hilb(\cY/\Delta)$ containing
	either a whole fiber or an individual fiber component.
	Both possibilities have been excluded. This proves $(3)$.
\end{proof}
	For a contraction $p$ with exceptional divisor $E$ as in Proposition~\ref{prop:root-faces-disappear-movable}(1) the associated contraction $P_E$ on the versal deformation is small. We show that its flop exists and under appropriate identifications, the induced map on $N^1(Y)$ is given by the reflection $\sigma_E$ defined in the introduction.
\begin{proposition}\label{prop:flopofroot}
	Let $Y$ be a smooth projective threefold with $H^2(Y, \cO_Y)=0$ and $\Def(Y)$ smooth. 
Let $p_E \colon Y \rightarrow \bar{Y}$ be a primitive $K$-trivial type III contraction with exceptional divisor $E$. Denote $q=(p_E)|_E \colon E \rightarrow C$ the induced conic bundle, $\ell_E \in N_1(Y)$ the class of a general fiber of $q$, and assume that $H^0(C, M_E) \neq 0$. 
	\begin{enumerate}
		\item 
		The extension $P_E \colon \cY \rightarrow \bar{\cY}$ over the miniversal deformation $S$ is a flopping contraction, and the flop $P_E^+ \colon \cY^+ \rightarrow \bar{\cY}$ over $S$ exists.
		\item 
		Denoting 
		\[
		\Phi_E \colon \cY \dashrightarrow \cY^+
		\]
		the flop, the induced rational map 
		\[
		\phi_E \colon Y \dashrightarrow Y^+
		\]
		on central fibers extends to an isomorphism, denoted $\bar{\phi}_E \colon Y \xrightarrow{\sim} Y^+$.
		In particular, $\cY^+ \rightarrow S$ is again a deformation of $Y$ (possibly after shrinking $S$).
		\item  The induced rational map $\cY_t \dashrightarrow \cY_t^+$ on a very general fiber is biregular if $q \colon E \rightarrow C$ is quasi-ruled and $M_E \cong \cO_C$.
		\item Identifying the central fibers $Y$ and $Y^+$ using $\bar{\phi}_E$, the isomorphism 
		\[
		\Phi^0_E \colon N^1(Y) \cong N^1(\cY/S) \cong N^1(\cY^+/S) \cong N^1(Y^+)
		\]
		is identified with the reflection 
		\[
		\sigma_E \in \Aut(N^1(Y)), \qquad \sigma_E(x)=x+(x, \ell_E)[E].
		\]
		\item There is an automorphism $\alpha$ of the etale germ $(S, 0)$ such that 
		$
		\alpha^*\cY^+ \cong \cY.
		$
		giving a commutative diagram
		\[
		\begin{tikzcd}
			\cY \arrow[r, dashed, "\Phi_E"] \arrow[d] & \cY \arrow[d] \\
			(S,0) \arrow[r, "\alpha"] & (S,0).
		\end{tikzcd}
		\]
		where we have suppressed the identification of $\cY^+$ and $\cY$. 
	\end{enumerate}
\end{proposition}
We emphasize that we are not claiming that $\Phi_E \colon \cY \dashrightarrow \cY^+$ is defined everywhere on $Y$. Only after restricting the birational map of total spaces to central fibers does it extend to an isomorphism. 
\begin{proof}
	\begin{enumerate}
		\item 
		The argument is very similar to the case of compact Calabi--Yaus~\cite[Proposition 4.11]{Lutz}. We only summarize the main steps:
		Replacing $S$ by an etale cover, we may assume that 
		\[
		N^1(\cY/S) \rightarrow N^1(Y)
		\]
		is an isomorphism.
		Let $H$ be a big and nef line bundle defining the contraction of $E$. Let $\cH$ and $\cE$ be line bundles extending $H$ and $\cO_Y(E)$ to $\cY$, and let $P_E \colon \cY \rightarrow \bar{\cY}$ be the contraction over $S$ defined by $\cH \in \Pic(\cY)$.
		By Proposition~\ref{prop:root-faces-disappear-movable}, $P_E$ is a small contraction. 
		To show that the flop exists, we show that $P_E$ is a flipping contraction for a suitable klt pair. \\
		Since the contraction $P_E$ defined by $\cH$ is birational on a very general fiber, $\cH$ is relatively big,
		and therefore 
		\[
		\cL=\cH+\epsilon \cE \in \Bbig(\cY/S)
		\]
		as well for $\epsilon$ a small positive rational number. By the relative Kodaira lemma, we may write 
		\[
		\cL \sim_{\RR, S} \mathcal{A}+\mathcal{D}
		\]
		where $\mathcal{A}$ is relatively ample and $\mathcal{D}$ is effective. Shrinking $S$ and replacing $\mathcal{A}$ with an $\RR$-linearly equivalent effective divisor, it follows that there is $\Delta \geq 0$ which is $\RR$-linearly equivalent to $\cL$. 
		
		Since $\cY/S$ is smooth, 
		the pair $(\cY,\delta \Delta)$ is klt for $0 < \delta \ll 1$. 
		We have 
		\[
		(K_{\cY/S}+\delta \Delta) \cdot \ell_E=(K_{\cY/S}+\delta(\cH+\epsilon \cE))\cdot \ell_E =0+\delta\epsilon E \cdot \ell_E<0.
		\]
		Since $\rho(\cY/\bar{\cY})=1$, it follows that 
		\[
		-(K_{\cY/S}+\delta \Delta)
		\]
		is $P_E$-ample, so that $P_E \colon \cY \rightarrow \bar{\cY}$ is a flipping contraction for the pair $(\cY, \delta\Delta)$.
		By \cite[Corollary 1.4.1]{BCHM}, the flip $P_E^+ \colon \cY^+ \rightarrow \bar{\cY}$ of $(\cY, \delta \Delta)$ exists. It is a flop since $\ell_E$ is $K$-trivial. 
		\item The contraction morphism $p_E \colon Y \rightarrow \bar{Y}$ is crepant by Proposition~\ref{prop:wilson-type-contraction-to-curve}, so the exact same argument as in \cite[Proposition 4.11]{Lutz} shows that $\phi_E$ extends to an isomorphism.

		\item 
		This is immediate from Proposition~\ref{prop:root-faces-disappear-nef}.
		\item This follows from \cite[Proposition 4.12]{Lutz}. The argument there only uses that $Y \rightarrow \bar{Y}$ contracts $E$ to a curve of $cDV$ singularities, which is true by Proposition~\ref{prop:wilson-type-contraction-to-curve}.
		\item In \cite{Lutz}, we deduced this from the infinitesimal Torelli theorem for Calabi--Yau manifolds, which is not available here. We instead give a more direct geometric argument. 
		Let
		\begin{equation}\label{eq:flop-diagram}
			\begin{tikzcd}
				\cY
				\arrow[rr, dashed, "\Phi_E"]
				\arrow[dr, "P_E"']
				&&
				\cY^+
				\arrow[dl, "P_E^+"]
				\\
				&
				\bar{\cY}
				&
			\end{tikzcd}
		\end{equation}
		be the flop of $P_E$ over $S$ constructed in part 1. Since $\cY/S$ is a miniversal deformation, there is a morphism $\alpha \colon S \rightarrow S$ of etale germs such that
		\[
		\alpha^*\cY \simeq \cY^+
		\]
		as deformations. 
		Pull back \eqref{eq:flop-diagram} via $\alpha$ to get 
		\begin{equation}\label{eq:flop-diagram2}
			\begin{tikzcd}
				\alpha^*\cY
				\arrow[rr, dashed, "\Phi_E"]
				\arrow[dr, "\alpha^*P_E"']
				&&
				\alpha^*\cY^+
				\arrow[dl, "\alpha^*P_E^+"]
				\\
				&
				\alpha^*\bar{\cY}
				&
			\end{tikzcd}
		\end{equation}
		We have to show that the diagram~\eqref{eq:flop-diagram2} of birational maps is still a flopping diagram. By construction of $P_E$ and $P_E^+$, the fibers of $\alpha^*P_E$ and $\alpha^*P_E^+$ have dimension $\leq 1$. It is therefore enough to show that the image of the exceptional locus of $\alpha^*P_E$ has codimension at least three. This will follow if we can show that $\alpha^*P_E$ and $P_E^+$ are the same morphism (up to an isomorphism of the target). 
		The birational morphism $\alpha^*P_E$ is defined by the line bundle $\alpha^*\cH$, which restricts on the central fiber to $H$. 
		On the other hand, the strict transform $\cH^+$ of $\cH$ on $\cY^+$ restricts on the central fiber to $\sigma_E(H)$ by part (4). Since $H \in \ell_E^\perp$, $\sigma_E(H)=H$. It follows that the morphisms $\alpha^*P_E$ and $P_E^+$ are defined by line bundles of the same numerical class and contract the same curves. They therefore define the same morphism, up to an isomorphism of the target. 
		
		$\alpha^*P_E^+$ must then be the flop of $P_E^+$
		and hence $\alpha^*\cY^+ \cong \cY$. Therefore 
		\[
		(\alpha \circ \alpha)^*\cY=\cY.
		\]
		We conclude that $d\alpha_0$ is an isomorphism, so that $\alpha$ is \'etale at the origin, and hence defines an automorphism of the etale germ $(S, 0)$.
	\end{enumerate}
\end{proof}
\section{Contractible models of $K$-trivial conic fibration surfaces}\label{sec:contractible-models}
So far, we have only studied conic fibration surfaces arising as the exceptional divisors of a contraction morphism. To understand the full movable cone, it is necessary to also look at the strict transforms of such surfaces on small modifications of $Y$. 

By definition, every $E \in \Delta_Y^{\mathrm{big}}$ admits a $K$-trivial conic model satisfying the nonvanishing condition, not necessarily a \emph{contractible} model. Our first result shows that a contractible $K$-trivial conic model 
\[
(\alpha \colon Y \dashrightarrow Z, q \colon E_Z \rightarrow C)
\]
satisfying the nonvanishing condition always exists. The conditions 
 $H^2(Y, \cO_Y)=0$ and $\Def(Y)$ is smooth used throughout Section 4 are invariant under crepant SQMs: indeed, the Hodge numbers are preserved by threefold flops, whereas the preservation of $\Def(Y)$ follows from \cite[11.10]{KollarMoriFlips}, see also \cite[Theorem 3.2.2]{KollarSurvey}. In particular, the theory developed in Section 4 applies to the strict transform $E_Z$.
 
\begin{proposition}\label{prop:existenceofcontractionofE}
	Let $Y$ be a smooth projective threefold satisfying $H^2(Y, \cO_Y)=0$ and $\Def(Y)$ smooth.
	\begin{itemize}
		
		\item Let $E \in \Delta_Y^{\mathrm{big}}$. Then $E$ has a contractible $K$-trivial conic model with nonvanishing obstruction group.
		\item Let $E \in \Delta_Y^{\mathrm{sm}}$, and let $q \colon E \rightarrow C$ be the defining quasi-ruling. Then $q$ is contractible. 
	\end{itemize}
\end{proposition}
\begingroup
\raggedbottom
\begin{proof}
	Let $E\in\Delta_Y^{\mathrm{big}}$. By definition, there exists a $K$-trivial conic model
	\[
	(\alpha\colon Y\dashrightarrow Z,\ q\colon E_Z\to C).
	\]
satisfying
\[
H^0(E,\omega_Z^{-1}|_{E}\otimes q^*\omega_C)\neq0.
	\]
To declutter notation, we write $E=E_Z$. Let $\ell_E$ be the class of a general fiber of $q$. 
	As usual, let 
	$
	\nu\colon \tilde{E}\to E
	$
	be the normalization, set $f\colon \tilde{E}\to E\hookrightarrow Z$, and
	write the Stein factorization of $q\circ\nu$ as
	\[
	\tilde{E}\xrightarrow{\tilde{q}}\tilde{C}\xrightarrow{\nu_C}C.
	\]
	\smallskip\noindent\emph{Step 1: There is an extremal ray
		$R\subset~\overline{\NE}(Z)$ such that $E\cdot R<0$ and $K_{Z}\cdot R=0.$}
	
	\smallskip\noindent\emph{Step 1a: We find an $E$-negative extremal ray $R$.}
	Choose an ample divisor $H$ and let $t>0$ be minimal such that
	$
	L:=E+tH
	$
	is nef. The face $L^\perp\cap\overline{\NE}(Z)$ contains an extremal
	ray $R$, and
	\[
	E\cdot R<0.
	\]
	\smallskip\noindent\emph{Step 1b: There is an extremal ray $\tilde{R} \in \overline{\NE}(\tilde{E})$ mapping to $R$.} 
	This step is similar to the proof of \cite[Claim 9.11]{Bhatt}.
	We first show that $R$ lies in the image of
	$f_* \colon \overline{\NE}(\tilde{E}) \rightarrow \overline{\NE}(Z)$.
	Choose a generator $r \in R$, and a sequence 
	of effective real one-cycles $Z_n$ such that $[Z_n] \to r$. Decompose each $Z_n$ as 
	\[
	Z_n=Z_n^E+W_n
	\]
	where $Z_n^E$ is the part of $Z_n$ supported on $E$, and $W_n$ the part not supported on $E$.
	The morphism $\nu\colon\tilde{E} \rightarrow E$ is finite and surjective, so we may choose effective real one-cycles $Z_n^{\tilde{E}}$ such that 
	\[
	f_*[Z_n^{\tilde{E}}]=[Z_n^E]
	\]
	Choose an ample $H$ on $Z$, since $f$ is finite, $f^*H$ is ample as well. The intersection of  $Z_n$ with $H$ is bounded, so the intersections of $\{[Z_n^{\tilde{E}}]\}$ with $f^*H$ resp. $\{[W_n]\}$ with $H$ are bounded as well. After passing to a subsequence, we may therefore assume that $[Z_n^{\tilde{E}}] \to a \in \overline{\NE}(\tilde{E})$, and $[W_n] \to b \in \overline{\NE}(Z)$, and hence $r=f_*a+b$. Since $E \cdot R<0$ and $E \cdot W_n \geq 0$, we must have $f_*(a) \neq 0$. Since $R$ is extremal, $f_*(a)$ spans $R$. 
	It follows that $f_*^{-1}(R) \cap \overline{\NE}(\tilde{E})$ is a nonempty face, and hence contains an extremal ray $\tilde{R}$, which maps to $R$ by construction. 
	
	\smallskip\noindent\emph{Step 1c: $\tilde{R} \subset \overline{\NE}(\tilde{E})$ is generated by the class of a $\tilde{q}$-vertical curve.} 
	Let
	$
	\cF:=\ker(T_{\tilde{E}}\to \tilde{q}^*T_{\tilde{C}})
	$
	be the rank-one foliation induced by $\tilde{q}$ (we refer to \cite{Cascini} for elementary definitions in the theory of foliations). There is an effective
	$\tilde{q}$-vertical divisor $R_{\tilde{q}}$ such that
	\[
	K_{\tilde{E}}=K_{\cF}+\tilde{q}^*K_{\tilde{C}}+R_{\tilde{q}}.
	\]
	where $K_\cF$ is the canonical divisor of the foliation~\cite[p.4-5]{Cascini}.
	Since $E$ is Gorenstein and satisfies the $S_2$-condition, we have by~\cite[Proposition~2.3]{ReidNonnormal},
	\[
	\nu^*K_{E}=K_{\tilde{E}}+B.
	\]
	where $B$ is the conductor divisor of the normalization. The horizontal part of $B$ is the double curve if the general fiber of $E$ is a line pair, and zero if the general fiber is a $\PP^1$.
	Our assumptions on $q$ show that 
	\begin{equation}\label{eq:vertical}
	-K_Z|_E+q^*K_C\sim V
	\end{equation}
	for $V$ a $q$-vertical effective divisor. 
	Adjunction therefore gives
	\[
	\begin{aligned}
		\nu^*(E|_{E})
		&\sim K_{\tilde{E}}+B+\nu^*(V)-\nu^*q^*K_C \\
		&=K_{\cF}+\tilde{q}^*K_{\tilde{C}}+R_{\tilde{q}}+B+\nu^*(V)-\nu^*q^*K_C
	\end{aligned}
	\]
	We have $\nu_C^*K_C\sim K_{\tilde{C}}-R_{\nu_C}$,
	where $R_{\nu_C}$ is the ramification divisor for $\nu_C$, and therefore obtain 
	\begin{equation}\label{eqref:foliation}
		\nu^*(E|_{E})\sim K_\cF+\Delta,
	\end{equation}
	where 
	\[
	\Delta=R_{\tilde{q}}+B+\nu^*(V)+\tilde{q}^*R_{\nu_C} \geq 0.
	\]
	Note that a curve $\Gamma \subset \tilde{E}$ is $\cF$-invariant if and only if $\Gamma$ is $\tilde{q}$-vertical. The foliated surface cone theorem of~\cite[Theorem~6.3]{Spicer} gives:
	\[
	\overline{\NE}(\tilde{E})
	=
	\overline{\NE}(\tilde{E})_{K_{\cF}+\Delta\geq0}
	+
	Z_{-\infty}
	+
	\sum_i\RR_{\geq0}[L_i],
	\]
	where the $L_i$ are $\tilde{q}$-vertical rational curves and
	$Z_{-\infty}$ is spanned by those components $D$ of $\Delta$ whose
	coefficient is greater than $\epsilon(D)$, with
	\[
	\epsilon(D)=
	\begin{cases}
		0,&D\text{ is $\tilde{q}$-vertical},\\
		1,&D\text{ is not $\tilde{q}$-vertical}.
	\end{cases}
	\]
	We claim that $Z_{-\infty}$ is spanned by $\tilde{q}$-vertical curves. Indeed, 
	all components of $R_{\tilde{q}}$, $\tilde{q}^*R_{\nu_C}$, $\nu^*V$ are
	$\tilde{q}$-vertical.
	Moreover, every horizontal component of $B$ has
	coefficient one and therefore does not contribute to $Z_{-\infty}$: since $q$ is a conic fibration, over a general point in $C$, the fiber is either a smooth $\PP^1$, in which case there is no
	horizontal conductor, or a pair of lines meeting in a node,
	whose conductor in the normalization is reduced. 
	
	By Equation~\eqref{eqref:foliation}
	\[
	(K_{\cF}+\Delta)\cdot \tilde{R}
	=
	E\cdot R<0,
	\]
	The foliated cone theorem shows that $\tilde{R}=\RR_{\geq 0}[\tilde{\Gamma}]$ for a $\tilde{q}$-vertical curve $\tilde{\Gamma}$. 
	
	\smallskip\noindent\emph{Step 1d: $R$ is $K$-trivial.}
	 $\tilde{R}$ is spanned by the class of the $\tilde{q}$-vertical curve $\tilde{\Gamma}$, so $R$ is spanned by the class of the $q$-vertical curve $\Gamma=\nu(\tilde{\Gamma})$. 
	Since $K_{Z}$ is $q$-trivial, we have
	$
	K_{Z}\cdot R=K_Z \cdot \Gamma=0.
	$
	
	\smallskip\noindent\emph{Step 2: Inductive step.}
	Let $R$ be the extremal ray constructed in Step 1. Choose
	$0<\epsilon\ll1$ such that $(Z,\epsilon E)$ is klt.
	Then
	\[
	(K_Z+\epsilon E)\cdot R<0,
	\]
	so $R$ admits a Mori contraction
	$
	p\colon Z\to\bar Z.
	$
	By Step 1, $R$ is generated by the class of a  $q$-vertical rational curve. We claim that \emph{every} curve contracted by $p$ is $q$-vertical. This is clear if $g(C)>0$. If $g(C)=0$, suppose $\Gamma$ maps to $C$ with degree $d>0$. Then using Equation~\eqref{eq:vertical}, we get 
	\[
	K_Z \cdot \Gamma \leq -2d+0<0,
	\]
	and therefore $[\Gamma] \notin R$.
	
	If $\ell_E\in R$, then $p$ contracts $E$ along its
	ruling to $C$ and we are done. Otherwise $p$ is small.
	Since $K_Z\cdot R=0$, the flip of $p$ is a flop
	\[
	Z\dashrightarrow Z^+,
	\qquad
	p^+\colon Z^+\to\bar Z,
	\]
	with $Z^+$ again smooth.
	Let $E^+$ be the strict transform of $E$, and put
	$\bar E:=p(E)$, $r:=p|_E$, and $r^+:=p^+|_{E^+}$.
	Since $-E$ is $p$-ample and $K_Z\equiv_p0$, relative
	Kawamata--Viehweg vanishing gives
	\[
	R^1p_*\cO_Z(-E)=0.
	\]
	Pushing forward the restriction sequence
	\[
	0\longrightarrow\cO_Z(-E)
	\longrightarrow\cO_Z
	\longrightarrow\cO_E
	\longrightarrow0
	\]
	therefore gives $r_*\cO_E=\cO_{\bar E}$.
	The flopping curves are $q$-vertical, so $q$ descends to
	$\bar q\colon\bar E\to C$. This gives a diagram
	\[
	\begin{tikzcd}[row sep=large, column sep=large]
		E
		\arrow[rr, dashed]
		\arrow[dr, "r"']
		\arrow[ddr, bend right=45, "q"']
		&&
		E^+
		\arrow[dl, "r^+"]
		\arrow[ddl, bend left=45, "q^+"]
		\\
		& \bar E \arrow[d, "\bar q"] &
		\\
		& C &
	\end{tikzcd}
	\]
	In particular, $q^+=\bar q\circ r^+\colon E^+\to C$
	is again a conic fibration with the same generic fiber.
	Since a flop is $K$-trivial, we have
	$K_{Z^+}|_{E^+}\equiv_{q^+}0$.
	Since $\omega_Z$ is $p$-trivial, it descends to $\bar Z$ by
	Lemma~\ref{lem:Every-Contraction-Is-Mori}.
	
	Set
	\[
	\cM:=\omega_{\bar{Z}}^{-1}|_{\bar E}\otimes\bar q^*\omega_C.
	\]
	We obtain an isomorphism
	\[
	\omega_Z^{-1}|_E\otimes q^*\omega_C\cong r^*\cM,
	\]
	and similarly on the other side of the flop. Since $r^+$ is dominant, the natural map $\cO_{\bar{E}} \rightarrow (r^+)_*\cO_{E^+}$ is injective. 
	The projection formula then gives
	\[
	\begin{aligned}
		0&\ne H^0(E,\omega_Z^{-1}|_E\otimes q^*\omega_C)\\
		&\cong H^0(\bar E,\cM)\\
		&\hookrightarrow H^0(E^+,\omega_{Z^+}^{-1}|_{E^+}
		\otimes(q^+)^*\omega_C).
	\end{aligned}
	\]
	Thus the nonvanishing condition is preserved. 
	
	\smallskip\noindent\emph{Step 3: Termination.}
	Rename $(Z^+,E^+,q^+)$ from Step 2 as $(Z_1,E_1,q_1)$.
	Repeating this construction gives a sequence of flips
	\[
	Z\dashrightarrow Z_1\dashrightarrow\cdots.
	\]
	for the pair $(Z, \epsilon E)$
	which terminates by termination of
	threefold flips. Since
	\[
	E_i\cdot\ell_{E_i}=-2
	\]
	at every stage, the process cannot terminate with $E_i$
	nef. Therefore at some stage the contraction is divisorial,
	and its extremal ray is the fiber class $\ell_{E_i}$.
	This proves that $E$ has a $K$-trivial contractible conic model.
	
	Finally, if $E\in\Delta_Y^{\mathrm{sm}}$, we may take $(\id_Y, q)$ as our $K$-trivial conic model. Since $q$ is quasi-ruled, every vertical curve is numerically proportional to $\ell_E$, so the first contraction already contracts $E$ along its
	ruling. Hence $q$ is contractible.
\end{proof}
\endgroup

In general, a rational surface $E \subset Y$ can of course admit several conic fibrations, which might be contractible on different birational models of $Y$. We find that the nonvanishing condition in Definition~\ref{def:deformation-root-classes} imposes enough rigidity for a uniqueness statement among $K$-trivial conic fibrations:
\begin{lemma}\label{lem:uniqueness-conic-fiber-class}
	Let $E\in\Delta_Y^{\mathrm{big}}$, and let
	\[
	(\alpha_i\colon Y\dashrightarrow Z_i,\;
	q_i\colon E_{Z_i}\to C_i),\qquad i=1,2,
	\]
	be two contractible $K$-trivial conic models of $E$ with nonvanishing obstruction group. 
	Denote by $\ell_i$ the class of a general fiber of $q_i$,
	and by $M_i\in\Pic(C_i)$ the associated obstruction bundle.
	
	\begin{itemize}
		\item Then $\beta=\alpha_2\circ\alpha_1^{-1}$ maps a general
	fiber of $q_1$ isomorphically onto a general fiber of $q_2$
	and induces an isomorphism
	$\bar\beta\colon C_1\to C_2$. Moreover,
	\[
	\beta_*(\ell_1)=\ell_2,
	\qquad
	\bar\beta^*M_2\cong M_1.
	\]
	\item If $E \in \Delta_Y^{\mathrm{big}}$ admits a contractible $K$-trivial conic fibration with nonvanishing obstruction group on $Y$, it is unique. In particular, every $E \in \Delta_Y^{\mathrm{sm}}$ admits a unique $K$-trivial quasi-ruling with nonvanishing obstruction group
	\end{itemize}
\end{lemma}
\begin{proof}
Note first that every $K_{Z_i}$-trivial rational curve
	$\Gamma\subset E_{Z_i}$ is $q_i$-vertical.
	Indeed, this is clear if $g(C_i)>0$. Otherwise,
	$H^0(C_i,M_i)\neq0$ implies
	$\deg(M_i\otimes\omega_{C_i}^{-1})\geq2$, so
	\[
	0=-K_{Z_i}\cdot\Gamma
	=q_i^*(M_i\otimes\omega_{C_i}^{-1})\cdot\Gamma
	\]
	forces $\Gamma$ to be vertical.
	
	Since $\beta$ factors as a composite of flops, its
	non-isomorphism locus consists of $K$-trivial rational
	curves. Those contained in $E_1$ are vertical, while
	the others meet $E_1$ in finitely many points.
	Thus $\beta$ is an isomorphism near a general fiber
	of $q_1$, and similarly for $\beta^{-1}$ and $q_2$.
	Therefore, $\beta$ identifies
	general fibers and induces an isomorphism
	$\bar\beta\colon C_1\to C_2$. In particular,
	$\beta_*(\ell_1)=\ell_2$.
	
	Identify the bases with $C$ via $\bar\beta$.
	Choose a common resolution
	$r\colon W\to Z_1$, $s\colon W\to Z_2$
	which is unchanged near general fibers, and let
	$\pi\colon E_W\to C$ be the induced morphism on the
	strict transform.
	Noting that $\omega_{Z_i}|_{E_{Z_i}}^{-1}=q_i^*(M_i\otimes\omega_{C_i}^{-1})$, crepancy of $\beta$ gives
	\[
	\pi^*(M_1\otimes\omega_C^{-1})
	\cong
	\pi^*(M_2\otimes\omega_C^{-1}).
	\]
	Tensoring with $\pi^*\omega_C$ and applying the projection formula therefore gives
	$M_1\cong M_2$, as required.
	
	For the second part, taking $\alpha_i=\id_Y$ gives the first claim. For the second claim, note that the defining ruling on $E \in \Delta_Y^{\mathrm{sm}}$ has nonvanishing obstruction group by definition, and is contractible by Proposition~\ref{prop:existenceofcontractionofE}
\end{proof}
Given this rigidity, we now define the fiber class associated to $E \in \Delta_Y^{\mathrm{big}}$.
\begin{definition}
	Let $E \in \Delta_Y^{\mathrm{big}}$, and let 		
	\[
	(\alpha \colon Y \dashrightarrow Z, \;q \colon E_Z \rightarrow C)
	\]
	be a contractible $K$-trivial conic model with nonvanishing obstruction group (which exists by Proposition~\ref{prop:existenceofcontractionofE}). We define 
	\[
	\ell_E=(\alpha^{-1})_*(\ell_{E_Z}), 
	\qquad \sigma_E(x)=x+(x, \ell_E)[E].
	\]
\end{definition}
By Lemma~\ref{lem:uniqueness-conic-fiber-class}, the definition of $\ell_E$, and hence also $\sigma_E$ is independent of the choice of $\alpha$ and intrinsic to the surface $E$. 
It follows that for an arbitrary crepant SQM $\alpha \colon Y \dashrightarrow Y'$, we have 
\begin{equation}\label{eq:reflection-compatible}
\sigma_E=\alpha^* \circ \sigma_{E_{Y'}} \circ (\alpha^*)^{-1}.
\end{equation}

We are now ready to extend Proposition~\ref{prop:root-faces-disappear-movable} to all deformation roots.
	\begin{proposition}\label{prop:injection-roots-faces}
		Let $Y$ be a smooth projective threefold with $\Def(Y)$ smooth and $H^2(Y,\cO_Y)=0$. Let $S$ (respectively $S'$) be the set of disappearing facets of $\overline{\Mov}(Y)$ (resp. $\Nef(Y)$). There is an injective map of sets 
		\[
		\Phi \colon \Delta_Y^{\mathrm{big}} \rightarrow S
		\]
		defined by sending $E$ to the facet of $\overline{\Mov}(Y)$ supported on $\ell_E^\perp$.
		
		Similarly, there is an injective map 
		\[
		\Phi' \colon \Delta_Y^{\mathrm{sm}} \rightarrow S',
		\]
	 defined by sending $E$ to the facet of $\Nef(Y)$ supported on $\ell_E^\perp$.
	\end{proposition}
	\begin{proof}
		Note first that the map is injective since distinct $E \in \Delta_Y^{\mathrm{big}}$ have nonproportional conic fiber classes. 
		
		By Proposition~\ref{prop:existenceofcontractionofE}, we may choose a contractible $K$-trivial conic model
		\[
		\alpha \colon Y \dashrightarrow Z, \qquad q\colon E_Z \rightarrow C.
		\]
		with nonvanishing obstruction group. Proposition~\ref{prop:root-faces-disappear-movable}
	 shows that the facet of $\overline{\Mov}(Z)$ supported on $\ell_{E_Z}^\perp$ disappears under deformation of $Z$. Therefore, $\ell_E^\perp$ supports a facet of $\overline{\Mov}(Y)$, and it remains to show that it disappears under deformation of $Y$. Indeed, the crepant SQM $\alpha$ is a composite of flops, and by~\cite[Theorem 11.10]{KollarMoriFlips}, these flops deform to give a deformation $\pi_Z\colon\cZ\to \Def(Z)$ of $Z$ and an identification $\Def(Y) \cong \Def(Z)$. In particular, $\alpha$ induces an SQM between the very general fibers, and therefore an identification of $\overline{\Mov}(Y^{\mathrm{gen}})$ with $\overline{\Mov}(Z^{\mathrm{gen}})$. 

If $E \in \Delta_Y^{\mathrm{sm}}$, then the quasi-ruling on $E$ is contractible by Proposition~\ref{prop:existenceofcontractionofE}, so that $\ell_E^\perp$ supports a facet of $\Nef(Y)$. By Proposition~\ref{prop:root-faces-disappear-nef}, this facet disappears.
	\end{proof}
		\section{Weyl groups and their geometric realizations}\label{sec:weyl}
	Throughout this section, $Y$ is a smooth projective threefold. 
	We will show that the pairs $(E, \ell_E)$, where $E \in \Delta_Y^{\mathrm{big}}$ and $\ell_E$ is its conic fiber class can be viewed as simple roots in a generalized sense. We make the following definition: 
	\begin{definition}\label{def:genroots}
		Let $V$ be a real vector space, and $V^*$ its dual. 
		A subset $\Delta \subset V$ together with an injection $\vee \colon \Delta \rightarrow V^*$ is called a set of generalized roots if 
		\[
		\alpha^\vee(\alpha)=-2
		\]
		for all $\alpha \in \Delta$, and for all distinct $\alpha, \beta \in \Delta$, we have
		\begin{enumerate}
			\item $\alpha^\vee(\beta) \in \ZZ_{\geq 0}$ 
			\item $\alpha^\vee(\beta) =0$ if and only if $\beta^\vee(\alpha)=0$
			\item $\{\alpha, \beta\}$ (respectively $\{\alpha^\vee, \beta^\vee\}$) is a linearly independent set in $V$ (respectively in $V^*$) 
		\end{enumerate}
	\end{definition}
	Let $\Delta$ be a set of generalized roots. 
							For each $\alpha\in \Delta$ we define the reflection
							\[
							\sigma_{\alpha} \colon V \rightarrow V, \quad x \mapsto x+\alpha^\vee(x) \alpha
							\]
			Let $W(\Delta)$ be the group generated by the $\sigma_\alpha$. Then $W(\Delta)$ is a Coxeter group (see \cite[Section 4]{Lutz}).
					For each $\alpha\in \Delta$, the inequality $\langle x, \alpha^\vee \rangle>0 $ defines an open half space $A_\alpha$. 
					We define the cone 
					\[
					\cC(\Delta)=\bigcap_{\alpha\in \Delta}A_\alpha \subset V
					\]
					and the Tits cone 
					\[
					U=\bigcup_{w \in W}w \cdot \bar{\cC}(\Delta).
					\]
					We have the following result, which is an adaptation of~\cite[V. 4.6 Proposition 6]{Bourbaki}:
					\begin{proposition}\label{cor:fundamentaldomain}\cite[Corollary 4.8]{Lutz}
							Suppose that $\cC(\Delta)$ is nonempty. Then the natural map $W(\Delta) \rightarrow \GL(V)$ is injective, and $\bar{\cC}(\Delta)$ is a fundamental domain for the action of $ W(\Delta)$ on $U$ .
						\end{proposition}
					In view of this, we call $\bar{\cC}(\Delta)$ the fundamental chamber. 
					Our next result shows that the sets $\Delta_Y^{\mathrm{big}}$ and $\Delta_Y^{\mathrm{sm}}$ form generalized roots:
	\begin{proposition}\label{prop:deformation-roots-form-a-root-system}
		Let $Y$ be a smooth projective threefold. Then $(N^1(Y, \RR), \Delta_Y^{\mathrm{big}})$ is a set of generalized roots, with the injection 
		\[
		\vee \colon \Delta_Y^{\mathrm{big}} \rightarrow N^1(Y, \RR)^* \cong N_1(Y, \RR)
		\]
		given by sending $E$ to its conic fiber class $\ell_E$. 
	\end{proposition}
	\begin{proof}
		We follow the argument in \cite[Proposition 4.9]{Lutz}.
		The equality
		\[
		E \cdot \ell_E=-2
		\]
		follows from Proposition~\ref{prop:wilson-type-contraction-to-curve}. 
		Suppose that $E_1, E_2 \in \Delta_Y^{\mathrm{big}}$ are distinct and write $\ell_i=\ell_{E_i}$. Since $E_1$ is irreducible, and curves of class $\ell_{2}$ move in a one-dimensional family, we must have 
		\[
		E_1 \cdot \ell_{2} \in \ZZ_{\geq 0}.
		\]
		In particular, this shows that $\vee$ is injective. 
		Linear independence of $\{E_1, E_2\}$ easily follows by intersecting with $\ell_{E_1}$ and an ample divisor. Linear independence of $\{\ell_{1}, \ell_{2}\}$ is analogous.
		
		It remains to prove that $E_1 \cdot \ell_2=0$ iff $E_2 \cdot \ell_1 = 0$.
		Let 
		\[
		(\alpha_i \colon Y \dashrightarrow Z_i, q_i \colon E_{i,Z_i} \rightarrow C_i), \qquad i=1,2
		\]  
		be $K$-trivial contractible conic models for $E_1$ and $E_2$, and denote $Z_i \rightarrow \bar{Z}_i$ the corresponding contraction.
		We write $E_{i, Z_j}$ for the strict transform of $E_i$ on $Z_j$ and $\ell_{i, Z_j}$ for the image of $\ell_i$ under the induced identification of curve classes. 
		Suppose without loss of generality that 
		\begin{equation}\label{eq:root-pairs}
			E_1 \cdot \ell_2 \neq 0, \qquad E_2 \cdot \ell_1=0
		\end{equation}
		Then $E_{1, Z_2} \cap E_{2, Z_2}$ must have a $q_2$-horizontal component, whereas every component of $E_{1, Z_1} \cap E_{2, Z_1}$ is $q_1$-vertical, and therefore rational.
		There are two cases to consider:
		\begin{itemize}
			\item $g(C_2)>0$. In this case $E_{1, Z_2} \cap E_{2, Z_2}$ has a component of positive genus. This component cannot be contracted by the SQM $\alpha_1 \circ \alpha_2^{-1} \colon Z_2 \dashrightarrow Z_1$, so that $E_{1, Z_1} \cap E_{2, Z_1}$ also has a component of positive genus, contradiction.
			\item $g(C_2)=0$. 
			Since $\alpha_i$ is crepant, we have that 
			\[
			E_1 \cdot E_2 \cdot K_Y=E_{1, Z_i} \cdot E_{2, Z_i} \cdot K_{Z_i}, \qquad i=1,2
			\]
			Computing this triple intersection on $Z_1$ and $Z_2$, it follows that 
			\[
			d_1(E_2 \cdot \ell_1)=d_2(E_1 \cdot \ell_2),
			\]
			where $d_i=\deg(-K_{\bar{Z}_i}|_{C_i})$. 
			Since $E_2 \in \Delta_Y^{\mathrm{big}}$ and $g(C_2)=0$, we must have $d_2 >0$. This contradicts Equation~\eqref{eq:root-pairs}.
		\end{itemize}
	\end{proof}
	To declutter notation, we will write
	\begin{equation*}
		W_Y^{big}=W(\Delta_Y^{\mathrm{big}}), \qquad
			\cC^{big}=\cC(\Delta_Y^{\mathrm{big}}),\qquad
			U^{big}=U(\Delta_Y^{\mathrm{big}})
	\end{equation*}
	and make similar definitions for $\Delta_Y^{\mathrm{sm}}$. 
	
	By Equation~\eqref{eq:reflection-compatible}, the group $W_Y^{big}$ is invariant under crepant SQMs. The smaller group $W_Y^{sm}$ is not, as we require its generators to correspond to quasi-ruled surfaces on $Y$.

		\begin{corollary}\label{thm:weylpreservescones}
			Let $Y$ be a smooth projective threefold with $\Def(Y)$ smooth and $H^2(Y, \cO_Y)=0$. $W_Y^{big}$ acts faithfully on $N^1(Y, \RR)$ and preserves $\overline{\Mov}(Y^{\mathrm{gen}})$, $\Eff(Y^{\mathrm{gen}})$ and  $\Bbig(Y^{\mathrm{gen}})$. 
			The action of $W_Y^{sm}$ additionally preserves $\Nef(Y^{\mathrm{gen}})$ and $\Amp(Y^{\mathrm{gen}})$.
		\end{corollary}
		\begin{proof}
			The fact that the action is faithful follows from the theory of Coxeter groups (see \cite[Section 4]{Lutz} for the argument). 
			
			Suppose first that $E \in \Delta_Y^{\mathrm{sm}}$. Then the quasi-ruling on $E$ is contractible, and  Proposition~\ref{prop:flopofroot} gives a diagram 
\[
\begin{tikzcd}
	\cY \arrow[r, "\Phi_E", dashed] \arrow[d, "\pi"'] &
	\cY \arrow[d, "\pi"] \\
	S \arrow[r, "\alpha"] &
	S .
\end{tikzcd}
\]
Since $\Phi_E$ restricts to an isomorphism on a very general fiber, we have that 
			\[
			(\Phi_{E,s})_*(\Nef(\cY_s))=\Nef(\cY_{\alpha(s)})
			\]
			for $s$ very general. Moreover, $\alpha$ is a local isomorphism in the \'etale topology and  therefore sends very general points to very general points. It follows that 
			\[
			(\Phi_{E,s})_*(\Nef(Y^{\mathrm{gen}}))=\Nef(Y^{\mathrm{gen}}).
			\]
			By Proposition~\ref{prop:flopofroot}, the reflection $\sigma_E \in \Aut(N^1(Y, \RR))$ is the composite of parallel transport and $\Phi_{E}$ and therefore preserves $\Nef(Y^{\mathrm{gen}})$, and its interior $\Amp(Y^{\mathrm{gen}})$.
			
			Suppose now that $E \in \Delta_Y^{\mathrm{big}}$. Choose a $K$-trivial contractible conic-model $(\alpha \colon Y \dashrightarrow Z, q \colon E_Z \rightarrow C)$. By Equation~\eqref{eq:reflection-compatible}, we have
			\[
			\sigma_E=\alpha^*\circ\sigma_{E_Z}\circ(\alpha^*)^{-1}.
			\]
			$\alpha$ is a composite of flops, so deforms to an SQM $\cY \dashrightarrow \cZ$ over $S$, and hence satisfies 
			\[
			\alpha^*(\overline{\Mov}(Z^{\mathrm{gen}}))=\overline{\Mov}(Y^{\mathrm{gen}}).
			\]
			Similarly, $\alpha^*$ preserves the effective and big cone of a very general fiber.
			For $\sigma_{E_Z}$, Proposition~\ref{prop:flopofroot} shows that $\Phi_{E_Z}$ restricts to a flop or an isomorphism on a very general fiber, so that we have 
			\[
		{\Phi_{E_Z}}_*(\overline{\Mov}(Z^{\mathrm{gen}}))=\overline{\Mov}(Z^{\mathrm{gen}}).
		\]
	It now follows that $\sigma_E$ preserves $\overline{\Mov}(Y^{\mathrm{gen}})$, and similarly, $\sigma_E$ preserves the effective cone of $Y^{\mathrm{gen}}$ and its interior $\Bbig(Y^{\mathrm{gen}})$.
		\end{proof}
		\section{The case of threefolds with nef anticanonical bundles}\label{sec:nef-anticanonical}
		The analysis in the previous section started with a facet of the nef cone for which there is an associated $K$-trivial birational contraction morphism. In general, we cannot globalize such a facet-by-facet argument to the whole nef cone: faces corresponding to $K$-positive contractions are much harder to classify, and even if a facet is $K$-nonpositive, it might not correspond to a contraction morphism. 
		It is therefore natural to restrict to threefolds with nef anticanonical bundle: this eliminates the possibility of $K$-positive faces, and furthermore, the basepointfree theorem shows that each facet of $\Nef(Y) \cap \Bbig(Y)$ actually corresponds to a contraction:
		\begin{proposition}[Nef-chamber decomposition of the movable cone]\label{prop:movable-nef-chamber-decomposition}
			Suppose that $Y$ is a smooth projective threefold with $-K_Y$ nef.
			The effective movable cone admits a decomposition 
			\begin{equation}\label{eq:movable-cone-decomposition}
			{\Mov}^e(Y)=\bigcup_{\alpha \colon Y\dashrightarrow Z} \alpha^*\Nef^e(Z)
			\end{equation}
			where the union is over all SQMs (small $\QQ$-factorial modifications) of $Y$. 
			Each SQM factors as a composite of flops. Moreover, the cones
			$\Bbig(Y) \cap \Nef(Y)$ and $\Bbig(Y) \cap \overline{\Mov}(Y)$ are locally rational polyhedral, and every facet $F$ of $\Bbig(Y) \cap \Nef(Y)$ corresponds to a contraction morphism $p_F \colon Y \rightarrow \bar{Y}$ of relative Picard rank $1$.
		\end{proposition}
		\begin{proof}
			The first two assertions are~\cite[Proposition 1.3]{StengerXie}. The statement about contractions follows from the base-point-free theorem~\cite[Theorem 3.3]{KollarMori}, using that $-K_Y$ is nef. For the local polyhedrality, we use a similar argument as in \cite[Theorem 2.6]{KawamataCY}. 
			
			Let $x \in  \Nef(Y) \cap \Bbig(Y)$. We first claim that there are finitely many extremal rays of $\overline{\NE}(Y)$ on which $x$ is zero. 
			Let $H$ be ample. For $\epsilon$ sufficiently small, $y=x-\epsilon H$ is big, so we may write $y \sim_\RR E$ for an effective $\RR$-divisor $E$. Choose $\delta$ small enough so that $(Y, \delta E)$ is klt. 
			If $x \cdot R=0$ then we have
			\[
			(K_Y+\delta E+\frac{\delta \epsilon}{2}H) \cdot R =(K_Y+\delta x -\frac{\delta \epsilon}{2}H) \cdot R<0,
			\]
			 where we have used that $-K_Y$ is nef and $H$ is ample.
			The cone theorem implies that there are only finitely many such extremal rays $R_1, \dots R_k$. 
		Locally around $x$, $\Nef(Y)$ is defined by the intersection of the half spaces $R_i^{\geq0}$ for some subset $I \subset \{1, \dots, k\}$ and is therefore locally rational polyhedral.
	
	Suppose now that $x \in \Bbig(Y) \cap \overline{\Mov}(Y)$. We first show that $x$ is contained in finitely many of the cones $\alpha^*(\Nef^e(Z) \cap \Bbig(Z))$. Fix one such cone, and let $Z \rightarrow \bar{Z}$ be the associated birational contraction morphism. All other cones correspond to 
	birational contractions $Z_i \rightarrow \bar{Z}$ where $Z \dashrightarrow Z_i$ is a small modification. Since $-K_Y$ is nef, the proof of \cite[Proposition 6.40, Theorem 6.42]{KollarMori} shows that there are only finitely many such $Z_i$, $i=1 \dots m$. By the first part, each $\alpha^*(\Nef^e(Z) \cap \Bbig(Z))$ is locally rational polyhedral, so we can choose a small neighbourhood $U \ni x$ with the property that every facet of $\alpha_i^*(\Nef^e(Z_i)) \cap U$ contains $x$. 
	We claim that
	\[
	\overline{\Mov}(Y) \cap U=\bigcup_{i=1}^m \alpha_i^*(\Nef^e(Z_i)) \cap U.
	\]
	Indeed, if there was a nef chamber not containing $x$ but meeting $U$, there would also be such a chamber which meets one of the $\alpha_i^*(\Nef^e(Z_i))$ along a common facet $F$. But by construction of $U$, $x \in F$.
	
	It follows that $\Bbig(Y) \cap \overline{\Mov}(Y)$ is locally rational polyhedral, as required.
		\end{proof}
		
		\begin{proposition}[Divisorial half-space description of the movable cone]\label{prop:movable-cone-divisorial-halfspaces}
			Let $Y$ be a smooth projective threefold with $-K_Y$ nef. Then $\overline{\Mov}(Y)$ is cut out inside $\overline{\Eff}(Y)$ by rational half-spaces corresponding to birational divisorial contractions $Y \dashrightarrow Z \rightarrow \bar{Z}$. 
		\end{proposition}
		\begin{proof}
			Since $-K_Y$ is nef, every facet $F$ of $\overline{\Mov}(Y) \cap \Bbig(Y)$ defines a primitive birational divisorial contraction $Y \dashrightarrow Z \rightarrow \bar{Z}$. We denote $R_F \in N_1(Y)$ the extremal ray contracted by $Z \rightarrow \bar{Z}$ under the identification $N_1(Y) \cong N_1(Z)$. Note that $R_F$ is independent of the choice of the marked minimal model $Z$. 
			We claim that 
			\[
			\overline{\Mov}(Y)=\overline{\Eff}(Y) \cap \bigcap_F R_F^{\geq 0}.
			\]
			Since $Z \rightarrow \bar{Z}$ is a divisorial contraction, curves with class in $R_F$ 
			sweep out a divisor on $Y$, so that every movable divisor is nonnegative on $R_F$. This gives the inclusion $\subset$.
			For the other inclusion, suppose that $x \in \overline{\Eff}(Y) \cap \bigcap_F R_F^{\geq 0}$. Let 
			$H$ be an ample divisor and set
			\[
			c=\min\{\lambda \in \RR \mid x+\lambda H \in \overline{\Mov}(Y)\}.
			\]
			Assume for a contradiction that $c>0$. The class $x+\epsilon H$ is big for all $\epsilon>0$, so in particular, $x+cH$ is contained in the boundary of $\overline{\Mov}(Y) \cap \Bbig(Y)$. Therefore $x+cH$ is contained in $R_F^\perp$ for some $F$. Since $H$ is ample, this implies that $x \cdot R_F<0$, contradiction. It follows that $c \leq 0$, and therefore $x\in \overline{\Mov}(Y)$.
		\end{proof}
		\begin{lemma}\label{lem:movable-cone-inclusions}
			Let $Y$ be a smooth projective threefold with $-K_Y$ nef. We have the following inclusions:
			\begin{enumerate}
				\item $\overline{\Mov}(Y) \subset \overline{\Mov}(Y^{\mathrm{gen}})$. 
				\item $\overline{\Mov}(Y) \cap \Bbig(Y) = \overline{\Mov}(Y) \cap \Bbig(Y^{\mathrm{gen}})$
			\end{enumerate}
		\end{lemma}
		\begin{proof}
			See \cite[Lemma 3.3]{Lutz}. The proof given there only uses the decomposition of the movable cone into nef cones of its marked minimal models, together with the fact that all SQMs are composites of flops. Both of these are satisfied if $-K_Y$ is nef, by Proposition~\ref{prop:movable-nef-chamber-decomposition}.
		\end{proof}
		\begin{proposition}\label{prop:bijection-roots-faces}
			Let $Y$ be a smooth projective threefold with $\Def(Y)$ smooth and $H^2(Y, \cO_Y)=0$. Suppose further that $-K_Y$ is nef. 
			The injections 
			\[
			\Phi \colon \Delta_Y^{\mathrm{big}} \rightarrow S, \qquad \Phi' \colon \Delta_Y^{\mathrm{sm}} \rightarrow S'
			\]
			from Proposition~\ref{prop:injection-roots-faces} are bijective. 
		\end{proposition}
		\begin{proof}
			Let $F$ be a disappearing facet of $\Nef(Y)$.
			If $F \cap \Bbig(Y)$ is empty, then $F$ persists under deformation, this follows from the inclusion of cones in Lemma~\ref{lem:ConeInclusions}. 
			We may therefore assume that $F$ meets $\Bbig(Y)$. 
			Since $-K_Y$ is nef, $F$ is $K$-nonpositive and corresponds to a primitive birational contraction $p_F$. Proposition~\ref{prop:root-faces-disappear-nef} shows that the exceptional divisor $E$ satisfies $E \in \Delta_Y^{\mathrm{sm}}$, and its defining quasi-ruling has class $\ell_E$. It follows that $\Phi'(E)=F$.

			Suppose now that $F$ is a disappearing facet of $\overline{\Mov}(Y)$. As before, using Lemma~\ref{lem:movable-cone-inclusions}, we may assume that $F$ meets $\Bbig(Y)$.  Since $-K_Y$ is nef, by Proposition~\ref{prop:movable-nef-chamber-decomposition}, we may choose a crepant SQM \(\alpha\colon Y\dashrightarrow Z\) such that $F$ contains a facet of $\alpha^*(\Nef(Z) \cap \Bbig(Z))$. Since $-K_Z$ is nef, this facet corresponds to a $K$-nonpositive primitive divisorial contraction $p \colon Z \rightarrow \bar{Z}$. 
Arguing as in Proposition~\ref{prop:injection-roots-faces}, we see that since $F$ disappears as a facet of $\overline{\Mov}(Y)$, $(\alpha^*)^{-1}(F)$ disappears as a facet of $\overline{\Mov}(Z)$. 
			Proposition~\ref{prop:root-faces-disappear-movable} now shows that $p$ is a $K$-trivial type III contraction with exceptional divisor $E$, and $H^0(M_E) \neq 0$. By definition, this means that $E \in \Delta_Z^{\mathrm{big}}$ and $(\alpha^*)^{-1}F$ is supported on $\ell_{E}^\perp$. Therefore $F=\Phi(E_Y)$, where $E_Y \in \Delta_Y^{\mathrm{big}}$ is the strict transform of $E$.
		\end{proof}
		\begin{proposition}\label{prop:movableisintersection}
			Let $Y$ be a smooth projective threefold with $\Def(Y)$ smooth and $H^2(Y, \cO_Y)=0$, and suppose that $-K_Y$ is nef. Then
			\begin{align}
				\Nef(Y)=\bar{\cC}^{sm} \cap \Nef(Y^{\mathrm{gen}}).	\label{eq:nef}\\
				\overline{\Mov}(Y)=\bar{\cC}^{big} \cap  \overline{\Mov}(Y^{\mathrm{gen}}). \label{eq:mov}
			\end{align}
		\end{proposition}
		\begin{proof} This is very similar to \cite[Proposition 3.10]{Lutz}
			Since the classes $\ell_E$ for $E \in \Delta_Y^{\mathrm{sm}}$ are effective, we have an inclusion 
			\[
			\Nef(Y) \subset \bar{\cC}^{sm} \cap \Nef(Y^{\mathrm{gen}}).
			\]
			For the reverse inclusion, let $x \in \bar{\cC}^{sm} \cap \Nef(Y^{\mathrm{gen}})$. Let $H$ be a generic ample divisor on $Y$, and consider the line $L: \{x+ tH \mid t \in \RR \}$ in the vector space $N^1(Y, \RR)$. $L$ must meet $\partial \Nef(Y)$ transversely, and we set 
			\[
			t_0=\inf\{t \in \RR \mid x+tH \in \Nef(Y) \}
			\]
			and $y=x+t_0H \in \partial  \Nef(Y) $.
			We claim that $t_0 \leq 0$, which will show that $x \in \Nef(Y)$ and complete the proof.	
			Suppose first that $y \notin \Bbig(Y)$.
			Then $y \notin \Bbig(Y^{\mathrm{gen}})$ by Lemma~\ref{lem:ConeInclusions}, so $y \in \partial \Nef(Y^{\mathrm{gen}})$, so $t_0 \leq 0$. It follows that $x\in \Nef(Y)$. 
			
			Suppose now that $y \in \Bbig(Y)$. By Proposition~\ref{prop:movable-nef-chamber-decomposition}, $\Nef(Y) \cap \Bbig(Y)$ is locally rational polyhedral in $\Bbig(Y)$.
			Since $H$ was chosen generically, we may assume that $y$ is contained in the relative interior of one facet $F$ of $\Nef(Y) \cap \Bbig(Y)$. 
			
			Suppose first that $F$ persists under deformation. Then $y \in \partial \Nef(Y^{\mathrm{gen}})$. We therefore have $t_0 \leq 0$, and as before, $x \in \Nef(Y)$.
			
			Suppose now that $F$ disappears under deformation. By Proposition~\ref{prop:bijection-roots-faces}, we have that $F \subset \ell_E^\perp$ for some $E \in \Delta_Y^{\mathrm{sm}}$. 
			We have
			\[
			0=y \cdot \ell_E=(x \cdot \ell_E) +t_0(H \cdot \ell_E)
			\]
			Since $H$ is ample, $H \cdot \ell_E>0$, and $x \cdot \ell_E \geq 0$ since $x \in \bar{\cC}^{sm}$. It follows that 
			\[
			t_0=-\frac{x\cdot \ell_E}{H \cdot \ell_E} \leq 0,
			\]
			and as before, $x \in \Nef(Y)$.
			
			The argument for the movable cone is exactly analogous.
		\end{proof}
	We have a fundamental chamber result for the action of the Weyl group.
					\begin{theorem}\label{thm:MMP}
						Suppose that $Y$ is a smooth projective threefold with $\Def(Y)$ smooth, $H^2(Y, \cO_Y)=0$, and $-K_Y$ nef. 
						Let $\pi \colon \cY \rightarrow S$ be a miniversal deformation of $Y$.
						\begin{enumerate}
							\item Each $D \in \Nef(Y^{\mathrm{gen}}) \cap \Bbig(Y^{\mathrm{gen}})$ is a $W_Y^{sm}$-translate of an element in $\Nef(Y) \cap \Bbig(Y)$.
							\item Each $D \in \overline{\Mov}(Y^{\mathrm{gen}}) \cap \Bbig(Y^{\mathrm{gen}})$ is a $W_Y^{big}$-translate of an element in $\overline{\Mov}(Y) \cap \Bbig(Y)$.
						\end{enumerate}
					\end{theorem}
					\begin{proof}
					Let $D \in \Bbig(Y^{\mathrm{gen}})$.
						Choose an $\RR$-divisor $\cD$ on $\cY$ whose class maps to $D$ under the isomorphism 
						\[
						N^1(\cY/S) \cong N^1(Y).
						\]
						By assumption $[\cD|_{\cY_s}] \in \Bbig(\cY_s)$ for $s$ very general, so we have that $[\cD] \in \Bbig(\cY/S)$.
						By the relative Kodaira lemma, we may write 
						\[
						\cD \sim_{\RR, S} \mathcal{A}+\mathcal{G}
						\]
						where $\mathcal{A}$ is a relatively ample $\QQ$-divisor and $\mathcal{G}$ is an effective $\RR$-divisor.
						Choose $0<\delta\ll1$ rational such that
						$(\cY,\delta\mathcal{G})$ is klt.
						
						Since $-K_Y$ is nef, the divisor
						$
						-K_Y+\delta\cA|_Y
						$
						is ample. Since ampleness is an open condition, we have that
						\[
						\cL:=-K_{\cY/S}+\delta\cA \in \Pic(\cY)_\QQ
						\]
						is relatively ample after shrinking $S$. Choosing $m\gg0$ sufficiently divisible and a general
						$H\in |m\cL|$, set $\Gamma=\frac1mH$. By \cite[Lemma 5.17]{KollarMori}, the pair
						$(\cY,\Gamma+\delta\mathcal{G})$ is klt. Moreover,
						\[
						K_{\cY}+\Gamma+\delta\mathcal{G}
						\sim_{\RR,S}\delta\cD.
						\]
					Since	$\Gamma \in \Amp(\cY/S)$ and $\delta \mathcal{G} \in \Eff(\cY)$, we have 
						\[
						\Gamma+\delta \mathcal{G} \in \Bbig(\cY/S).
						\]
						By \cite[Corollary 1.4.2]{BCHM}, we may run a
						$(K_{\cY/S}+\Gamma+\delta\mathcal{G})$-MMP over $S$, which is precisely a
						$\cD$-MMP. After shrinking $S$, we may assume that each step of the MMP meets the central fiber. 
						
						The remainder of the argument is identical to~\cite[Theorem 4.15]{Lutz}, we sketch it here for the convenience of the reader. 
						
						Suppose first that $D \in \Nef(Y^{\mathrm{gen}})$, and let $P \colon \cY \rightarrow \bar{\cY}$ be the first contraction in the MMP, and let $p \colon Y \rightarrow \bar{Y}$ be its restriction to the central fiber. The contraction corresponds to a facet $F$ of $\Nef(Y)$. Denoting $R_p$ the contracted extremal ray, we have 
						\[
						D \cdot R_p<0, \qquad D \in  \Nef(Y^{\mathrm{gen}}).
						\]
						Hence the facet $F$ must disappear under deformation. Since $-K_Y$ is nef,  Proposition~\ref{prop:bijection-roots-faces} implies that $F \subset \ell_E^{\perp}$ for some $E \in \Delta_Y^{\mathrm{sm}}$, so that $p$ is the contraction of $E$ along its (quasi)-ruling, and the first step in the MMP is the flop of $E$ constructed in Proposition~\ref{prop:flopofroot}.
						We now repeat the same argument. After finitely many steps the MMP terminates, and we have that 
						\[
						[\cD^+|_Y] \in \Nef(Y) \cap \Bbig(Y), 
						\]
						where $\cD^+$ is the strict transform of $\cD$ and we have used Lemma~\ref{lem:ConeInclusions}(4). 
						By Proposition~\ref{prop:flopofroot}, the resulting automorphism of $N^1$ is given by the product of the root reflections $\sigma_E$, and therefore an element of $W_Y^{sm}$.
						
						Suppose now that $D \in \overline{\Mov}(Y^{\mathrm{gen}})$. We keep the same notation.
						If $F$ is not contained in a facet of $\overline{\Mov}(Y)$, then the induced contraction $p$ on central fibers is small. Since $-K_Y$ is nef, and $Y$ is smooth, $p$ must be a flopping contraction, and (after shrinking $S$), $P$ is just the deformation of $p$ given by \cite[11.10]{KollarMoriFlips}.
						
						If $F$ is contained in a facet of $\overline{\Mov}(Y)$, then $F\subset\ell_E^\perp$ for some $E \in \Delta_Y^{\mathrm{big}}$ by Proposition~\ref{prop:bijection-roots-faces}, and $p$ is the contraction of $E$ constructed in Proposition~\ref{prop:existenceofcontractionofE}.
						
						The MMP is therefore a sequence of flops, each of which is either a deformation of a flop of the current central fiber, or a flop constructed in Proposition~\ref{prop:flopofroot} associated to a deformation root on that fiber.
				Let $\alpha \colon Y \dashrightarrow Z$ be the induced SQM on central fibers. 
				After the MMP has terminated, the strict transform of $\cD$ restricts to a class in $\Nef(Z)\cap\Bbig(Z)$ on the final central fiber. Further composing with $\alpha^{-1} \colon Z \dashrightarrow Y$, the strict transform is in $\overline{\Mov}(Y) \cap \Bbig(Y)$. 
			On $N^1$, the contribution of the flops making up $\alpha$ cancels out by repeatedly applying Equation~\eqref{eq:reflection-compatible}. Hence the resulting automorphism of $N^1$ is an element of $W_Y^{big}$.
					\end{proof}
We are now ready to give the proof of Theorem~\ref{thm:withnef}.
					\begin{proof}[Proof of Theorem~\ref{thm:withnef}] With all our preliminary work, the proof is now similar to the proof of~\cite[Theorem 1.4]{Lutz}.
						By Proposition~\ref{prop:deformation-roots-form-a-root-system}, the set $\Delta_Y^{\mathrm{sm}}$ forms a set of generalized simple roots. The fundamental chamber is nonempty, since it contains the ample cone of $Y$. By \cite[Corollary 4.8]{Lutz}, 
						the fundamental chamber $\bar{\cC}^{sm}$ is a fundamental domain for the action of $W_Y^{sm}$ on the Tits cone $U^{sm}$. By Corollary~\ref{thm:weylpreservescones}, the action of $W_Y^{sm}$ preserves $\Nef(Y^{\mathrm{gen}})$, and it follows that $W_Y^{sm}$ acts on $U^{sm} \cap \Nef(Y^{\mathrm{gen}})$ with fundamental domain 
						\[
						\bar{\cC}^{sm} \cap \Nef(Y^{\mathrm{gen}})=\Nef(Y)
						\]
						where the last equality is~\eqref{eq:nef}.
						By Theorem~\ref{thm:MMP}, every element of $\Nef(Y^{\mathrm{gen}}) \cap \Bbig(Y^{\mathrm{gen}})$ is a $W_Y^{sm}$-translate of $\Nef(Y) \cap \Bbig(Y)$, so that
						\[
						\Nef(Y^{\mathrm{gen}}) \cap \Bbig(Y^{\mathrm{gen}}) \subset U^{sm}
						\]
						This shows that 
						$W_Y^{sm}$ acts on $\Nef(Y^{\mathrm{gen}}) \cap \Bbig(Y^{\mathrm{gen}})$ with fundamental domain 
						\[
						\Nef(Y) \cap \Bbig(Y^{\mathrm{gen}})=\Nef(Y) \cap \Bbig(Y),
						\]
						as required (we have used Lemma~\ref{lem:ConeInclusions}(4) for the last equality). 
						The proof of the second part is exactly analogous, using Lemma~\ref{lem:movable-cone-inclusions} instead.
					\end{proof}
					\section{Blowups of $\PP^3$}\label{sec:examples}
					
					\subsection{The Mukai lattice}
					Let $Y_S$ be obtained by successively blowing up an ordered collection
					$S=(p_1,\dots,p_n)$ of points of $\PP^3$, where we allow infinitely near
					points. Let $E_i$ be the pullback of the class of the $i$-th exceptional
					divisor, and let $H$ be the pullback of the hyperplane class. Then
					$H,E_1,\dots,E_n$ give a $\ZZ$-basis for
					$\Pic(Y_S)\cong N^1(Y_S,\ZZ)$, and
					\[
					-K_{Y_S}=4H-2\sum_{i=1}^n E_i.
					\]
					Let $h$ be the class of the strict transform of a general line in $\PP^3$,
					and let $\ell_i$ be the class of a line in the $i$-th exceptional divisor.
					We define
					\begin{align*}
						\alpha_0&=H-E_1-E_2-E_3-E_4,
						&\alpha_i&=E_i-E_{i+1}\qquad i=1,\dots,n-1,\\
						\alpha_0^\vee&=2h-\ell_1-\ell_2-\ell_3-\ell_4,
						&\alpha_i^\vee&=\ell_i-\ell_{i+1}\qquad i=1,\dots,n-1.
					\end{align*}
					
					Following Mukai~\cite{Mukai}, consider the bilinear form
					\[
					D_1\star D_2=D_1\cdot D_2\cdot\left(-\frac{1}{2}K_{Y_S}\right).
					\]
					A direct calculation gives
					$
					\alpha_i\star\alpha_i=-2$. 
					Denoting $T_{p,q,r}$ the Dynkin diagram in Figure~\ref{fig}, we have that for $i\neq j$,
					\[
					\alpha_i\star\alpha_j=
					\begin{cases}
						1 & \text{if the corresponding vertices are joined in }T_{2,4,n-4},\\
						0 & \text{otherwise.}
					\end{cases}
					\]
					The $\alpha_i$ form a $\ZZ$-basis of $K_{Y_S}^\perp$, so that $K_{Y_S}^\perp$ is the lattice associated to the Coxeter diagram
					$T_{2,4,n-4}$. 
					We write
					\[
					s_{\alpha_i}(x)=x+(x\star\alpha_i)\alpha_i
					\]
					for the corresponding root reflection.
					
					\subsection{A special infinitely near configuration}
					If the points of $S$ are in general position, then the $\alpha_i$ are not
					represented by effective divisors. We show that there is a suitable infinitely near
					configuration such that all of them become effective.
					\begin{proposition}\label{prop:Ye}
						For each $n \geq 4$, there exists an ordered collection $S=(p_1,\dots,p_n)$
						of possibly infinitely near points of $\PP^3$, such that, denoting $Y_e=Y_S$ the resulting blowup:
						\begin{itemize}
							\item each $\alpha_i$ is represented by an irreducible effective divisor
							$D_i$ on $Y_e$;
							\item each $D_i\in\Delta_{Y_e}^{big}$, and the associated fiber class is
							\[
							\ell_{D_i}=\alpha_i^\vee;
							\]
							\item $D_i\in\Delta_{Y_e}^{sm}$ for $i>0$, whereas
							\[
							D_0\in\Delta_{Y_e}^{big}\setminus\Delta_{Y_e}^{sm};
							\]
							\item if $n\leq 8$, then $-K_{Y_e}$ is nef. If $n>8$, then
							$-K_{Y_e}$ has negative degree along a unique curve.
						\end{itemize}
					\end{proposition}
					
					\begin{proof}
						Choose an elliptic curve $(C,P)$ and embed it into $\PP^3$ using the
						linear system $\cO_C(4P)$. We define $Y_e$ by blowing up successively
						along the strict transform of $C$: set $p_1=P$, and, after blowing up
						$p_i$, let $p_{i+1}$ be the intersection of the new exceptional divisor
						with the strict transform of $C$.
						
						There is a plane $H_0\subset\PP^3$ such that
						$
						H_0|_C=4P.
						$
						Hence the first four blowup centers lie on the successive strict
						transforms of $H_0$, whereas starting from the fifth, they do not. The strict transform
						$D_0$ of $H_0$ therefore has class
						\[
						[D_0]=\alpha_0.
						\]
						The pencil of conics on $H_0$ through the length-four cluster
						$p_1,\dots,p_4$ induces a conic fibration
						$
						q_0 \colon D_0\longrightarrow\PP^1
						$
						with fiber class
						\[
						\ell_{D_0}=2h-\ell_1-\ell_2-\ell_3-\ell_4=\alpha_0^\vee.
						\]
						For $i>0$, let $D_i$ be the strict transform of the $i$-th exceptional
						divisor after the blowup at $p_{i+1}$. Then $D_i\cong\FF_1$ and
						\[
						[D_i]=E_i-E_{i+1}=\alpha_i.
						\]
						Denoting $q_i \colon D_i \rightarrow \PP^1$ the ruling, we see that 
						its fiber class is 
						$
						\ell_{D_i}=\alpha_i^\vee.
						$
						In each case the fiber class is $K_{Y_e}$-trivial.
						
						The elliptic curve $C$ is the complete intersection of a pencil of
						quadrics. The strict transforms of these quadrics form a pencil contained
						in
						$
						\left|-\frac{1}{2}K_{Y_e}\right|
						$
						whose base locus is the strict transform $\widetilde C$ of $C$. We have
						\[
						\left(-\frac{1}{2}K_{Y_e}\right)\cdot\widetilde C=8-n.
						\]
						If $\Gamma\neq\widetilde C$ is an irreducible curve, a general member of the
						pencil does not contain $\Gamma$, and hence
						$
						-K_{Y_e}\cdot\Gamma\geq0.
						$
						It follows that $-K_{Y_e}$ is nef for $n\leq8$. If $n>8$, then
						$\widetilde C$ is the unique curve on which $-K_{Y_e}$ has negative
						degree.
						
						Finally, a direct intersection calculation gives
						$
						D_i\cdot\left(-\frac{1}{2}K_{Y_e}\right)=\ell_{D_i}
						$
						as curve classes. In particular, 
						\[
						\omega_{Y_e}^{-1}|_{D_i}\cong q_i^*(\cO_{\PP^1}(2))
						\]
						and therefore, $M_{D_i}\cong\cO_{\PP^1}$. Hence $D_i\in\Delta_{Y_e}^{big}$. For $i>0$ the
						conic fibration on $D_i\cong\FF_1$ is a $\PP^1$-bundle, so
						$D_i\in\Delta_{Y_e}^{sm}$. We have $\rho(D_0)=5$, so $D_0$ does not admit a quasi-ruling (for instance, the conic fibration on $D_0$ constructed above has singular
						fibers), and therefore
						\[
						D_0\in\Delta_{Y_e}^{big}\setminus\Delta_{Y_e}^{sm}.
						\]
					\end{proof}
We find that the Mukai reflection in $\alpha_i$ coincides with the reflection $\sigma_{D_i}$:
					\begin{proposition}\label{prop:identification-reflections}
						For every $i$ we have
						$
						s_{\alpha_i}=\sigma_{D_i}.
						$
						Consequently,
						\[
						W(A_{n-1})\subset W_{Y_e}^{sm},
						\qquad
						W(T_{2,4,n-4})\subset W_{Y_e}^{big}.
						\]
					\end{proposition}
					\begin{proof}
						By Proposition~\ref{prop:Ye}, $[D_i]=\alpha_i$ and
						$
						D_i\cdot\left(-\frac{1}{2}K_{Y_e}\right)=\ell_{D_i}.
						$
						Hence 
						\[
						s_{\alpha_i}(x)
						=x+(x\star\alpha_i)\alpha_i
						=x+(x\cdot\ell_{D_i})[D_i]
						=\sigma_{D_i}(x).
						\]
					\end{proof}
					The deformation space of a point blowup of $\PP^3$ is smooth by \cite[Lemma 3.1]{Sano} and a very general deformation of $Y_e$ is the blowup of $\PP^3$ in $n$ points in very general position. By Corollary~\ref{thm:weylpreservescones}, it follows that 
					$W(T_{2,4,n-4})$ acts on the movable cone of the blowup of $\PP^3$ in
					$n$ very general points, and this action coincides with the action defined by Mukai~\cite{Mukai} and Dolgachev--Ortland~\cite{Dolgachev-Ortland}.
					 
					Our construction gives a deformation explanation of this action: it arises from the flop of total spaces induced by root divisors on the special degeneration $Y_e$. 
					\subsection{Cayley octads and admissible divisors}
					We now specialize to $n=8$. Let $N\subset|\cO_{\PP^3}(2)|$ be a net of
					quadrics with zero-dimensional base scheme of length eight. We impose the
					nondegeneracy condition of~\cite[Assumption 1]{PrendergastSmithextremal}:
					there exist $Q,Q'\in N$ such that $Q\cap Q'$ is smooth along the base
					locus of $N$. Equivalently, at every $\PP^3$-basepoint of the net there
					is at most one quadric in $N$ which is singular there.
					
					Blowing up the resulting cluster of eight, possibly infinitely
					near, basepoints resolves the rational map defined by $N$. We call such
					a cluster a Cayley octad and write
					\[
					\pi\colon Y_S\longrightarrow\PP^2
					\]
					for the induced morphism. By~\cite{PrendergastSmithextremal}, $\pi$ is an
					elliptic fibration and is defined by the complete linear system
					$
					\left|-\frac{1}{2}K_{Y_S}\right|.
					$
					Its general fiber is the intersection of two general quadrics in the net,
					and its fiber class is
					\[
					F=\left(-\frac{1}{2}K_{Y_S}\right)^2.
					\]
					In particular, for divisors $D_1,D_2$ on $Y_S$, the product
					$D_1\star D_2$ is the intersection product of their restrictions to a
					general surface in $|-\frac12K_{Y_S}|$.
					
					The special configuration of Proposition~\ref{prop:Ye} with $n=8$ is of
					this form. Indeed, write the elliptic quartic $C$ as
					\[
					C=Q_1\cap Q_2,
					\]
					with $Q_1$ and $Q_2$ smooth.
					Then the corresponding net of quadrics is $\{Q_1, Q_2, 2H\}$, where $H$ is the unique plane with $H \cap C=4P$. 
					Following~\cite{PrendergastSmith}, we make the following definition.
					
					\begin{definition}
						Let $\pi\colon Y_S\to\PP^2$ be the elliptic fibration associated to a
						Cayley octad. An irreducible $\pi$-vertical divisor is \emph{admissible}
						if it is an irreducible component of a reducible divisor
						\[
						\pi^{-1}(B)
						\]
						for a curve $B\subset\PP^2$. We write $\Adm(Y_S)$ for the set of
						admissible divisors.
						
						The classes of admissible vertical divisors span a root system $\tilde{R} \subset K_{Y_S}^\perp$, which we call the \emph{affine root system} associated to the elliptic fibration.
						Its image $R$ in 
						\[
						K_Y^\perp/\ZZ \cdot(-\tfrac{1}{2}K_Y) \cong E_7
						\]
						is the \emph{finite root system} associated to the elliptic fibration. Equivalently, $R$ is spanned by the classes of admissible vertical divisors not meeting a chosen zero section. 
					\end{definition}
					
					Admissible divisors are three-dimensional analogues of the components
					of reducible fibers of an elliptic surface. We recall some of the properties from~\cite{PrendergastSmithextremal}.
					
					\begin{proposition}\label{prop:PS-classification}
						Let $E\in\Adm(Y_S)$ and write $q=\pi|_E$. Then
						\begin{itemize}
							\item $q(E)=L$ is a line, and $q\colon E\to L$ has connected fibers;
							\item $E$ is normal and
							\[
							E\star E=-2;
							\]
						\end{itemize}
					\end{proposition}
					
					\begin{proof}
						The assertion that $L=q(E)$ is a line is from~\cite{PrendergastSmithextremal}. 
						
						We next show that $E \star E=-2$ and $E$ is normal. 
						Let $p \colon Y \rightarrow \PP^3$ be the blowdown. If $E$ is contracted by $p$, then $E$ must be the strict transform of one of the exceptional divisors. Since $E$ is vertical for $\pi$, we must have that 
						\[
						[E]=E_i-\sum_{j>i}n_jE_j
						\]
						 for some $i=1, \dots, 7$ and $n_j \geq 0$. Since $-K_{Y_S}$ is nef, we must in fact have that 
						 \[
						 [E]=E_i-E_j
						 \]
						 for some $j>i$.
						 In particular, $E$ is normal and has $E\star E=-2$ in this case. 
						If $E$ is not contracted, then $p(E)\subset Q$, where $Q$ is one of the quadrics in the net. 
						Since $E$ is a component of a reducible divisor, $Q$ must be singular. If $Q$ has rank $3$, then $p(E)$ is a cone with vertex a base point, so that 
						\[
						[E]=2H-\sum_i n_iE_i
						\]
						where $n_i=2$ for the exceptional divisor corresponding to the base point, $n_i=0$ for a unique $i$, and the remaining $n_i$ are one.
						Again, we see that $E$ is normal and has $E\star E=-2$. 
						If $Q$ has rank $2$ or rank $1$, then $p(E)$ is a plane, so that
						\[
						[E]=H-E_i-E_j-E_k-E_l
						\]
						for $1 \leq i,j,k,l \leq 8$.
						Again, we see that $E$ is normal and has $E\star E=-2$. 

						We now claim that $q$ has connected fibers. Indeed, for a general point $p\in L$, consider the inverse image $D$ of a general line $L' \subset \PP^2$ through $p$. Note that $D \rightarrow L'$ is a rational elliptic surface, and $C=D \cap E$ is contained in the fiber over $p=L \cap L'$, hence a union of $(-2)$-curves. Since $E \star E=-2$, $C$ is connected. It follows that $q$ has connected fibers.
					\end{proof}
					
					For an admissible divisor $E$, define
					\[
					s_E(x)=x+(x\star E)E.
					\]
					The group generated by these reflections is the Weyl group
					$W(\widetilde R)$ considered in \cite{PrendergastSmithextremal}. We find that $W(\tilde{R})$ is nothing other than the big Weyl group of $Y_S$.
					\begin{proposition}\label{prop:bijectionbigroots}
						Let $Y_S$ be the blowup of $\PP^3$ in a Cayley octad. Then
						$
						\Adm(Y_S)=\Delta_{Y_S}^{big}
						$ 
						For every \(E\) in this set, the restriction
						$
						\pi|_E:E\longrightarrow\pi(E)\cong\mathbb P^1
						$
						is the unique \(K_{Y_S}\)-trivial conic fibration on \(E\), with fibre class \(\ell_E\).
						We have
						$s_E=\sigma_E$, and in particular, 
						\[
						W(\widetilde R)=W_{Y_S}^{big}.
						\]
					\end{proposition}
					\begin{proof}
						Let $E\in\Adm(Y_S)$. By Proposition~\ref{prop:PS-classification},
						$
						q=\pi|_E\colon E\longrightarrow L\cong\PP^1
						$
						is a conic fibration. Since
						$
						\omega_{Y_S}^{-1}\cong\pi^*\cO_{\PP^2}(2),
						$
						we have 
						$
						\omega_{Y_S}^{-1}|_E\cong q^*\cO_{\PP^1}(2).
						$
						Thus the obstruction bundle $M_E$ 
						is trivial, and hence $E\in\Delta_{Y_S}^{big}$.
						
						Conversely, let $E\in\Delta_{Y_S}^{big}$, and choose a $K$-trivial contractible
						conic model
						\[
						\alpha\colon Y_S\dashrightarrow Z,
						\qquad
						q\colon E_Z\longrightarrow C.
						\]
						Every flop occurring in $\alpha$ is the flop of a $\pi$-vertical curve, so that $\alpha$ preserves the fibration $\pi$. 
						The induced elliptic fibration
						\[
						\pi_Z\colon Z\longrightarrow\PP^2
						\]
						is defined by $-\frac{1}{2}K_Z$. Since $K_Z \cdot \ell_{E_Z}=0$, it follows that $\pi_Z|_{E_Z}$ factors through $q$. Let 
						\[
						B=\pi_Z(E_Z).
						\]
						 If we had $mE_Z=\pi_Z^{-1}(B)$, then $E_Z$ would be proportional to $-K_Z$ and $E_Z \cdot \ell_{E_Z}=0$, contradiction. It follows that $E_Z$ is an irreducible component of a reducible divisor vertical. Therefore, $E$ is an irreducible component of a reducible vertical divisor, and hence admissible, so $B$ is rational by Proposition~\ref{prop:PS-classification}. By the same Proposition, $\pi_Z|_{E_Z} \colon E_Z \rightarrow B$ has connected fibers, so we must have $C=B$ and $q=\pi_Z|_{E_Z}$. Consequently, the general fiber of $\pi|_E$ represents $\ell_E$.
					Finally, for an admissible divisor $E$, we have
					\[
					x\star E=x\cdot E\cdot(-\tfrac12K_{Y_S})=x\cdot\ell_E,
					\]
					because $(-\tfrac12K_{Y_S})|_E$ is the class of a fiber of $\pi|_E=q$. Therefore
					\[
					s_E(x)=x+(x\cdot\ell_E)E=\sigma_E(x),
					\]
					 which also gives the equality of Weyl groups.
					\end{proof}
					\subsection{The work of Stenger--Xie}
					By \cite{PrendergastSmithextremal}, the root lattice of $Y_e$ is $\tilde{E}_7$, so that we have that 
					\[
					W_{Y_e}^{big}=W(\tilde{E}_7)
					\]
					by Proposition~\ref{prop:bijectionbigroots}.
										We can now recover the fundamental-domain theorem of Stenger--Xie.
					\begin{theorem}[Stenger--Xie]\label{thm:Stenger--Xie--proof}
						Let $Y^{\mathrm{gen}}$ be the blowup of $\PP^3$ in eight very general points. The
						group
						\[
						W(T_{2,4,4})=W(\widetilde E_7)
						\]
						acts on ${\Mov}^+(Y^{\mathrm{gen}})$ with rational polyhedral
						fundamental domain $\overline{\Mov}(Y_e)$. 
					\end{theorem}
					
					\begin{proof}
						The threefold $Y_e$ is an iterated blowup of $\PP^3$ at points, so
						$H^2(\cO_{Y_e})=0$. Deformations of the point blowup $Y_e$ are unobstructed since deformations of $\PP^3$ are unobstructed, see for example \cite[Lemma 3.1]{Sano}. Finally, a very general
						deformation of $Y_e$ is the blowup of $\PP^3$ in eight very general points. By
						Proposition~\ref{prop:Ye}, $-K_{Y_e}$ is nef. Hence
						Theorem~\ref{thm:withnef} and Proposition~\ref{prop:bijectionbigroots} give an action of $W_{Y_e}^{big}=W(\tilde{E}_7)$ on
						$
						\overline{\Mov}(Y^{\mathrm{gen}})\cap\Bbig(Y^{\mathrm{gen}})
						$
						with fundamental domain
						$
						\overline{\Mov}(Y_e)\cap\Bbig(Y_e).
						$
						In particular, we have 
						\begin{equation}\label{eq:mov-empty}
						{\Mov}(Y_e)^\circ \cap w({\Mov}(Y_e)^\circ) =\emptyset, \qquad (w \neq 1).
						\end{equation}
						We next show that $\overline{\Mov}(Y_e)$ is rational polyhedral.
						
						We first compute the effective cone. Set
						$
						\alpha_8=E_8.
						$
						The classes $\alpha_0,\dots,\alpha_8$ form a $\ZZ$-basis of
						$\Pic(Y_e)$ and are represented by effective divisors. We claim that
						\[
						\Eff(Y_e)=\sum_{i=0}^8\RR_{\geq0}\alpha_i.
						\]
						It is enough to consider a prime effective divisor $D$. If $D$ is
						contracted by $Y_e\to\PP^3$, then $D$ is one of the strict transforms of
						the exceptional divisors, so $[D]=\alpha_i$ for some $1\leq i\leq8$.
						Otherwise write
						\[
						D=dH-\sum_{i=1}^8m_iE_i,
						\]
						where $d>0$ and $m_i\geq0$ are the successive multiplicities at the
					 infinitely near points. Writing
						\[
						[D]=\sum_{i=0}^8c_i\alpha_i
						\]
						gives
						\[
\begin{alignedat}{3}
  c_0&=d,                  &\qquad c_1&=d-m_1,                  &\qquad c_2&=2d-m_1-m_2,\\
  c_3&=3d-m_1-m_2-m_3,     &\qquad c_4&=4d-m_1-m_2-m_3-m_4,   &\qquad c_5&=4d-\textstyle\sum_{i=1}^5m_i,\\
  c_6&=4d-\textstyle\sum_{i=1}^6m_i, &\qquad c_7&=4d-\textstyle\sum_{i=1}^7m_i, &\qquad c_8&=4d-\textstyle\sum_{i=1}^8m_i.
\end{alignedat}
\]
						At each blowup the multiplicity at the new infinitely near point is at
						most the degree $d$, so $0\leq m_i\leq d$. Hence
						\[
						c_0,c_1,c_2,c_3,c_4\geq0.
						\]
						On the other hand, the fiber class
						$
						F=\left(-\frac{1}{2}K_{Y_e}\right)^2
						$
						must be nonnegative on every effective divisor, so that
						\[
						c_8=D\cdot F=4d-\sum_{i=1}^8m_i\geq0.
						\]
						Since the $m_i$ are nonnegative, we also have
						\[
						c_5\geq c_6\geq c_7\geq c_8\geq0.
						\]
						This proves the claim. In particular, $\Eff(Y_e)$ is closed and rational
						polyhedral.
						
						By Proposition~\ref{prop:movable-cone-divisorial-halfspaces},
						$\overline{\Mov}(Y_e)$ is obtained from $\Eff(Y_e)$ by imposing the
						half-space inequalities associated to birational divisorial contractions
						\[
						Y_e\overset{\alpha}{\dashrightarrow}Z\longrightarrow\bar Z
						\]
						on SQMs $Z$ of $Y_e$. We show that only finitely many
						half-spaces occur.
						
						Suppose first that $p \colon Z\to\bar Z$ is $K_Z$-negative. Since $-K_Z$ is
						two-divisible, the classification of divisorial contractions from smooth
						threefolds~\cite{MoriThreefolds} rules out contractions to a curve and all
						contractions to a point except the blowup of a smooth point. If $E$ is its exceptional
						divisor, then
						\[
						E\cdot\left(-\frac{1}{2}K_Z\right)^2=1.
						\]
						Thus $E$ is a rational section of the elliptic fibration $\pi_Z \colon Z\to\PP^2$. However, by \cite{PrendergastSmithextremal}, we have that the Mordell--Weil group $MW(Y_e)$ is trivial, so that $E$ must be the strict transform of $E_8$, the last exceptional divisor of the blowup $Y_e \rightarrow \PP^3$. The indeterminacy locus of $\alpha$ is a union of $\pi$-vertical curves, and therefore disjoint from a general line in $E_8$. It follows that the corresponding curve classes satisfy
						\[
						\alpha_*\ell_8=\ell,
						\]
						where $\ell$ is the class of a line in $E$. It follows that the facet $F \subset \ell_E^\perp$ of $\overline{\Mov}(Z)$ corresponding to $p$ is identified under $\alpha^*$ with the facet of $\overline{\Mov}(Y_e)$ corresponding to the contraction of $E_8$. 
						There is therefore a unique facet of $\overline{\Mov}(Y_e)$ corresponding to divisorial $K$-negative contractions on an SQM of $Y_e$.
						
						Suppose now that $p$ is $K$-trivial. Since
						$-\frac12K_Z$ defines the elliptic fibration $\pi_Z$, we have that $\pi_Z$ must factor through $p$. In particular, $p$ cannot be of type II (divisor to point). 
						It follows that $p$ is a type III contraction, we let $E$ be its exceptional divisor. Write $B=\pi_Z(E)$. If $E=\pi_Z^{-1}(B)$ was the full inverse image of $B$, then $p$ would contract the fiber class of $\pi_Z$, contradicting that $p$ is birational. It follows that $E$ must be an irreducible component of $\pi_Z^{-1}(B)$, and therefore $E \in \Adm(Z)$.  
						
						By Proposition~\ref{prop:bijectionbigroots}, the set of admissible divisors is in bijection with the simple roots of the $\tilde{E}_7$ Dynkin-diagram. In particular, there are finitely many facets corresponding to $K$-trivial birational contractions on an SQM of $Y$. 
						
						It follows that $\overline{\Mov}(Y_e)$ is the intersection of a rational
						polyhedral cone with finitely many rational half-spaces, and is therefore
						rational polyhedral.
						
						Set
						\[
						P=\overline{\Mov}(Y_e).
						\]
						The interior of $\overline{\Mov}(Y^{\mathrm{gen}})$ is contained in the big cone,
						so Theorem~\ref{thm:withnef} gives
						\[
						\overline{\Mov}(Y^{\mathrm{gen}})^{\circ}
						\subset W(\tilde{E}_7)\cdot P.
						\]
				By \cite[Theorem 4.1]{LooijengaCones}, it follows that 
					\[
				\overline{\Mov}(Y^{\mathrm{gen}})^+
				= W(\tilde{E}_7)\cdot P.
				\]
				Using Equality~\eqref{eq:mov-empty}, it now follows that $P$ is a rational polyhedral fundamental domain.
					\end{proof}
					The original proof of Stenger--Xie explicitly constructs a covering 
					domain $\Pi$ for the action of $W(T_{2,4,4})$. Moreover, as a byproduct, they prove that 
					\[
					{\Mov}^e(Y^{\mathrm{gen}})={\Mov}^+(Y^{\mathrm{gen}})
					\]
					The following relates the two domains $P$ and $\Pi$. The argument below was obtained in conversation with Isabel Stenger. 
					\begin{lemma}
						We have that
					\[
					\overline{\Mov}(Y_e)=\Pi \cap \overline{\cC}^{big}
					\]
					where $\overline{\cC}^{big}$ is the fundamental chamber defined by $\Delta_Y^{\mathrm{big}}$. 
				\end{lemma}
				\begin{proof}
					By \cite{StengerXie}, we have 
					$\Pi \subset \overline{\Mov}(Y^{\mathrm{gen}})
						$ so intersecting with $\overline{\cC}^{big}$ gives using \eqref{eq:mov}
					\[
					\Pi \cap \overline{\cC}^{big} \subset \overline{\Mov}(Y_e).
					\]
					Conversely, we have $\overline{\Mov}(Y_e) \subset \overline{\cC}^{big}$, every 
					\[
					D= dH-\sum_i m_iE_i \in \overline{\Mov}(Y_e)
					\]
					is in \emph{standard form} (\cite[Definition 3.3]{StengerXie}) and we showed in the proof of Theorem~\ref{thm:Stenger--Xie--proof} that $d \geq m_i \geq 0$. Lemma~\cite[Lemma 6.3]{StengerXie} shows that $D \in \Pi$. 
					It follows that 
					\[
					\overline{\Mov}(Y_e) 	\subset\Pi \cap \overline{\cC}^{big}.
					\]
				\end{proof}
					\subsection{On a question of Prendergast-Smith}
					Let $Y_S$ be the blowup of $\PP^3$ in a Cayley octad associated to a net $N$ of quadrics. We have shown that the affine Weyl group $W(\tilde{R})$ of the elliptic fibration agrees with $W_Y^{big}$. 
					On the other hand, $W(R)$ need not agree with $W_Y^{sm}$. This can easily be seen for $Y_e$ constructed in the previous section, for which we have
					\[
					W_{Y_e}^{big}=W(\tilde{R})=W(\tilde{E}_7), \qquad W_{Y_e}^{sm}=W(A_7), \qquad W(R)=W(E_7).
					\]
					In this section, we show that $Y$ has a relative minimal model $Z$ which, in addition to the elliptic fibration, has the structure of a $\PP^1$-bundle over a weak del Pezzo surface.  We show that $W_Z^{sm}$ coincides with $W(R)$ and also give a geometric explanation for the observation in \cite{PrendergastSmithextremal} that the ADE type of $R$ agrees with the ADE singularity type of the determinant quartic in the dual projective space parametrizing singular quadrics in the net.
					
					Let $\PP^2$ be the projective space which is the base of the elliptic fibration. Points in $\PP^2$ correspond to pencils in $N$, and lines in $\PP^2$ correspond to quadrics in $N$. 
					Points of the dual projective space $(\PP^2)^*$ therefore parametrize quadrics in $N$. 
					The plane quartic 
					\[
					\Delta=\left\{\det(\lambda_1Q_1+\lambda_2Q_2+\lambda_3Q_3)=0\right\}
					\subset(\mathbb P^2)^*
					\]
					parametrizes singular quadrics in $N$ and has at worst ADE singularities by~\cite[Proposition 1.2]{Hattori}.
					Let $\bar{S}$ be the double cover of $(\PP^2)^*$ branched along $\Delta$. 
					$\bar{S}$ is a del Pezzo surface of degree $2$ with du Val singularities (corresponding to the singularities of the plane quartic $\Delta$), and we let $S$ denote its minimal resolution. The morphism $a \colon S \rightarrow (\PP^2)^*$ is given by the complete linear system $|-K_{{S}}|$. Let $\pi' \colon Z \rightarrow \PP^2=\PP(H^0(S, \omega_S^{-1})^*)$ be the universal family of $|-K_S|$, that is, $Z$ is the incidence variety
					\begin{align}
					Z&=\{(s, E) \in {S} \times \PP(H^0(S, \omega_S^{-1})^*) \mid E(s)=0\} \\
					&=\{(s, t) \in {S} \times \PP^2 \mid t \in L_{a(s)}\} \label{eq:incidence}
					\end{align}
					The second projection $\pi' \colon Z \rightarrow \PP^2$ is the universal family, and the first projection $r \colon Z \rightarrow S$ exhibits $Z$ as a $\PP^1$-bundle over $S$. 
					Note that 
					\[
					Z \in |p_1^*\omega_S^{-1}\otimes p_2^*\cO(1)|,
					\]
					 so that $\omega_Z \cong \pi'^*\cO(-2)$.

					\begin{proposition}\label{prop:YZflops}
						The elliptic fibrations $ \pi \colon Y \rightarrow \PP^2$ and $\pi' \colon Z \rightarrow \PP^2$ are related by a crepant SQM $f \colon Y \dashrightarrow Z$ over $\PP^2$. 
					\end{proposition}
					\begin{proof}
						Let $K=k(\mathbb P^2)$ and write $\eta=[x_1:x_2:x_3]$ for
						the generic point of the elliptic base. The corresponding pencil of quadrics
						is parametrized by the line
						\[
						\ell_\eta=\{x_1\lambda_1+x_2\lambda_2+x_3\lambda_3=0\}
						\subset(\mathbb P^2)^*_K.
						\]
						Its base locus is $Y_\eta$, while the incidence description of $Z$
						gives $Z_\eta=a_K^{-1}(\ell_\eta)$. 
						The identification between $Y_\eta$ and $Z_\eta$ is a standard argument, but we provide it here for the convenience of the reader:
						let $U=(\PP^2)^* \setminus \Delta$ and consider the variety of lines 
						\[
						\cF = \{(Q, \ell) \in U \times \Gr(2,4) \mid \ell \subset Q \}.
						\]
						Let $\cF \rightarrow T_U \rightarrow U$ be the Stein factorization of the first projection. 
						Since a smooth quadric has two distinct rulings, $T_U \rightarrow U$ is an etale double cover.
						Let $u \colon T \rightarrow (\PP^2)^*$ be the normalization of $(\PP^2)^*$ in the function field of $T_U$. Over a general point of $\Delta$, the quadric has rank $3$ and contains a unique ruling, so that $u$ ramifies over $\Delta$. It follows that 
						$
						T \cong \bar{S}.
						$
						It follows that 
						\[
						Z_\eta=\ell_\eta \times_{(\PP^2)^*_K}S_K\cong \ell_\eta \times_{(\PP^2)^*_K}T_K,
						\]
						since $\ell_\eta$ avoids the images of the exceptional curves. 
						Choose $P\in Y_\eta(K)$ using a section of $\pi$, for instance the exceptional divisor of the last blowup.
						For a general point $Q\in Y_\eta$, the line $\overline{PQ}$ lies
						in a unique quadric of the pencil parametrized by $\ell_\eta$: the restrictions of the two
						generating quadrics to $\overline{PQ}$ are proportional, since
						both vanish at $P$ and $Q$, and they are not both identically zero.
						The quadric and the ruling containing $\overline{PQ}$ determine
						a point of $Z_\eta$. Conversely, given a general point of $Z_\eta$,
						take the unique line through $P$ in the specified ruling.
						Its intersection with $Y_\eta$ consists of $P$ and a residual
						point $Q$. These constructions are inverse and defined over $K$.
						They therefore give a birational map between the smooth projective
						curves $Y_\eta$ and $Z_\eta$, hence an isomorphism.
						
						This induces a birational map $f\colon Y\dashrightarrow Z$ over
						$\mathbb P^2$. Both varieties are smooth, and
						\[
						\omega_Y\cong\pi^*\mathcal O_{\mathbb P^2}(-2),
						\qquad
						\omega_Z\cong\pi'^*\mathcal O_{\mathbb P^2}(-2).
						\]
						Their canonical divisors are therefore relatively nef.
						By~\cite[Theorem~1]{Kawamata}, $f$ is a composite of flops over
						$\mathbb P^2$, and in particular a crepant SQM.
						\end{proof}
					On the birational model $Z$ of $Y$, there is an explicit description of $W_Z^{sm}$.

\begin{proposition}\label{prop:smallweylgroups}
	Let $R$ be the finite root system associated to the elliptic fibration
	$\pi'\colon Z\to\PP^2$, let $Q(R)$ be its root lattice, and let
	$r\colon Z\to S$ be the $\PP^1$-bundle defined above.
	The assignment $\Gamma\mapsto r^{-1}(\Gamma)$ identifies the irreducible
	$(-2)$-curves on $S$ with $\Delta_Z^{\mathrm{sm}}$.
 In particular, denoting $M$ the intersection lattice of $(-2)$-curves on $S$, we have 
	\[
	W_Z^{\mathrm{sm}}\cong W(R), \qquad M \cong Q(R).
	\]
\end{proposition}

\begin{proof}
	Let $f$ be the class of the fiber of the $\PP^1$-bundle. 
	An irreducible divisor $E$ satisfies $E\cdot f=0$ precisely when
	$E=r^{-1}(\Gamma)$ for an irreducible curve $\Gamma\subset S$.
	By Equation~\eqref{eq:incidence}, $E$ is
	$\pi'$-vertical precisely when $a(\Gamma)$ is a point, or equivalently
	$(-K_S)\cdot \Gamma=0$, where $a \colon S \rightarrow (\PP^2)^*$ is the anticanonical morphism. By the Hodge index theorem such a curve must have $\Gamma^2<0$, and adjunction shows
	that $\Gamma$ is a $(-2)$-curve.
	
	It follows that we have a bijection between the set of $(-2)$-curves on $S$ and $\pi'$-vertical divisors $E$ satisfying $E \cdot f=0$.
	
	Suppose now that $\Gamma$ is a $(-2)$-curve. 
Then
	\[
	E:=r^{-1}(\Gamma)=\Gamma \times L_{a(\Gamma)}\cong\PP^1\times\PP^1.
	\]
	The ruling $q \colon E \rightarrow \PP^1$ given by the second projection satisfies $q=\pi'|_E$. It is contracted by $\pi'$ and therefore $K_Z$-trivial. We have
	\[
	\omega_Z^{-1}|_E\otimes q^*\omega_{\PP^1}
	\cong q^*\bigl(\cO_{\PP^1}(2)\otimes\omega_{\PP^1}\bigr)
	\cong\cO_E,
	\]
	so $[E]\in\Delta_Z^{\mathrm{sm}}$.
	
	Conversely, let $E \in \Delta_Z^{sm}$, and write $q \colon E \rightarrow C$ for the defining $K$-trivial ruling. Since flops identify admissible divisors, Proposition~\ref{prop:bijectionbigroots}, shows that $q=\pi'|_E$ and $C=\PP^1$.
	Quasi-ruledness therefore makes $E$ a $\PP^1$-bundle.
	Thus $\rho(E)=2$. If $E\cdot f>0$, then $r|_E\colon E\to S$ is generically finite and surjective, giving
	$\rho(E)\geq\rho(S)=8$, a contradiction, so we must have $E \cdot f=0$.
	$E$ is also $\pi'$-vertical
	because $\pi'$ contracts its $K_Z$-trivial ruling fibers. It follows that $\Gamma=r(E)$ is a $(-2)$-curve. 
	This proves the asserted correspondence.
	
	For the identifications of Weyl groups, recall from Proposition~\ref{prop:PS-classification} that every admissible divisor lies over a line. Since $f$ maps to a line in $\PP^2$, we have $-\tfrac{1}{2}K_Z\cdot f=1$.
	Writing
	\[
	\pi'^*L=\sum_{i=0}^k m_iE_i,
	\qquad
	1=\sum_{i=0}^k m_i(E_i\cdot f),
	\]
	we find, after reindexing, that
	$m_0=E_0\cdot f=1$ and $E_i\cdot f=0$ for $i>0$.
	Thus
	\[
	[E_0]\equiv-\sum_{i>0}m_i[E_i]\pmod{\ZZ \cdot (-\tfrac{1}{2}K_Z)}.
	\]
	The images of the small roots therefore generate the root
	lattice of $R$ in $K_Z^\perp/\ZZ \cdot (-\tfrac{1}{2}K_Z)$. Since 
	\[
	E_1 \star E_2=E_1 \cdot E_2 \cdot -\tfrac{1}{2}K_Z=E_1 \cdot \ell_{E_2}
	\]
	for $E_1, E_2 \in \Delta_Z^{sm}$, we obtain that $W(R)=W_Z^{sm}$. 
	
	Finally, for $\alpha,\beta\in N^1(S)$, we compute 
	\[
	r^*\alpha\star r^*\beta
	=r^*\alpha\cdot r^*\beta\cdot (-\tfrac{1}{2}K_Z)
	=(\alpha\cdot\beta)_S (-\tfrac{1}{2}K_Z \cdot f)=(\alpha \cdot \beta)_S,
	\]
	so that the intersection lattice $M$ of $(-2)$-curves is identified with $Q(R)$. 
\end{proof}

\begin{corollary}\label{cor:PS}
	The ADE singularity type of the determinant quartic $\Delta$
	agrees with the ADE type of the finite root system $R$
	associated to $\pi \colon Y\to\PP^2$.
\end{corollary}

\begin{proof}
	The double cover $\bar S\to(\PP^2)^*$ branched along $\Delta$
	has du Val singularities of the same ADE types as $\Delta$.
	Their exceptional curves on the minimal resolution $S$ are
	precisely the $(-2)$-curves on $S$: both are exactly the curves
	contracted by the anticanonical morphism
	$S\to\bar S\to(\PP^2)^*$.
	The intersection lattice $M$ of the $(-2)$-curves on $S$ therefore has the ADE type of
	$\Delta$. 
We conclude by Proposition~\ref{prop:smallweylgroups}.
\end{proof}
				\appendix
\section{Background on Hodge modules}\label{app:hodge}

We recall the facts about Saito's theory of Hodge modules that are needed in the proof of Theorem~\ref{thm:surjective}. We will follow the exposition of \cite{Schnell}.
Let \(X\) be a smooth complex algebraic variety. Roughly speaking, a mixed Hodge module \(M\) on \(X\) consists of the data
\[
M=(\mathcal M,F_\bullet\mathcal M,K,W_\bullet),
\]
where \(\mathcal M\) is a regular holonomic \(\mathscr D_X\)-module, \(F_\bullet\mathcal M\) is a filtration, called the Hodge filtration, \(K\) is a perverse sheaf of \(\QQ \)-vector spaces, and \(W_\bullet\) is a finite increasing weight filtration. The \(\mathscr D\)-module and perverse-sheaf are related by the Riemann--Hilbert correspondence:
\[
\operatorname{DR}(\mathcal M)\simeq K\otimes_{\QQ }\mathbf C .
\]
where 
\[
\operatorname{DR}(\cM)
=
\left[
\cM
\longrightarrow
\Omega_X^1 \otimes_{\cO_X} \cM
\longrightarrow
\Omega_X^2 \otimes_{\cO_X} \cM
\longrightarrow
\cdots
\longrightarrow
\Omega_X^n \otimes_{\cO_X} \cM
\right][n].
\] 
is the de Rham complex of $\cM$.
These data are required to satisfy additional compatibility and admissibility conditions which we do not discuss, see \cite{Schnell, SaitoMHM}. On a singular variety $X$, the definition is made by embedding $X$ into a smooth variety and considering Hodge modules supported on $X$.

Mixed Hodge modules form an abelian category
$
\operatorname{MHM}(X),
$
and there is an exact faithful realization, or rationalization, functor
\[
\operatorname{rat}\colon
\operatorname{MHM}(X)\longrightarrow
\operatorname{Perv}(X,\QQ ),
\]
sending a mixed Hodge module to its underlying perverse sheaf. 

Saito proved that the usual six operations on constructible complexes lift to the derived categories \(D^b\operatorname{MHM}(X)\). Moreover, the Verdier duality functor $\DD$ lifts to a duality functor on  \(D^b\operatorname{MHM}(X)\), also denoted $\DD$.

The Tate twist $M(m)$ has $F_k(M(m))=F_{k-m}M$ and changes
the weight by $-2m$. 

Polarizable variations of Hodge structures give basic examples of Hodge modules. Let \(X\) be a smooth complex variety of dimension \(d\), and let \((V,F^\bullet\mathcal V)\) be a polarizable variation of rational Hodge structures of weight \(w\) on \(X\), where \(\mathcal V=V\otimes_{\mathbb Q}\mathcal O_X\) is the associated flat vector bundle. Its Hodge filtration \(F^\bullet\mathcal V\) consists of holomorphic subbundles and satisfies Griffiths transversality. By the holomorphic Poincaré lemma, \(\operatorname{DR}(\mathcal V)\) is quasi-isomorphic to \((V\otimes_{\mathbb Q}\mathbb C)[d]\). Setting \(F_k\mathcal V=F^{-k}\mathcal V\), the data \((\mathcal V,F_\bullet\mathcal V,V[d])\) define a polarizable pure Hodge module of weight \(w+d\).

A mixed Hodge module on a point is the same thing as a mixed Hodge structure. We denote $\QQ^{H}$ the Hodge module corresponding to the weight zero Hodge structure $\QQ(0)$. 

For any complex algebraic variety $X$, we then define the constant Hodge complex 
\[
\QQ_X^H=a_X^*\QQ^H \in D^b\operatorname{MHM}(X),
\]
where $a_X \colon X \rightarrow pt$.
Its underlying rational complex is $\QQ_X$.
If \(X\) is smooth of pure dimension \(d\), then
\[
\QQ_X^H[d]=(\cO_X,F_\bullet\cO_X,\QQ_X[d])
\]
is a pure Hodge module of weight $d$, with $F_k\cO_X=0$ for $k<0$ and $F_k\cO_X=\cO_X$ for $k\ge0$, and the compatibility 
\[
\DR(\cO_X) \simeq \CC_X[d]
\]
is the well-known fact that the de Rham complex is quasi-isomorphic to the constant sheaf $\CC_X$. 

For smooth $X$ of dimension $d$,
$\DD\QQ_X^H=\QQ_X^H(d)[2d]$; see \cite[Sections 7 and 29]{Schnell}.

For general $X$, $\QQ_X^H[d]$ is a complex of Hodge modules, and the $D$-module is not easy to describe. However, Saito proved~\cite[Theorem 0.2]{Saito2000} that there is a functorial quasi-isomorphism
\begin{equation}\label{eq:appendix-dubois}
(\operatorname{DR}(\QQ _X^H[d]), F)
\simeq
(\underline{\Omega}_X^{\bullet}, F)[d]
\end{equation}
of filtered complexes, where $\underline{\Omega}_X^{\bullet}$ is the du Bois complex of $X$.
Define the functor
\[
G_a=\Gr^F_{-a}DR(-).
\]
then we have
\[
G_a(\QQ_X^H(r)[n]) \cong \underline{\Omega}_X^{a+r}[n-r-a]
\]
If $X$ is smooth, then there is a filtered, functorial quasi-isomorphism
$
\underline{\Omega}_X^{\bullet} \cong {\Omega}_X^{\bullet}
$
so that 
\[
G_a(\QQ_X^H(r)[n]) \cong {\Omega}_X^{a+r}[n-r-a]
\]
For every variety $X$, there is a natural morphism
$
\mathcal O_X\longrightarrow\underline{\Omega}_X^0
$
and \(X\) is said to have Du Bois singularities if it is a quasi-isomorphism. Hence, if \(X\) is Du Bois,
\begin{equation}\label{eq:appendix-dubois-zero}
G_0(\QQ_X^H[d])
\simeq\mathcal O_X[d].
\end{equation}

The functor $G_a$ behaves well with respect to proper direct image and duality: if 
$f \colon X \rightarrow Y$ is a proper morphism between
varieties and $M$ is a Hodge module, by \cite[2.3.7 and Section 2.4]{Saito88}, see also \cite[Equation 9.2]{Popa} there are natural quasi-isomorphisms of functors
\begin{equation}\label{eq:compatibilitydirectimageduality}
	G_a\circ f_*\simeq Rf_*\circ G_a,\qquad G_a\circ\DD\simeq D_X\circ G_{-a}
\end{equation}
where $Rf_*$ is the right derived functor of the usual $\cO$-module push-forward, and $D_X=R\mathcal{H}om_{\cO_X}(_,\omega_X^\bullet)$ is the Grothendieck duality functor. 

Let \(Y\) be a pure \(d\)-dimensional complex algebraic variety. Saito constructs~\cite[4.5.12]{SaitoMHM} a fundamental-class morphism
\begin{equation}
\operatorname{cl}_Y^H:
\QQ_Y^H
\longrightarrow
\mathbb D\QQ_Y^H(-d)[-2d]
\label{eq:appendix-fundamental-class}
\end{equation}
in \(D^b\operatorname{MHM}(Y)\), and shows that it is uniquely determined by its restriction to the smooth locus of $Y$. 
The underlying morphism of rational complexes is the usual fundamental-class morphism
\[
\QQ _Y\longrightarrow
\mathbb D\QQ _Y[-2d].
\]
which, under the isomorphism $\Hom(\QQ_Y,\DD\QQ_Y[-2d]) \cong H_{2d}^{BM}(Y, \QQ)$ corresponds to the fundamental class $[Y]$. 
Here $BM$ denotes Borel--Moore homology, which agrees with singular homology if $Y$ is compact.
If $Y$ is smooth, then $cl_Y^{H}$ is an isomorphism.
After shifting by $[a]$, the induced morphism on $G_a$ is given by 
\begin{equation}\label{eq:appendix-fundamental-graded}
	\underline{\Omega}_Y^a\longrightarrow D_Y(\underline{\Omega}_Y^{d-a})[-d].
\end{equation}
If $Y$ is smooth, this is the isomorphism given by the perfect pairing 
\[
\Omega_Y^a \times \Omega_Y^{d-a} \rightarrow \omega_Y
\]
Finally, we discuss the subcategory of pure Hodge modules. 
A pure Hodge module is a mixed Hodge module with trivial weight filtration.
Let $X$ be smooth projective, and let $Z \subset X$ be an irreducible subvariety. By \cite[Theorem 15.1]{Schnell}, every polarizable variation of Hodge structures of weight $w -\dim Z$ on the smooth locus of $Z$ extends uniquely to an object in $HM(X, w)$ with strict support $Z$. 
Conversely, every object $M \in HM(X, w)$ is of the form 
\[
M \cong  \bigoplus_{Z \subset X} M_Z
\]
where each $M_Z$ is the unique extension with strict support $Z$ of a polarizable variation of Hodge structure of weight $w-\dim Z$ on a Zariski-open subset of the smooth locus of $Z$.
In particular, the category $HM(X, w)$ is semisimple.

We will need a special case of the Hodge-module decomposition theorem. Recall that a morphism $f \colon X \rightarrow Y$ is semismall if 
\[
\dim\{y:\dim f^{-1}(y)\ge k\}\le\dim X-2k
\]
for all $k\ge0$. In particular, a type III contraction $Y \rightarrow \bar{Y}$ as in Proposition~\ref{prop:wilson-type-contraction-to-curve} is semismall.
\begin{theorem}\label{thm:semismall-hodge}
	Let $f \colon X \rightarrow Y$ be a semismall projective morphism of complex varieties with $X$ smooth of pure dimension $d$. Then
	$Rf_*\QQ_X^H[d]$ is a pure (and hence semisimple) Hodge module of weight $d$. 
\end{theorem}
\begin{proof}
	By the decomposition theorem for Hodge modules~\cite[Theorem 4.2.6]{CataldoMiglioriniDecomposition}, we have that 
	\[
	Rf_*\QQ^H_X[d]\cong\bigoplus_{j\in\ZZ}\cH^j(Rf_*\QQ_X^H[d])[-j]
	\]
	with each $\cH^j(Rf_*\QQ_X^H[d])$ pure of weight $d+j$ and semisimple. 
	The rationalization functor commutes with cohomology and direct image, so the underlying perverse sheaf of $\cH^j(Rf_*\QQ_X^H[d])$ is ${}^p\cH^j(Rf_*\QQ_X[d])$. Since $f$ is semismall, $Rf_*\QQ_X[d]$ is a perverse sheaf (\cite[Theorem 2.1.1, Remark 4.7.3]{CataldoMigliorini}), and hence
	\[
	{}^p\cH^j(Rf_*\QQ_X[d])=0
	\]
	for $j\ne0$. We conclude that 
	$Rf_*\QQ^H_X[d]\cong\cH^0(Rf_*\QQ_X^H[d])$ is pure and semisimple.
\end{proof}
Write  
\[
Rf_*\QQ^H_X[d] \cong  \bigoplus_{Z \subset Y} M_Z
\]
as in the strict-support decomposition above. 
Passing to perverse sheaves, we obtain 
\[
Rf_*\QQ_X[d] \cong  \bigoplus_{Z \subset Y} IC_Z(L)
\]
where $Z$ ranges over the relevant strata of $f$, $L$ is the local system of top-dimensional fiber cohomology on the corresponding stratum (\cite[Theorem 3.4.1]{CataldoMigliorini2}), and $IC_Z(L)$ is its intersection complex. For example, for a $K$-trivial type III contraction $p \colon Y \rightarrow \bar{Y}$ contracting a surface $E$ to a smooth curve $C$, and such that $q \colon E \rightarrow C$ has general fiber $\PP^1$, the closures of the relevant strata are $\bar{Y}$ and $C$, and the local systems are trivial. Passing to hypercohomology, we obtain isomorphisms
\[
H^k(Y, \QQ) \cong IH^k({\bar{Y}, \QQ}) \oplus H^{k-2}(C, \QQ)(-1)
\]
If $E \rightarrow C$ is quasi-ruled, then $\bar{Y}$ is a $\QQ$-homology manifold so that 
\[
IH^k({\bar{Y}, \QQ}) \cong H^k(\bar{Y}, \QQ).
\]

					\color{black}
					\bibliography{ReflectionGroupsThreefolds}
					\bibliographystyle{alpha}
					
				\end{document}